\documentclass[11pt]{article}
\usepackage[a4paper,margin=27mm]{geometry}
\usepackage{amsmath,amssymb,amsthm,mathtools,booktabs,array,longtable}
\usepackage{xcolor,graphicx}
\usepackage[T1]{fontenc}\usepackage{lmodern,microtype}
\usepackage[colorlinks=true,linkcolor=blue,citecolor=blue,urlcolor=blue]{hyperref}
\hypersetup{pdftitle={Global small-data stability near the KZ spectrum},pdfauthor={Research manuscript},pdfsubject={Detailed proof with stated EMV and CDG inputs}}
\newtheorem{theorem}{Theorem}[section]
\newtheorem{lemma}[theorem]{Lemma}
\newtheorem{proposition}[theorem]{Proposition}
\newtheorem{definition}[theorem]{Definition}
\newtheorem{remark}[theorem]{Remark}

\newcommand{\Z}{\mathcal Z}
\newcommand{\C}{\mathcal C}
\newcommand{\F}{\mathcal F}
\newcommand{\norm}[1]{\left\|#1\right\|}
\newcommand{\dd}{\,\mathrm d}
\newcommand{\address}[1]{\AtEndDocument{\par\bigskip\noindent\textit{#1}\par}}
\newcommand{\email}[1]{\AtEndDocument{\noindent\textit{Email:} \texttt{#1}\par}}
\newcommand{\esssup}{\mathop{\mathrm{ess\,sup}}}
\newcommand{\Id}{\mathrm I}

\numberwithin{equation}{section}
\newsavebox{\equationbox}
\newcommand{\fitdisplay}[1]{\sbox{\equationbox}{$\displaystyle #1$}%
\ifdim\wd\equationbox>.96\linewidth
\resizebox{.96\linewidth}{!}{\usebox{\equationbox}}%
\else\usebox{\equationbox}\fi}
\title{Global Stability of the Kolmogorov--Zakharov Spectrum
		for the Wave Kinetic Equation\\[3pt]\large }
\author{Jiajun Wang}
\address{Courant Institute of Mathematical Sciences,
	New York University, New York, NY 10012, USA}
\email{jw9409@nyu.edu}
\date{}
\begin{document}\maketitle
\begin{abstract}
	We prove global nonlinear stability near the mass
	Kolmogorov--Zakharov spectrum
	$f_M(\omega)=\omega^{-7/6}$ for the isotropic cubic
	wave kinetic equation, including small time-dependent
	external forces.
	For sufficiently small initial perturbations and
	forcing, the solution remains uniformly close to
	$f_M$ for all time in a weighted norm.
	This norm allows
	$f(t,\omega)\sim a_0(t)\omega^{-7/6}$ as $\omega\to0$
	and $f(t,\omega)\sim a_\infty(t)\omega^{-7/6}$
	as $\omega\to\infty$.
	
	When the external force $\phi$ is independent of time, the solution
	converges at an algebraic rate to a stationary solution
	in weaker weighted norms.
	We also give a necessary and sufficient condition
	on the initial perturbation
	for convergence in the original weighted norm.
	These results resolve, for small perturbations of the
	Kolmogorov--Zakharov spectrum, the question of
	convergence to a stationary solution under
	time-independent forcing raised by
	Collot, Dietert, and Germain
	\cite[Section~1.3]{CDG}.
	
	To establish this stability, the collision integral
	is expanded around $a_0(t)f_M$ at low frequencies
	and $a_\infty(t)f_M$ at high frequencies, separating
	it into a linear term with time-dependent coefficients
	and a nonlinear remainder.
	This decomposition yields improved control of the
	nonlinear remainder.
	
	The main difficulty is that $a_0(t)$ may vary in time,
	so the linear operator itself changes with time and
	its evolution cannot be represented by the semigroup
	of a single fixed operator.
	To address this difficulty, we use a parametrix
	construction to explicitly construct the evolution
	operator $U_{a_0,a_\infty}$ and establish the required
	uniform bounds.
	Moreover, our arguments provide a general framework
	for proving global stability that may also apply
	to other evolution equations.
\end{abstract}
\tableofcontents
\section{Equation, spaces, and main results}\label{sec:setting}
\subsection{Introduction}

Wave turbulence theory describes the statistical evolution of
large systems of weakly interacting dispersive waves.
Its applications include surface water waves, plasma waves,
nonlinear optics, and Bose gases.
Rather than following each wave individually, the theory studies
the distribution of wave intensity among Fourier modes.
Under suitable assumptions on weak nonlinearity and the statistics
of the waves, resonant interactions lead to a wave kinetic equation
for this distribution.
A central prediction is the transfer of conserved quantities through stationary power-law spectra.
The monographs \cite{ZLF1992,Nazarenko2011} provide systematic
accounts of these ideas and their physical applications.

The rigorous derivation of wave kinetic equations from nonlinear
dispersive equations has undergone substantial development.
For the cubic nonlinear Schr\"odinger equation with random initial
data on large tori, Buckmaster, Germain, Hani, and Shatah
\cite{BGHS2021} established the onset of the kinetic description
on a nontrivial time scale.
Deng and Hani \cite{DH2023} subsequently derived the wave kinetic
equation up to a multiple of the kinetic time scale
in a large box with weak-nonlinearity.
Their later work \cite{DHLongTime} extends this description to
arbitrary fixed times within the lifespan of the limiting kinetic
solution.
These results establish a rigorous connection between the
underlying dispersive dynamics and the kinetic equation.
They also motivate the study of the kinetic equation itself,
particularly its long-time dynamics and the stability of its
stationary spectra.

In this paper we consider the following isotropic
four-wave kinetic equation associated with the
cubic nonlinear Schr\"odinger equation.
\begin{equation}\label{eq:PDE}
	\partial_t f=\C(f),\qquad f(0,\omega)=f_0(\omega),
\end{equation}
where
\begin{equation}\label{eq:collision}
	\C(f)(\omega_1)=
	\iint_{\omega_3,\omega_4>0,\ \omega_2>0}
	W\big[(f_1+f_2)f_3f_4-(f_3+f_4)f_1f_2\big]
	\dd\omega_3\dd\omega_4.
\end{equation}
Here $f=f(t,\omega)$ is the isotropic wave spectrum,
$t\ge0$ is time, $\omega=|k|^2>0$ is the squared modulus
of the wave vector $k\in\mathbb R^3$.
And we write $f_i=f(t,\omega_i)$ and use the notations
\begin{equation*}
	\omega_2=\omega_3+\omega_4-\omega_1>0,
	\qquad
	W=\frac{\min(\sqrt{\omega_1},\sqrt{\omega_2},
		\sqrt{\omega_3},\sqrt{\omega_4})}{\sqrt{\omega_1}}.
\end{equation*}
We also study the corresponding equation with an external force
$\phi=\phi(t,\omega)$,
\begin{equation*}
	\partial_t f=\C(f)+\phi.
\end{equation*}

For sufficiently integrable solutions, the equation (\ref{eq:PDE})
formally conserves the following mass and energy:
\begin{equation*}
	\int_0^\infty \omega^{1/2}f(\omega)\,\dd\omega
	\qquad\text{and}\qquad
	\int_0^\infty \omega^{3/2}f(\omega)\,\dd\omega,
\end{equation*}
respectively.
The transfer of these quantities toward low or high frequencies
is an important feature of the dynamics.
Escobedo and Vel\'azquez \cite{EV2015} developed a theory of weak
and mild solutions for the isotropic equation and established
results on existence, long-time behavior, blow-up, and
finite-time condensation.
For the four-wave kinetic equation without an isotropy assumption
on the solution, Germain, Ionescu, and Tran \cite{GIT2020}
proved local well-posedness in nearly critical weighted spaces,
allowing a class of radial dispersion relations.
These works provide approaches to the Cauchy problem
and to the transfer of mass and energy across frequencies.

Our interest is in the Kolmogorov--Zakharov spectrum
\begin{equation*}
	f_M(\omega)=\omega^{-7/6}.
\end{equation*}
It satisfies $\C(f_M)=0$, but the stationarity of $f_M$ does not mean that mass is not transferred between frequencies. In fact, it describes an inverse mass cascade,
with mass transferred from high to low frequencies
\cite{ZLF1992,CDG}.

Let $J_M(f)(\omega)$ denote the net mass transferred
per unit time across the frequency $\omega$, with
positive sign for transfer from lower to higher frequencies.
The collision-integral definition of this flux is given
in \cite{CDG}.
For the unforced equation (\ref{eq:PDE}), the corresponding mass 
on a frequency interval $[a,b]\subset(0,\infty)$ satisfies
\begin{equation*}
	\frac{d}{dt}\int_a^b \omega^{1/2}f(t,\omega)\,\dd\omega
	=
	J_M(f(t))(a)-J_M(f(t))(b).
\end{equation*}
Thus the mass in $[a,b]$ remains constant whenever the
fluxes at $a$ and $b$ agree.
For $f_M$, the flux is independent of frequency and
strictly negative.
Mass therefore enters $[a,b]$ through $b$ and leaves
through $a$ at the same rate.
This explains how $f_M$ can be stationary while
supporting a nonzero mass flux toward zero frequency.

Although $f_M(\omega)=\omega^{-7/6}$ is unbounded
as $\omega\downarrow0$, its mass and energy densities
are integrable near zero.
Its total mass and energy diverge because of its
high-frequency tail:
\begin{equation*}
	\int_1^\infty \omega^{1/2}f_M(\omega)\,\dd\omega
	=
	\int_1^\infty \omega^{-2/3}\,\dd\omega
	=\infty,
\end{equation*}
\begin{equation*}
	\int_1^\infty \omega^{3/2}f_M(\omega)\,\dd\omega
	=
	\int_1^\infty \omega^{1/3}\,\dd\omega
	=\infty.
\end{equation*}
Requiring the solution itself to have finite mass
and energy would therefore exclude $f_M$.
Moreover, even a small change in its amplitude,
$f=(1+\varepsilon)f_M$ with $\varepsilon\ne0$,
produces a perturbation whose absolute mass and energy
integrals are infinite.
We instead measure $f-f_M$ in the weighted space
$\mathcal Z$ defined below.

Several earlier results are particularly close to this problem.
Escobedo, Mischler, and Vel\'azquez \cite{EMVL}
constructed the fundamental solution of the linear equation
obtained by linearization around $\omega^{-7/6}$
and derived detailed asymptotic estimates.
Their subsequent work \cite{EMVN}
established local existence and uniqueness for solutions
of the Uehling--Uhlenbeck equation that behave like
$a(t)\omega^{-7/6}$ as $\omega\to0$,
under additional regularity assumptions near zero and
exponential decay at infinity.
The fundamental-solution estimates in \cite{EMVL,EMVN}
provide the starting point for our linear analysis.

For the same Kolmogorov--Zakharov spectrum,
Collot, Dietert, and Germain \cite{CDG}
proved stability under small time-independent forcing
in the stationary setting.
They constructed nonnegative solutions of
$-\mathcal C(f)=\phi$ close to $\omega^{-7/6}$
and identified their leading coefficients at zero and infinity,
together with the associated mass and energy fluxes.
Their weighted trilinear estimates are also used below
to control the nonlinear remainder.
In contrast, we study the evolution from prescribed initial data
and seek bounds uniform for all positive times.

Recent work of Costa, Krstulovic, and Nazarenko
\cite{CKN2026Acoustic} investigates the stability of
Rayleigh--Jeans and Kolmogorov--Zakharov spectra for
acoustic wave kinetic equations through theoretical analysis
and numerical simulations.
For the cubic NLS wave kinetic equation, their subsequent
preprint \cite{CKN2026NLS} studies the spectrum
$f_M(\omega)=\omega^{-7/6}$.
They analyze the linearized equation using the Mellin transform
and a Carleman equation, and compare the predicted evolution
of perturbations with simulations of the nonlinear equation.
Their analysis does not establish global existence, uniqueness,
and a uniform weighted-norm bound for the nonlinear
Cauchy problem.

Global nonlinear stability has also been established
near Rayleigh--Jeans equilibria.
Menegaki \cite{Menegaki2024} proved an $L^2$ stability result
with a frequency cutoff.
More recently, Escobedo and Menegaki \cite{EM2026RJ}
treated the full frequency domain and proved global
well-posedness and exponential relaxation near
$(\omega+\mu)^{-1}$, with $\mu>0$.
Their theorem assumes small relative perturbations in
$L^2((0,\infty);\sqrt{\omega}\,d\omega)$,
orthogonal to the kernel of the linearized operator.
These results similarly use linear estimates to establish
nonlinear stability.
However, Rayleigh--Jeans equilibria have different behavior
at both zero and infinity from $\omega^{-7/6}$,
so these theorems concern different backgrounds and
perturbation spaces.

We prove global nonlinear stability near $f_M$
on the full frequency domain.
For sufficiently small initial perturbations, the solution
of \eqref{eq:PDE} remains uniformly close to $f_M$ for all
$t\ge0$ in the weighted space $\mathcal Z$.
This space allows small changes in the asymptotic behavior at zero and infinity:
\begin{equation*}
	f(t,\omega)\sim a_0(t)\omega^{-7/6}
	\quad\text{as }\omega\to0,
	\qquad
	f(t,\omega)\sim a_\infty(t)\omega^{-7/6}
	\quad\text{as }\omega\to\infty,
\end{equation*}
where $a_0(t)$ and $a_\infty(t)$ remain close to one.
The same stability conclusion holds under a sufficiently
small time-dependent external force, measured in the
space-time norm introduced in \eqref{eq:force-norm}.
No convergence of the force as $t\to\infty$ is required.

To establish this stability, the collision integral is
expanded around $a_0(t)f_M$ at low frequencies and
$a_\infty(t)f_M$ at high frequencies.
This separates the equation into a linear term with
time-dependent coefficients and a nonlinear remainder.
The main difficulty is that $a_0(t)$ and $a_\infty(t)$
depend on the evolving solution.
The resulting linear operator therefore changes with time,
and its evolution is not given by the semigroup of a fixed
operator.

To handle this difficulty, we explicitly construct the
evolution operator $U_{a_0,a_\infty}$ by a parametrix method.
Starting from the fixed linear evolution, we use two
different time changes for the low- and high-frequency
parts and combine them with a frequency cutoff.
The resulting error is corrected by a convergent series
of time integrals.
This construction yields uniform operator bounds and
Lipschitz estimates with respect to $a_0$ and $a_\infty$.

This decomposition also improves control of the nonlinear
remainder.
After the linear contribution around $a_0(t)f_M$ is
removed, every remaining collision term contains at least
two factors of $f-a_0(t)f_M$, which are less singular
at zero.
Similarly, subtracting the linear contribution around
$a_\infty(t)f_M$ leaves at least two factors of
$f-a_\infty(t)f_M$, which decay faster at infinity.
Together with the estimates for $U_{a_0,a_\infty}$, this
closes the global stability argument.
The proof provides a general framework that may also
apply to other evolution equations for which analogous
linear estimates and nonlinear decompositions are available.

We now introduce the symmetric trilinear form used to
expand \eqref{eq:collision}.
First define
\begin{equation}\label{eq:barC}
	\overline\C(f,g,h)=
	2\iint_{0<\omega_3<\omega_4,\ \omega_2>0}
	W\big[(f_1+f_2)g_3h_4-(g_3+g_4)h_1f_2\big]
	\dd\omega_3\dd\omega_4.
\end{equation}
We define $\C(f,g,h)$ as the average of
$\overline\C$ evaluated at the six permutations of $(f,g,h)$.
In particular $\C(f,f,f)=\C(f)$.

\subsection{Weights and the two coefficients}
Fix 
\begin{equation}\label{eq:parameters}
 0<\delta<\frac1{12}.
\end{equation}
For real numbers $a,b$, $E_{a,b}$ is the space of measurable functions with norm
\begin{equation}\label{eq:E}
 \norm{g}_{E_{a,b}}=
 \max\left(\esssup_{0<\omega\le1}\omega^a|g(\omega)|,
           \esssup_{\omega\ge1}\omega^b|g(\omega)|\right)<\infty.
\end{equation}
Set
\begin{equation*}
 X=E_{\frac76-\delta,\frac76+\delta},\qquad Y=E_{\frac32-2\delta,\frac32+2\delta}.
\end{equation*}
Choose a fixed smooth function $\chi:(0,\infty)\to[0,1]$ such that
$\chi=0$ on $(0,1]$ and $\chi=1$ on $[2,\infty)$.
Define
\begin{equation}\label{eq:modes}
 f_M(\omega)=\omega^{-\frac76},\qquad \Z=\mathbb R (1-\chi)f_M\oplus\mathbb R \chi f_M\oplus X.
\end{equation}
Every $z\in\Z$ admits the following unique decomposition,
in terms of which we define its norm:
\begin{equation}\label{eq:Znorm}
	z=z_0(1-\chi)f_M+z_\infty\chi f_M+h,
	\qquad
	\norm{z}_\Z:=|z_0|+|z_\infty|+\norm{h}_X,
\end{equation}
where $z_0,z_\infty\in\mathbb R$ and $h\in X$.

Indeed $\omega^{7/6} h(\omega)$ tends to zero at both $0$ and $\infty$. Thus $z_0,z_\infty$ are the
corresponding limits of $\omega^{7/6}z$. This proves uniqueness,
and also that the two coefficient maps are bounded linear functionals.
The direct sum is complete. An important pointwise consequence is
\begin{equation}\label{eq:envelope}
 |z(\omega)|\le\norm{z}_\Z\omega^{-\frac76} \quad\hbox{for almost every }\omega>0.
\end{equation}
For $f\in\Z$, write $a_0(f),a_\infty(f)$ for its two coefficients:
\begin{equation}\label{eq:endpointfunctionals}
 a_0(f)=\mathop{\mathrm{ess\,lim}}_{\omega\downarrow0}\omega^{7/6}f(\omega),
 \qquad
 a_\infty(f)=\mathop{\mathrm{ess\,lim}}_{\omega\to\infty}\omega^{7/6}f(\omega).
\end{equation}
Thus $a_0(z)=z_0$, $a_\infty(z)=z_\infty$, and
$|a_0(z)|+|a_\infty(z)|\le\norm z_\Z$.
We will repeatedly use the elementary inclusion
\begin{equation*}
 a\le c,\quad b\ge d\quad\Longrightarrow\quad
 E_{a,b}\subset E_{c,d},\qquad
 \norm{g}_{E_{c,d}}\le\norm{g}_{E_{a,b}}.
\end{equation*}

\subsection{Solution class and the main results}
\begin{definition}[Local initial trace]
	Let $0\le s<T<\infty$ and
	$g\in L^\infty_{\mathrm{loc}}(0,\infty)$.
	A function $u:(s,T]\to L^\infty_{\mathrm{loc}}(0,\infty)$
	has local initial trace $g$ at time $s$ if
	\begin{equation*}
		\lim_{t\downarrow s}
		\|u(t)-g\|_{L^\infty(I)}=0
		\qquad
		\text{for every compact interval }I\subset(0,\infty).
	\end{equation*}
	In particular, $u(t)$ converges to $g$ almost everywhere.
\end{definition}

For a function $v$ of the frequency variable $\omega$,
$v|_I$ denotes its restriction to the interval $I$.

\begin{definition}[Integral solution]\label{definition of integral solution}
	Let $0\le s<T<\infty$ and
	$g\in L^\infty_{\mathrm{loc}}(0,\infty)$.
	Consider the equation $\partial_tu=\mathcal A(t,u)$.
	A function $u:[s,T]\to L^\infty_{\mathrm{loc}}(0,\infty)$
	is an integral solution with initial value $g$ if
	$\mathcal A(t,u(t))$ is defined for almost every
	$t\in(s,T)$ and, for every compact interval
	$I\subset(0,\infty)$,
	\begin{equation*}
		u(\cdot)|_I\in C([s,T];L^\infty(I)),
	\end{equation*}
	\begin{equation*}
		\mathcal A(\cdot,u(\cdot))|_I
		\in L^1((s,T);L^\infty(I)),
	\end{equation*}
	and
	\begin{equation*}
		u(t)|_I
		=
		g|_I+\int_s^t\mathcal A(r,u(r))|_I\,\dd r,
		\qquad s\le t\le T.
	\end{equation*}
	The integral is a Bochner integral in $L^\infty(I)$,
	and the equality holds in that space.
	In particular, $u(s)=g$ and $u$ has local initial trace
	$g$ at time $s$.
	
	In particular, a global integral solution of
	\eqref{eq:PDE} satisfies
	\begin{equation}\label{eq:integralsolution}
		f(t)|_I
		=
		f_0|_I+\int_0^t\mathcal C(f(r))|_I\,\dd r,
		\qquad t\ge0,
	\end{equation}
	for every compact interval $I\subset(0,\infty)$,
	with the same interpretation of the integral and equality.
\end{definition}

Our first result establishes global stability for the
unforced equation \eqref{eq:PDE}.

\begin{theorem}[Global stability]\label{thm:nonlinear}
	For the fixed $\delta,\chi$ in \eqref{eq:parameters}--\eqref{eq:modes},
	there exist $\varepsilon>0$, $C<\infty$, and $R>0$ such that,
	if $f_0-f_M\in\Z$ and $\norm{f_0-f_M}_\Z\le\varepsilon$, equation
	\eqref{eq:PDE} has a global integral solution satisfying
	\begin{equation*}
		\sup_{t\ge0}\norm{f(t)-f_M}_\Z\le C\norm{f_0-f_M}_\Z.
	\end{equation*}
	The solution has the local initial trace $f_0$ and is unique among
	global integral solutions with
	$\sup_{t\ge0}\norm{f(t)-f_M}_\Z\le R$.
	
\end{theorem}

We next consider an external force that may depend on both
time and frequency:
\begin{equation*}
	\partial_t f(t,\omega)
	=\mathcal C(f(t))(\omega)+\phi(t,\omega),
	\qquad f(0,\omega)=f_0(\omega).
\end{equation*}
This is equation \eqref{eq:forced-equation}.
The force is measured in the space $\mathcal P$, defined
in Section~\ref{sec:time-dependent-forcing} by the
space-time norm \eqref{eq:force-norm}.
In particular, $\mathcal P$ contains
$L^\infty((0,\infty);Y)$.
The following theorem extends
Theorem~\ref{thm:nonlinear} to sufficiently small forces
in $\mathcal P$, without requiring convergence of
$\phi(t)$ as $t\to\infty$.

\begin{theorem}[Global stability with external force]\label{thm:forced-stability}
	There exist constants
	$\varepsilon_{\mathrm{ext}}>0$,
	$R_{\mathrm{ext}}>0$, and
	$C_{\mathrm{ext}}<\infty$, with the following property.
	
	Let $f_0-f_M\in\mathcal Z$, $\phi\in\mathcal P$ and 
	\begin{equation}\label{eq:forced-smallness}
		\|f_0-f_M\|_{\mathcal Z}
		+\|\phi\|_{\mathcal P}
		\le\varepsilon_{\mathrm{ext}}.
	\end{equation}
	Then \eqref{eq:forced-equation} has a global integral
	solution satisfying
	\begin{equation}\label{eq:forced-stability}
		\sup_{t\ge0}\|f(t)-f_M\|_{\mathcal Z}
		\le
		C_{\mathrm{ext}}
		\bigl(
		\|f_0-f_M\|_{\mathcal Z}
		+\|\phi\|_{\mathcal P}
		\bigr).
	\end{equation}
	The solution is unique among integral solutions for
	which $f-f_M$ satisfies
	\begin{equation*}
		\sup_{t\ge0}\|f(t)-f_M\|_{\mathcal Z}
		\le R_{\mathrm{ext}}.
	\end{equation*}
	
\end{theorem}

When the force is independent of time, uniform stability
can be strengthened to convergence to a stationary solution.
This addresses, for small perturbations of $f_M$, the
question raised in \cite[Section~1.3]{CDG} of whether
solutions of the forced evolution approach a stationary
state.
The conclusion below concerns this question near the mass
Kolmogorov--Zakharov spectrum; it does not establish the
full conjectured scenario for arbitrary initial data
and stationary states with different power laws at zero
and infinity.

More precisely, the high-frequency coefficient
$a_\infty(f(t))$ remains equal to $a_\infty(f_0)$ and
determines the stationary solution approached by $f(t)$.
Convergence holds in every $\mathcal Z_\theta$ with
$0<\theta<\delta$, with the algebraic rate stated in
\eqref{eq:asymptotic-nonlinear-rate}.
Here $\mathcal Z_\theta$ is defined in the same way as
$\mathcal Z=\mathcal Z_\delta$, with $\delta$ replaced
by $\theta$, and $\beta>0$ is specified in
\eqref{eq:asymptotic-beta}.
The theorem also gives a necessary and sufficient
condition for convergence in the original space
$\mathcal Z_\delta$.

\begin{theorem}[Convergence under time-independent forcing]
	\label{thm:stationary-convergence}
	There exists $\varepsilon_*>0$ such that the following holds.
	Suppose that
	\begin{equation*}
		f_0-f_M\in\mathcal Z_\delta,
		\qquad \phi\in Y_\delta,
		\qquad
		\|f_0-f_M\|_{\mathcal Z_\delta}+\|\phi\|_{Y_\delta}
		\le\varepsilon_*.
	\end{equation*}
	Let $f$ be the global integral solution of
	\eqref{eq:asymptotic-equation} given by
	Theorem~\ref{thm:forced-stability}.
	Set $b=a_\infty(f_0)$, and let $f_*$ be the stationary solution
	in Lemma~\ref{lem:asymptotic-stationary} with
	$a_\infty(f_*)=b$.
	
	For every $0<\theta<\delta$,
	\begin{equation}\label{eq:asymptotic-nonlinear-rate}
		\|f(t)-f_*\|_{\mathcal Z_\theta}
		\lesssim_\theta
		(1+t)^{-\beta\min\{1,\,2(\delta-\theta)/\delta\}}
		\|f_0-f_*\|_{\mathcal Z_\delta},
		\qquad t\ge0.
	\end{equation}
	Moreover,
	\begin{equation}\label{eq:asymptotic-original-convergence}
		\lim_{t\to\infty}\|f(t)-f_*\|_{\mathcal Z_\delta}=0
	\end{equation}
	holds if and only if
	\begin{equation}\label{eq:asymptotic-initial-tail}
		\lim_{R\to\infty}
		\operatorname*{ess\,sup}_{\omega\ge R}
		\omega^{7/6+\delta}
		|f_0(\omega)-b f_M(\omega)|=0.
	\end{equation}
	The smallness threshold $\varepsilon_*$ is independent of $\theta$.
\end{theorem}

\subsection{Disclosure, organization, and notation}

\vspace{10pt}
\noindent
\textbf{Disclosure of AI assistance.}
The author used ChatGPT-6 Astra during the preparation and
revision of this paper to assist with developing partial
arguments, checking calculations, and revising the exposition
and \LaTeX{} code.
The author takes full responsibility for the mathematical
claims, proofs, and references presented in the paper.

\vspace{10pt}
\noindent
\textbf{Organization.}
Section~\ref{sec:inputs} states the results from the literature
used in the proof and specifies their normalizations.
Section~\ref{sec:fixed} establishes the estimates for the fixed
linear evolution and the commutator.
Section~\ref{sec:linear} constructs the evolution operator
$U_{a_0,a_\infty}$ for the linear equation with time-dependent
coefficients and proves its bounds, uniqueness, composition
property, and dependence on the coefficients.
Section~\ref{nonlinear} proves the estimates for
the nonlinear remainder.
Section~\ref{sec:contraction} combines these results to prove
Theorem~\ref{thm:nonlinear}.
Section~\ref{sec:time-dependent-forcing} introduces the
space-time norm for the external force and proves the
corresponding global stability Theorem \ref{thm:forced-stability}.
The final section proves
Theorem~\ref{thm:stationary-convergence}, establishing
convergence to a stationary solution under time-independent
forcing, with algebraic decay in weaker weighted norms
and a necessary and sufficient condition for convergence
in the original stability norm.
The appendix justifies the differentiation, commutation,
and integration formulas used in the main text.

\vspace{10pt}
\noindent
\textbf{Notation.}
\begin{itemize}
	\item By $u\in C([0,T];B)$ (or $L^p([0,T];B)$) for a
	Banach space $B$, we mean that $u$ is a continuous
	(or Bochner $L^p$) map from $[0,T]$ to $B$;
	see \cite[page 301]{evans2022partial}.
	
	\item By $A\lesssim B$ (resp. $A\sim B$), we mean that
	there exists a positive constant $C$ such that $A\le CB$
	(resp. $C^{-1}B\le A\le CB$).
	If the constant $C$ depends on $p$, then we write
	$A\lesssim_p B$ (resp. $A\sim_p B$).
	
	\item By $A\ll B$, we mean that $A/B$ is sufficiently small.

	\item By $a+$, we mean $a+\varepsilon$ for an arbitrarily
	small $\varepsilon>0$. An estimate with exponent $a+$
	holds for every $\varepsilon>0$, with an implicit constant
	that may depend on $\varepsilon$.

\end{itemize}

\vspace{10pt}
\section{External results: statements and normalizations}\label{sec:inputs}

\subsection{Linear operator and fundamental solution}
Define the linearization of the evolution equation at $f_M$ by
\begin{equation}\label{eq:Ldefinition}
	Lg=3\C(f_M,f_M,g).
\end{equation}

\begin{proposition}[{\cite[Proposition 2.1]{EMVL}}]
The operator \eqref{eq:Ldefinition} can be written
\begin{equation}\label{eq:L}
 Lg(\omega)=-\mathfrak a\omega^{-1/3}g(\omega)
 +\omega^{-4/3}\int_0^\infty K(r/\omega)g(r)\dd r.
\end{equation}
Here $\mathfrak a>0$, and $K$ is smooth off $1$. There are real constants
$A_1,\ldots,A_6$ such that
\begin{equation*}
K(r)\sim A_1r^{1/2}\quad(r\downarrow0),\qquad K(r)\sim A_2r^{-7/6}\quad(r\to\infty).
\end{equation*}
\begin{equation*}
K(r)=A_3(1-r)^{-5/6}+A_4+O((1-r)^{1/6}) (r\uparrow1),
\end{equation*}
\begin{equation*}
K(r)=A_5(r-1)^{-5/6}+A_6+O((r-1)^{1/6}) (r\downarrow1).
\end{equation*}

The implicit constants and $K$ are determined by the collision integral
\eqref{eq:Ldefinition}; explicit formulas are given in \cite[Appendix A]{EMVL}.
\end{proposition}
There is no free choice of $K$ or $\mathfrak a$ in this proposition. For
example $\mathfrak a$ is minus the coefficient of $g_1$ at $\omega_1=1$,
namely
\begin{equation*}
\fitdisplay{\mathfrak a=-\iint_{\omega_3,\omega_4>0,\ \omega_2>0}
 W\big|_{\omega_1=1}\left[\omega_3^{-7/6}\omega_4^{-7/6}
 -\omega_2^{-7/6}(\omega_3^{-7/6}+\omega_4^{-7/6})\right]\,d\omega_3\,d\omega_4,
 \quad\omega_2=\omega_3+\omega_4-1.}
\end{equation*}
The locally integrable kernel $K$ is uniquely characterized by
\eqref{eq:L} on smooth compactly supported $g$. We only need the consequence
\begin{equation}\label{eq:Kbounds}
 |K(r)|\lesssim \begin{cases}
 r^{1/2},&0<r\le1/2,\\
 |r-1|^{-5/6},&1/2<r<2,\\
 r^{-\frac76},&r\ge2.
 \end{cases}
\end{equation}

\begin{theorem}[{\cite[Theorem 3.5]{EMVN}}]\label{thm:EMV}
For each $r>0$ the equation $\partial_tF=LF$ with initial measure
$\delta_r$ has a unique fundamental solution in the class described below.
It has a Dirac part at $\omega=r$ and a function part; the latter has
bounds $C\omega^{-\frac76}$ for $\omega\le r/2$, $C\omega^{-\frac{11}{6}}$ for
$\omega\ge2r$, and $C|\omega-r|^{-5/6}$ for $|\omega-r|\le r/2$,
with local bounds elsewhere. The constants here may depend on $t,r$.
It obeys
\begin{equation*}
 F(t,\omega,r)=r^{-1}F(t/r^{1/3},\omega/r,1).
\end{equation*}
For $0<z<2$ it has the decomposition
\begin{equation}\label{eq:EMVdecomp}
 F(\tau,z,1)=e^{-\mathfrak a\tau}\delta_1(z)
             +\sigma(\tau)z^{-\frac76}+R_1(\tau,z)+R_2(\tau,z).
\end{equation}
The function $\sigma$ belongs to $C([0,\infty))$. There is a real constant $A$ such that
$\sigma(\tau)=A\tau^4+O(\tau^{4+\epsilon})$ as $\tau\downarrow0$,
and $\sigma(\tau)=O(\tau^{-(3v_0-5/2)})$ as $\tau\to\infty$.
For every $1<b<\frac76$ and sufficiently small $\epsilon>0$ the remaining bounds are
\begin{equation*}
R_1(\tau,z)=0\quad (|z-1|\ge1/2),\quad 
 |R_1(\tau,z)|\lesssim  e^{-(\mathfrak a-\epsilon)\tau}|z-1|^{-5/6},
\end{equation*}
\begin{equation*}
|R_2(\tau,z)|\lesssim z^{-b}
 \begin{cases}\tau^{3b-5/2-\epsilon},&\tau\le1,\\
 \tau^{3b-3v_0+\epsilon},&\tau>1,
 \end{cases}
\end{equation*}
\begin{equation}\label{eq:EMVfar}
|F(\tau,z,1)|\lesssim z^{-\frac{11}{6}}
 \begin{cases}\tau^{1-\epsilon},&\tau\le1,\\
 \tau^{-(3v_0-9/2)+\epsilon},&\tau>1,
 \end{cases}\qquad z>2.
\end{equation}

Here $v_0=1.84020\ldots$ is computed numerically.
We use only the inequality $v_0>11/6$.
\end{theorem}

The corresponding solution operators are given by
\begin{equation*}
 (S(t)g)(\omega)=\int_0^\infty F(t,\omega,r)g(r)\dd r, \qquad \forall t>0.
\end{equation*}
We also additionally set $S(0)=\Id$. 

We take the following \eqref{eq:fundamental-tested-identity} and
\eqref{eq:fundamental-initial-trace} as the interpretation of the
fundamental solution $F$ satisfying the equation $\partial_tF=LF$ and initial data $\delta_r$: for $0<t_1<t_2$, $r>0$, and
$\eta\in C_c^\infty(0,\infty)$,
\begin{equation}\label{eq:fundamental-tested-identity}
 \int\eta(\omega)\,[F(t_2,\omega,r)-F(t_1,\omega,r)]d\omega
 =\int_{t_1}^{t_2}\langle LF(\tau,\cdot,r),\eta\rangle\,d\tau.
\end{equation}
The pairing is obtained by applying the kernel formula \eqref{eq:L}
to $F$ and testing in $\omega$. The initial condition means
\begin{equation}\label{eq:fundamental-initial-trace}
 \lim_{t\downarrow0}\int\eta(\omega)F(t,\omega,r)d\omega=\eta(r).
\end{equation}

\begin{proposition}[{\cite[Remark 3.8]{EMVN}}]
\label{prop:derivative}
 For every $1<b<7/6$ and sufficiently small $\epsilon>0$,
the following estimates hold. At positive times the function part of the
kernel is differentiable in time away from its singularity $\omega=r$.
For each $\tau>0$, the differentiated decomposition on $(0,2)$ is
\begin{equation*}
 \tau\partial_\tau F(\tau,z,1)
 =-\mathfrak a\tau e^{-\mathfrak a\tau}\delta_1(z)
   +\tau\sigma'(\tau)z^{-7/6}
   +\tau\partial_\tau R_1(\tau,z)+\tau\partial_\tau R_2(\tau,z).
\end{equation*}
This equality is understood in $\mathcal D'((0,2))$.
The coefficient satisfies
\begin{equation}\label{eq:EMVsigma-derivative}
 |\tau\sigma'(\tau)|\lesssim 
 \begin{cases}\tau^4,&0<\tau\le1,\\
 \tau^{-3v_0+5/2},&\tau>1.\end{cases}
\end{equation}
The first remainder has $\partial_\tau R_1(\tau,z)=0$ when
$|z-1|\ge1/2$, and
\begin{equation}\label{eq:EMVR1-derivative}
 |\tau\partial_\tau R_1(\tau,z)|
 \lesssim e^{-(\mathfrak a-\epsilon)\tau}|z-1|^{-5/6},
 \qquad 0<|z-1|<1/2.
\end{equation}
The second remainder satisfies
\begin{equation}\label{eq:EMVR2-derivative}
 |\tau\partial_\tau R_2(\tau,z)|\lesssim z^{-b}
 \begin{cases}\tau^{3b-5/2-\epsilon},&0<\tau\le1,\\
 \tau^{3b-3v_0+\epsilon},&\tau>1,\end{cases}
 \qquad 0<z<2.
\end{equation}
For $z>2$, the differentiated kernel satisfies
\begin{equation}\label{eq:EMVfar-derivative}
 |\tau\partial_\tau F(\tau,z,1)|\lesssim z^{-11/6}
 \begin{cases}\tau^{1-\epsilon},&0<\tau\le1,\\
 \tau^{-3v_0+9/2+\epsilon},&\tau>1.\end{cases}
\end{equation}
The implicit constants may depend on $b,\epsilon$, but not on $\tau,z$.
\end{proposition}

\subsection{The collision estimates }
\begin{lemma}[{\cite[Lemma 7.1 and Remark 7.2]{CDG}}]\label{lem:CDG}
Let $f\in E_{a_f,b_f}$, $g\in E_{a_g,b_g}$, $h\in E_{a_h,b_h}$, with
$0\le a_f,a_g,a_h<5/4$, none equal to $1$, and
$b_f,b_g,b_h>1$, with $b_f+b_g\ne5/2$.
For nonnegative inputs define
\begin{equation*}
 \widetilde\C(f,g,h)=2\iint_{0<\omega_3<\omega_4,\ \omega_2>0}
 W[(f_1+f_2)g_3h_4+(g_3+g_4)h_1f_2]\dd\omega_3\dd\omega_4.
\end{equation*}
Then
\begin{equation*}
 \norm{\widetilde\C(|f|,|g|,|h|)}_{E_{a',b'}}
 \le C\norm{f}_{E_{a_f,b_f}}\norm{g}_{E_{a_g,b_g}}\norm{h}_{E_{a_h,b_h}},
\end{equation*}
where
\begin{equation*}
a'=\max(a_f,a_h,a_f+a_g+a_h-2,a_f+a_g-1,a_h+a_g-1),
\end{equation*}
\begin{equation*}
b'=\min(b_f+b_g+b_h-2,b_h+1/2).
\end{equation*}

Also $|\overline\C(f,g,h)|\le\widetilde\C(|f|,|g|,|h|)$.
The excluded exponent values avoid logarithmic endpoint terms.
\end{lemma}
For the symmetric $\C$ in \eqref{eq:barC}, apply this lemma to all six
permutations: use the largest resulting low exponent and the smallest
resulting high exponent. 

\begin{lemma}[{\cite[Lemma 7.6]{CDG}}]\label{lem:stationary}
For every $\omega>0$, the absolutely convergent collision integral satisfies
$\C(f_M)(\omega)=0$.
\end{lemma}
Consequently $Lf_M=3\C(f_M)=0$. We will prove the special cutoff estimates needed later directly from Lemmas~\ref{lem:CDG} and~\ref{lem:stationary}.

\vspace{10pt}
\section{The fixed linear operator}\label{sec:fixed}

We retain the spaces and operators defined in
Sections~\ref{sec:setting} and~\ref{sec:inputs}.
Choose
\begin{equation*}
	0<\kappa<\min\{3\delta,3v_0-11/2\},
\end{equation*}
and define
\begin{equation*}
	k(t)=
	\begin{cases}
		t^{-1+3\delta},&0<t\le1,\\
		t^{-1-\kappa},&t>1.
	\end{cases}
\end{equation*}
Then $k$ is decreasing and
\begin{equation*}
	\int_0^\infty k(t)\,dt=\frac1{3\delta}+\frac1\kappa<\infty.
\end{equation*}
All constants below depend only on $\delta,\kappa,\chi$ and the
fixed constants in the external kernel estimates.

We set $S(0)=\Id$. For $g\in\mathcal Z\cup Y$, we call $u(t)=S(t)g$ a mild solution of $u_t=Lu$ with initial datum $g$.
Lemmas~\ref{lem:appendix-local-equation}
and~\ref{lem:appendix-initial-trace} show that this mild solution
is also an integral solution on every compact frequency interval.

For $g\in\mathcal Z\cup Y$ and $t>0$, the notation $S'(t)g$
means the strong derivative of $t\mapsto S(t)g$ in $\mathcal Z$:
\begin{equation*}
	S'(t)g=\lim_{h\to0}\frac{S(t+h)g-S(t)g}{h}
	\quad\text{in }\mathcal Z.
\end{equation*}
Its existence and the identities
\begin{equation*}
	S'(t)g=LS(t)g
	=\int_0^\infty\partial_tF(t,\cdot,r)g(r)\,dr
	\quad\text{in }\mathcal Z
\end{equation*}
are established in Lemma~\ref{lem:appendix-Z-differentiability}.

\begin{lemma}[Estimates for $X$ and $Y$ data]\label{lem:fixed-Y}
	For every $t>0$,
	\begin{equation}\label{eq:fixedX}
		\|S(t)\|_{X\to\mathcal Z}
		+\|tS'(t)\|_{X\to\mathcal Z}\lesssim 1,
	\end{equation}
	and
	\begin{equation}\label{eq:fixedY}
		\|S(t)\|_{Y\to\mathcal Z}
		+\|tS'(t)\|_{Y\to\mathcal Z}\lesssim k(t).
	\end{equation}
	For every $0<t_1<t_2$,
	\begin{equation}\label{eq:kernel-X-continuity}
		\|S(t_2)-S(t_1)\|_{X\to\mathcal Z}
		\lesssim \int_{t_1}^{t_2}\frac{dr}{r},
	\end{equation}
	and
	\begin{equation}\label{eq:kernel-Y-continuity}
		\|S(t_2)-S(t_1)\|_{Y\to\mathcal Z}
		\lesssim \int_{t_1}^{t_2}k(r)\,\frac{dr}{r}.
	\end{equation}
	The operators defined for $X$ and $Y$ data agree on $X\cap Y$.
\end{lemma}
\begin{proof}
	
	Choose $\epsilon>0$ small enough for
	Theorem~\ref{thm:EMV} and Proposition~\ref{prop:derivative}, with
	\begin{equation}\label{eq:auxiliary}
		0<\epsilon<
		\min\left\{
		\delta,\frac{\mathfrak a}{2},
		\frac{3v_0-11/2-\kappa}{3}
		\right\}.
	\end{equation}
	Write $v_0^-=v_0-\epsilon$ throughout this proof.
	Then
	\begin{equation*}
		\frac{13}{12}<\frac76-\delta,\qquad
		v_0^->\frac{11}{6},\qquad
		3v_0^--\frac92>1+\kappa.
	\end{equation*}
	We also choose $b=13/12$ in the estimate for $R_2$.

	For $g\in X\cup Y$, the scaling of $F$ and
	\eqref{eq:EMVdecomp} express $S(t)g(\omega)$ as the sum of
	\begin{equation}\label{eq:fixed-part-dirac}
		e^{-\mathfrak a t\omega^{-1/3}}g(\omega),
	\end{equation}
	\begin{equation}\label{eq:fixed-part-R1}
		\int_{\omega/2}^{\infty}
		R_1(t/r^{1/3},\omega/r)g(r)\,\frac{dr}{r},
	\end{equation}
	\begin{equation}\label{eq:fixed-part-R2}
		\int_{\omega/2}^{\infty}
		R_2(t/r^{1/3},\omega/r)g(r)\,\frac{dr}{r},
	\end{equation}
	\begin{equation}\label{eq:fixed-part-far}
		\int_0^{\omega/2}
		F(t/r^{1/3},\omega/r,1)g(r)\,\frac{dr}{r},
	\end{equation}
	and
	\begin{equation}\label{eq:fixed-part-sigma}
		\omega^{-7/6}\int_{\omega/2}^{\infty}
		r^{1/6}\sigma(t/r^{1/3})g(r)\,dr.
	\end{equation}
	We use \eqref{eq:EMVdecomp} only when $r>\omega/2$.
	For \eqref{eq:fixed-part-far}, where $\omega/r>2$,
	we use estimate (\ref{eq:EMVfar}) directly.
	
	We will bound
	\eqref{eq:fixed-part-dirac}--\eqref{eq:fixed-part-far}
	in $X$.
	The expression \eqref{eq:fixed-part-sigma} may have
	$\omega^{-7/6}$ behavior at zero, which is not allowed in $X$.
	We therefore subtract
	\begin{equation}\label{eq:fixed-low-mode}
		(1-\chi(\omega))\omega^{-7/6}
		\int_0^\infty r^{1/6}\sigma(t/r^{1/3})g(r)\,dr.
	\end{equation}
	This is a multiple of $(1-\chi)f_M$ in \eqref{eq:Znorm}.
	We will bound its coefficient and show that
	\eqref{eq:fixed-part-sigma} minus \eqref{eq:fixed-low-mode}
	belongs to $X$.
	
	We first establish the estimates that will be used
	for both $X$ and $Y$.
	Suppose that $|g(r)|\le r^{-a}$, where
	\begin{equation*}
		a\in
		\left\{
		\frac76-\delta,\; \frac76,\; \frac76+\delta,\; 
		\frac32-2\delta,\; \frac32+2\delta
		\right\}.
	\end{equation*}
	All these exponents satisfy
	\begin{equation*}
		\frac56+\frac{\epsilon}{3}<a<v_0^-.
	\end{equation*}

	For the coefficient in \eqref{eq:fixed-low-mode}, we claim
	\begin{equation}\label{eq:cpower}
		\int_0^\infty r^{1/6}|\sigma(t/r^{1/3})|\,|g(r)|\,dr
		\lesssim t^{7/2-3a}.
	\end{equation}
	The bounds for $\sigma$ in Theorem~\ref{thm:EMV} give
	\begin{equation*}
		\int_0^{t^3}r^{1/6}|\sigma(t/r^{1/3})|\,|g(r)|\,dr
		\lesssim
		t^{-3v_0^-+5/2}
		\int_0^{t^3}r^{v_0^--2/3-a}\,dr
		\sim t^{7/2-3a}.
	\end{equation*}
	They also give
	\begin{equation*}
		\int_{t^3}^\infty r^{1/6}|\sigma(t/r^{1/3})|\,|g(r)|\,dr
		\lesssim
		t^4\int_{t^3}^\infty r^{-7/6-a}\,dr
		\sim 
		t^{7/2-3a}.
	\end{equation*}
	Thus, one can conclude \eqref{eq:cpower}.
	
	For \eqref{eq:fixed-part-dirac} and \eqref{eq:fixed-part-R1},
	we claim
	\begin{equation}\label{eq:diagonal}
		|e^{-\mathfrak a t\omega^{-1/3}}g(\omega)|
		+\int_{\omega/2}^\infty
		|R_1(t/r^{1/3},\omega/r)|\,|g(r)|\,\frac{dr}{r}
		\lesssim e^{-ct\omega^{-1/3}}\omega^{-a},
	\end{equation}
	where $c>0$ is fixed.
	The estimate for \eqref{eq:fixed-part-dirac} follows directly
	from $|g(\omega)|\le\omega^{-a}$.
	The support of $R_1$ restricts \eqref{eq:fixed-part-R1}
	to $2\omega/3<r<2\omega$.
	With $r=\omega s$, its absolute integral is bounded by
	\begin{equation*}
		e^{-ct\omega^{-1/3}}\omega^{-a}
		\int_{2/3}^{2}s^{-1-a}\bigg|\frac{1}{s}-1\bigg|^{-5/6}\,ds.
	\end{equation*}
	The integral is finite because $5/6<1$.
	This proves \eqref{eq:diagonal}.
	
	For \eqref{eq:fixed-part-R2}, we can similarly deduce that
	\begin{equation}\label{eq:R2shape}
		\int_{\omega/2}^\infty
		|R_2(t/r^{1/3},\omega/r)|\,|g(r)|\,\frac{dr}{r}
		\lesssim 
		\begin{cases}
			t^{13/4-3a}\omega^{-13/12},&0<\omega\le t^3,\\
			\omega^{-a},&\omega\ge t^3.
		\end{cases}
	\end{equation}

	For \eqref{eq:fixed-part-far}, one also has
	\begin{equation}\label{eq:farshape}
		\int_0^{\omega/2}
		|F(t/r^{1/3},\omega/r,1)|\,|g(r)|\,\frac{dr}{r}
		\lesssim 
		\begin{cases}
			t^{-3v_0^-+9/2}\omega^{v_0^--3/2-a},
			&0<\omega\le t^3,\\
			\omega^{-a},
			&\omega\ge t^3,\ a<3/2+\epsilon/3,\\
			t^{11/2-3a}\omega^{-11/6},
			&\omega\ge t^3,\ a>3/2+\epsilon/3.
		\end{cases}
	\end{equation}
   The calculation and verification are left to interested readers.
	
	We now apply \eqref{eq:diagonal}, \eqref{eq:R2shape},
	and \eqref{eq:farshape} with the weights of $X$ and $Y$.
	By homogeneity, it is enough to consider $\|g\|_X\le1$
	or $\|g\|_Y\le1$.
	The corresponding pointwise bounds are
	\begin{equation*}
		|g(r)|\le
		\min\{r^{-7/6+\delta},r^{-7/6-\delta}\},
		\qquad \|g\|_X\le1,
	\end{equation*}
	and
	\begin{equation*}
		|g(r)|\le
		\min\{r^{-3/2+2\delta},r^{-3/2-2\delta}\},
		\qquad \|g\|_Y\le1.
	\end{equation*}

	For $\|g\|_X\le1$ and
	$a\in\{7/6-\delta,\,7/6+\delta\}$, we claim that
	each of \eqref{eq:fixed-part-dirac}--\eqref{eq:fixed-part-far}
	is bounded by $C\omega^{-a}$.
	For \eqref{eq:fixed-part-dirac} and \eqref{eq:fixed-part-R1},
	this follows from \eqref{eq:diagonal}.
	For $\omega\le t^3$, the bounds in
	\eqref{eq:R2shape} and \eqref{eq:farshape} satisfy
	\begin{equation*}
		t^{13/4-3a}\omega^{-13/12}
		=
		\omega^{-a}\left(\frac{\omega}{t^3}\right)^{a-13/12}
		\le\omega^{-a},
	\end{equation*}
	and
	\begin{equation*}
		t^{-3v_0^-+9/2}\omega^{v_0^--3/2-a}
		=
		\omega^{-a}
		\left(\frac{\omega}{t^3}\right)^{v_0^--3/2}
		\le\omega^{-a}.
	\end{equation*}
	Here $a>13/12$ and $v_0^->3/2$.
	For $\omega\ge t^3$, both
	\eqref{eq:R2shape} and \eqref{eq:farshape}
	already give $C\omega^{-a}$, since $a<3/2+\epsilon/3$.
	Using $a=7/6-\delta$ below $\omega=1$ and
	$a=7/6+\delta$ above $\omega=1$ proves that each of
	\eqref{eq:fixed-part-dirac}--\eqref{eq:fixed-part-far}
	has $X$-norm $\lesssim 1$.
	
	For $\|g\|_Y\le1$, we first estimate
	\eqref{eq:fixed-part-dirac},
	\eqref{eq:fixed-part-R1}, and
	\eqref{eq:fixed-part-R2} in $X$.
	Using \eqref{eq:diagonal} and \eqref{eq:R2shape},
	with $a=3/2-2\delta$ for $0<t\le1$ and
	$a=3/2+2\delta$ for $t\ge1$, we obtain,
	after simple calculations, the following bound
	for the $X$-norm of each of these expressions:
	\begin{equation*}
		\lesssim
		\begin{cases}
			t^{-1+3\delta}, & 0<t\le1,\\
			t^{-1-3\delta}, & t\ge1.
		\end{cases}
	\end{equation*}
	
	We next estimate \eqref{eq:fixed-part-far} in $X$.
	For $0<t\le1$, apply \eqref{eq:farshape}
	with $a=3/2-2\delta$.
	For $t\ge1$, apply \eqref{eq:farshape}
	with $a=3/2-2\delta$ on $0<\omega\le1$
	and $a=3/2+2\delta$ on $\omega\ge1$.
	After simple calculations, the $X$-norm of
	\eqref{eq:fixed-part-far} is bounded by
	\begin{equation*}
		\lesssim
		\begin{cases}
			t^{-1+3\delta}, & 0<t\le1,\\
			t^{-3v_0^-+9/2}+t^{-1-3\delta}, & t\ge1.
		\end{cases}
	\end{equation*}
	
	It remains to estimate \eqref{eq:fixed-part-sigma}
	minus \eqref{eq:fixed-low-mode}, namely
	\begin{equation}\label{eq:residue}
		\omega^{-7/6}
		\left[
		\int_{\omega/2}^\infty r^{1/6}\sigma(t/r^{1/3})g(r)\,dr
		-(1-\chi(\omega))
		\int_0^\infty r^{1/6}\sigma(t/r^{1/3})g(r)\,dr
		\right].
	\end{equation}
	We claim that \eqref{eq:residue} has $X$ norm $\lesssim 1$
	for $\|g\|_X\le1$, and 
	$\lesssim t^{-1+3\delta}$ for $0<t\le1$ and
	$\lesssim t^{-1-3\delta}$ for $t\ge1$ when $\|g\|_Y\le1$.
	
For $0<\omega\le1$, the bracket in \eqref{eq:residue} equals
$-\int_0^{\omega/2}r^{1/6}\sigma(t/r^{1/3})g(r)\,dr$.
On $\omega\le\min(1,t^3)$, use
$|\sigma(\tau)|\lesssim\tau^{-3v_0^-+5/2}$ for $\tau\ge1$
from Theorem~\ref{thm:EMV}.
On $t^3<\omega\le1$, use \eqref{eq:cpower}.
Choosing $a=7/6-\delta$ or $3/2-2\delta$ for $t\le1$,
and $a=7/6+\delta$ or $3/2+2\delta$ for $t\ge1$,
we obtain, after simple calculations,
\begin{equation*}
	\sup_{0<\omega\le1}
	\omega^{-\delta}
	\left|
	\int_0^{\omega/2}
	r^{1/6}\sigma(t/r^{1/3})g(r)\,dr
	\right|
	\lesssim
	\begin{cases}
		1, & \|g\|_X\le1,\ t>0,\\
		t^{-1+3\delta}, & \|g\|_Y\le1,\ 0<t\le1,\\
		t^{-1-3\delta}, & \|g\|_Y\le1,\ t\ge1.
	\end{cases}
\end{equation*}

For $\omega\ge 2$, \eqref{eq:residue} equals
\eqref{eq:fixed-part-sigma}.
On $\omega\ge\max(2,2t^3)$, use
$|\sigma(\tau)|\lesssim\tau^4$ for $0<\tau\le1$
from Theorem~\ref{thm:EMV}.
On $2\le\omega\le2t^3$, use \eqref{eq:cpower}.
Taking $a=7/6+\delta$ for $X$ data and
$a=3/2+2\delta$ for $Y$ data gives, after simple calculations,
\begin{equation*}
	\sup_{\omega\ge 2}
	\omega^\delta
	\left|
	\int_{\omega/2}^\infty
	r^{1/6}\sigma(t/r^{1/3})g(r)\,dr
	\right|
	\lesssim
	\begin{cases}
		1, & \|g\|_X\le1,\ t>0,\\
		t^{-1+3\delta}, & \|g\|_Y\le1,\ 0<t\le1,\\
		t^{-1-3\delta}, & \|g\|_Y\le1,\ t\ge1.
	\end{cases}
\end{equation*}

On $[1,2]$, the absolute value of \eqref{eq:residue},
multiplied by the $X$ weight, is bounded by
\eqref{eq:cpower}, which gives
\begin{equation}\label{eq:cfinal}
	\int_0^\infty
	r^{1/6}|\sigma(t/r^{1/3})|\,|g(r)|\,dr
	\lesssim
	\begin{cases}
		t^{3\delta}, & \|g\|_X\le1,\ 0<t\le1,\\
		t^{-3\delta}, & \|g\|_X\le1,\ t\ge1,\\
		t^{-1+6\delta}, & \|g\|_Y\le1,\ 0<t\le1,\\
		t^{-1-6\delta}, & \|g\|_Y\le1,\ t\ge1.
	\end{cases}
\end{equation}
These bounds complete the estimate for \eqref{eq:residue}.
	
	Combining the bounds for
	\eqref{eq:fixed-part-dirac}--\eqref{eq:fixed-part-far}
	and \eqref{eq:residue}, and restoring the norm of $g$, gives
	\begin{equation}\label{eq:fixed-X-remainder}
		\left\|
		S(t)g-(1-\chi)f_M
		\int_0^\infty r^{1/6}\sigma(t/r^{1/3})g(r)\,dr
		\right\|_X
		\lesssim \|g\|_X,\qquad g\in X,
	\end{equation}
	and
	\begin{equation}\label{eq:remainderfinal}
		\left\|
		S(t)g-(1-\chi)f_M
		\int_0^\infty r^{1/6}\sigma(t/r^{1/3})g(r)\,dr
		\right\|_X
		\lesssim k(t)\|g\|_Y,\qquad g\in Y.
	\end{equation}
	For $t\ge1$, \eqref{eq:auxiliary} ensures
	\begin{equation*}
		t^{-1-3\delta}+t^{-3v_0^-+9/2}
		\le 2t^{-1-\kappa}.
	\end{equation*}
	By \eqref{eq:Znorm}, the estimates
	\eqref{eq:cfinal}--\eqref{eq:remainderfinal} therefore imply
	\begin{equation}\label{eq:kernel-XZ}
		\left\|\int_0^\infty F(t,\cdot,r)g(r)\,dr\right\|_{\mathcal Z}
		\lesssim \|g\|_X,\qquad g\in X,
	\end{equation}
	and
	\begin{equation}\label{eq:kernel-YZ}
		\left\|\int_0^\infty F(t,\cdot,r)g(r)\,dr\right\|_{\mathcal Z}
		\lesssim k(t)\|g\|_Y,\qquad g\in Y.
	\end{equation}
	
   Similarly, by using estimates in Proposition \ref{prop:derivative}, we can derive
	\begin{equation}\label{eq:kernel-derivative-XZ}
		\left\|
		\int_0^\infty t\partial_tF(t,\cdot,r)g(r)\,dr
		\right\|_{\mathcal Z}
		\lesssim \|g\|_X,\qquad g\in X,
	\end{equation}
	and
	\begin{equation}\label{eq:kernel-derivative-YZ}
		\left\|
		\int_0^\infty t\partial_tF(t,\cdot,r)g(r)\,dr
		\right\|_{\mathcal Z}
		\lesssim k(t)\|g\|_Y,\qquad g\in Y.
	\end{equation}
	For \eqref{eq:kernel-X-continuity} and
	\eqref{eq:kernel-Y-continuity}, fix $0<t_1<t_2$.
	By \eqref{eq:fixed-scalar-time-integral} in
	Lemma~\ref{lem:appendix-positive-continuity},
	\begin{equation*}
		S(t_2)g(\omega)-S(t_1)g(\omega)
		=
		\int_{t_1}^{t_2}
		\left(
		\int_0^\infty
		\partial_\tau F(\tau,\omega,r)g(r)\,dr.
		\right)d\tau
	\end{equation*}
	
	We first estimate the difference of the coefficients in
	\eqref{eq:fixed-low-mode}.
	Using \eqref{eq:EMVsigma-derivative} as in
	\eqref{eq:cpower} and \eqref{eq:cfinal} gives
	\begin{equation*}
		\int_0^\infty
		r^{-1/6}|\sigma'(\tau/r^{1/3})|\,|g(r)|\,dr
		\lesssim
		\begin{cases}
			\tau^{-1}\|g\|_X, & g\in X,\\
			\tau^{-1}k(\tau)\|g\|_Y, & g\in Y.
		\end{cases}
	\end{equation*}
	Both bounds are integrable over $[t_1,t_2]$.
	The fundamental theorem of calculus and Fubini therefore give
	\begin{equation}\label{eq:fixed-coefficient-time-integral}
		\int_0^\infty r^{1/6}
		[\sigma(t_2/r^{1/3})-\sigma(t_1/r^{1/3})]g(r)\,dr
		=
		\int_{t_1}^{t_2}\int_0^\infty
		r^{-1/6}\sigma'(\tau/r^{1/3})g(r)\,dr\,d\tau.
	\end{equation}
In particular,
\begin{equation*}
	\left|
	\int_0^\infty r^{1/6}
	[\sigma(t_2/r^{1/3})-\sigma(t_1/r^{1/3})]g(r)\,dr
	\right|
	\lesssim
	\begin{cases}
		\displaystyle
		\|g\|_X\int_{t_1}^{t_2}\frac{d\tau}{\tau},
		& g\in X,\\[6pt]
		\displaystyle
		\|g\|_Y\int_{t_1}^{t_2}k(\tau)\,\frac{d\tau}{\tau},
		& g\in Y.
	\end{cases}
\end{equation*}
	
	We next estimate the difference after subtracting
	\eqref{eq:fixed-low-mode}.
	Subtracting \eqref{eq:fixed-coefficient-time-integral},
	multiplied by $(1-\chi(\omega))f_M(\omega)$, from
	\eqref{eq:fixed-scalar-time-integral} gives
	\begin{equation*}
		\begin{gathered}
			\left\|
			S(t_2)g-S(t_1)g
			-(1-\chi)f_M
			\int_0^\infty r^{1/6}
			[\sigma(t_2/r^{1/3})-\sigma(t_1/r^{1/3})]g(r)\,dr
			\right\|_X
			\\
			\le
			\int_{t_1}^{t_2}
			\left\|
			\int_0^\infty\partial_\tau F(\tau,\cdot,r)g(r)\,dr
			-(1-\chi)f_M
			\int_0^\infty r^{-1/6}\sigma'(\tau/r^{1/3})g(r)\,dr
			\right\|_X\,d\tau.
		\end{gathered}
	\end{equation*}
	
	The remainder estimates used to prove
	\eqref{eq:kernel-derivative-XZ} and
	\eqref{eq:kernel-derivative-YZ} state that
	\begin{equation*}
		\begin{gathered}
			\left\|
			\int_0^\infty\partial_\tau F(\tau,\cdot,r)g(r)\,dr
			-(1-\chi)f_M
			\int_0^\infty r^{-1/6}\sigma'(\tau/r^{1/3})g(r)\,dr
			\right\|_X
			\\
			\lesssim
			\begin{cases}
				\tau^{-1}\|g\|_X, & g\in X,\\
				\tau^{-1}k(\tau)\|g\|_Y, & g\in Y.
			\end{cases}
		\end{gathered}
	\end{equation*}
	By \eqref{eq:Znorm}, it proves
	\eqref{eq:kernel-X-continuity} and
	\eqref{eq:kernel-Y-continuity}.

	Finally, Lemma~\ref{lem:appendix-Z-differentiability}
	identifies the expressions in
	\eqref{eq:kernel-derivative-XZ} and
	\eqref{eq:kernel-derivative-YZ} with $tS'(t)g$.
	Together with \eqref{eq:kernel-XZ} and
	\eqref{eq:kernel-YZ}, this proves
	\eqref{eq:fixedX} and \eqref{eq:fixedY}.
	The definitions on $X$ and $Y$ agree on $X\cap Y$
	because they use the same absolutely convergent kernel integral.

\end{proof}

The next lemma supplies the short-time estimate in $Y$
needed to verify the local integral equation for the
inhomogeneous evolution.

\begin{lemma}[Short-time estimate in $Y$]
	For every $0<t\le1$ and almost every $\omega>0$,
	\begin{equation}\label{eq:kernel-YY-absolute}
		\int_0^\infty |F(t,\omega,r)|
		\min\{r^{-3/2+2\delta},r^{-3/2-2\delta}\}\,dr
		\lesssim
		\min\{\omega^{-3/2+2\delta},\omega^{-3/2-2\delta}\}.
	\end{equation}
	Consequently,
	\begin{equation}\label{eq:fixedYY}
		\sup_{0<t\le1}\|S(t)\|_{Y\to Y}\lesssim 1 .
	\end{equation}
\end{lemma}
\begin{proof}
	We estimate the absolute contributions in
	\eqref{eq:fixed-part-dirac}--\eqref{eq:fixed-part-sigma}
	with the weight defining $Y$.
	
	For \eqref{eq:fixed-part-dirac},
	\eqref{eq:fixed-part-R1}, and \eqref{eq:fixed-part-R2},
	apply \eqref{eq:diagonal} and \eqref{eq:R2shape}
	with $a=3/2-2\delta$ on $\omega\le1$
	and $a=3/2+2\delta$ on $\omega\ge1$.
	After simple calculations, each contribution
	has $Y$-norm $\lesssim1$.
	
	For \eqref{eq:fixed-part-far}, the same choices of $a$
	in \eqref{eq:farshape} give the following bounds:
	\begin{equation*}
		\lesssim \begin{cases}
			t^{-3v_0^-+9/2}\omega^{v_0^--3/2},
			& 0<\omega\le t^3,\\
			1, & t^3\le\omega\le1,\\
			t^{1-6\delta}\omega^{-1/3+2\delta},
			& \omega\ge1.
		\end{cases}
	\end{equation*}
	Each expression is $\lesssim1$, since
	$v_0^->3/2$ and $0<\delta<1/12$.
	
	For \eqref{eq:fixed-part-sigma}, we estimate the absolute
	integral without subtracting \eqref{eq:fixed-low-mode}.
	On $\omega\le2t^3$, apply \eqref{eq:cpower};
	on $\omega\ge2t^3$, apply
	$|\sigma(\tau)|\lesssim\tau^4$ for $0<\tau\le1$
	from Theorem~\ref{thm:EMV}.
	Choosing $a=3/2-2\delta$ on $\omega\le1$
	and $a=3/2+2\delta$ on $\omega\ge1$ gives,
	after simple calculations, the bounds
	\begin{equation*}
		\lesssim \begin{cases}
			t^{-1+6\delta}\omega^{1/3-2\delta},
			& 0<\omega\le\min(1,2t^3),\\
			t^4\omega^{-4/3},
			& \omega\ge2t^3,\\
			t^{-1-6\delta}\omega^{1/3+2\delta},
			& 1\le\omega\le2t^3,\; t>2^{-1/3}.
		\end{cases}
	\end{equation*}
	Consequently, they are all $\lesssim1$. Summing all the contributions in
	\eqref{eq:fixed-part-dirac}--\eqref{eq:fixed-part-sigma}
	proves \eqref{eq:kernel-YY-absolute}.
	
\end{proof}

To extend the $X\to\mathcal Z$ estimate to all of $\mathcal Z$,
we will also need to control $L(\chi f_M)$.
The following lemma provides this control and the two cutoff
estimates used in Section~\ref{sec:linear}.

\begin{lemma}[Cutoff and commutator estimates]\label{lem:commutator}
	For every $g\in E_{7/6,7/6}$, define
	\begin{equation*}
		(L\chi-\chi L)g:=L(\chi g)-\chi Lg.
	\end{equation*}
	The terms on the right are defined by the absolutely convergent
	kernel integrals in Lemma~\ref{lem:appendix-L-mapping}.
	For almost every $\omega>0$,
	\begin{equation}\label{eq:fixed-commutator-kernel}
		(L\chi-\chi L)g(\omega)
		=
		\omega^{-4/3}\int_0^\infty
		K(r/\omega)\bigl(\chi(r)-\chi(\omega)\bigr)g(r)\,dr.
	\end{equation}
	Moreover,
	\begin{equation}\label{eq:fixed-commutator-bound}
		\|(L\chi-\chi L)g\|_{E_{1/6,11/6}}
		\lesssim \|g\|_{E_{7/6,7/6}},
	\end{equation}
	and
	\begin{equation}\label{eq:fixed-cutoff-L-bound}
		\|(1-\chi)\chi Lg\|_Y
		\lesssim \|g\|_{E_{7/6,7/6}}.
	\end{equation}
	
	Multiplication by $\chi$ and by $1-\chi$ is bounded on each
	of $X$, $Y$, and $\mathcal Z$.
	In particular,
	\begin{equation}\label{eq:fixed-L-cutoff-mode}
		L(\chi f_M)=(L\chi-\chi L)f_M\in Y,
		\qquad \|L(\chi f_M)\|_Y\lesssim 1.
	\end{equation}
\end{lemma}

\begin{proof}
     Lemma~\ref{lem:appendix-L-mapping} defines $L(\chi g)$ and $Lg$
     by absolutely convergent integrals for $g\in E_{7/6,7/6}$.
	Thus \eqref{eq:fixed-commutator-kernel} is an equality of
	functions almost everywhere.
	
	For \eqref{eq:fixed-commutator-kernel}, the multiplication term
	$-\mathfrak a\omega^{-1/3}g(\omega)$ in \eqref{eq:L}
	commutes with multiplication by $\chi$.
	Subtracting the integral terms in $L(\chi g)$ and $\chi Lg$
	gives exactly \eqref{eq:fixed-commutator-kernel}.
	
	For \eqref{eq:fixed-commutator-bound}, we estimate
	\eqref{eq:fixed-commutator-kernel} near zero, near infinity,
	and on $[1/2,4]$.
	If $\omega\le1/2$, then $\chi(\omega)=0$ and the integrand
	vanishes for $r\le1$.
	Using \eqref{eq:Kbounds},
	\begin{equation*}
		|(L\chi-\chi L)g(\omega)|
		\lesssim
		\|g\|_{E_{7/6,7/6}}\omega^{-1/6}
		\int_1^\infty r^{-7/3}\,dr
		\lesssim \|g\|_{E_{7/6,7/6}}\omega^{-1/6}.
	\end{equation*}
	If $\omega\ge4$, then $\chi(\omega)=1$ and the integrand
	vanishes for $r\ge 2$.
	Again by \eqref{eq:Kbounds},
	\begin{equation*}
		|(L\chi-\chi L)g(\omega)|
		\lesssim
		\|g\|_{E_{7/6,7/6}}\omega^{-11/6}
		\int_0^2 r^{-2/3}\,dr
		\lesssim \|g\|_{E_{7/6,7/6}}\omega^{-11/6}.
	\end{equation*}
	For $1/2\le\omega\le4$, use
	$|\chi(r)-\chi(\omega)|\le1$ and substitute $r=\omega s$
	in \eqref{eq:fixed-commutator-kernel}.
	This gives
	\begin{equation*}
		|(L\chi-\chi L)g(\omega)|
		\le
		\|g\|_{E_{7/6,7/6}}\omega^{-3/2}
		\int_0^\infty |K(s)|s^{-7/6}\,ds.
	\end{equation*}
	The integral is finite: \eqref{eq:Kbounds} gives powers
	$s^{-2/3}$ at zero, $s^{-7/3}$ at infinity, and the
	integrable singularity $|s-1|^{-5/6}$ near one.
	Since $\omega\in[1/2,4]$, the resulting bound is uniform.
	The three frequency ranges prove \eqref{eq:fixed-commutator-bound}.
	
	For \eqref{eq:fixed-cutoff-L-bound}, the factor
	$(1-\chi)\chi$ vanishes outside $[1,2]$.
	The local estimate \eqref{eq:fixed-L-E}, together with boundedness
	of the $Y$ weight on $[1,2]$, gives
	\begin{equation*}
		\|(1-\chi)\chi Lg\|_Y
		\lesssim \|Lg\|_{L^\infty([1,2])}
		\lesssim \|g\|_{E_{7/6,7/6}}.
	\end{equation*}
	
	For boundedness of multiplication on $X$ and $Y$, use
	$0\le\chi\le1$ and $0\le1-\chi\le1$ directly in the norms.
	For $g\in\mathcal Z$, write its decomposition as in
	\eqref{eq:Znorm}. Then
	\begin{equation*}
		\chi g
		=
		g_\infty\chi f_M
		+(g_0-g_\infty)\chi(1-\chi)f_M+\chi h,
	\end{equation*}
	and
	\begin{equation*}
		(1-\chi)g
		=
		g_0(1-\chi)f_M
		+(g_\infty-g_0)\chi(1-\chi)f_M+(1-\chi)h.
	\end{equation*}
	The function $\chi(1-\chi)f_M$ is supported in $[1,2]$
	and belongs to $X$.
	Taking the norms in \eqref{eq:Znorm} therefore gives
	\begin{equation*}
		\|\chi g\|_{\mathcal Z}+\|(1-\chi)g\|_{\mathcal Z}
		\lesssim \bigl(|g_0|+|g_\infty|+\|h\|_X\bigr)
		=\|g\|_{\mathcal Z}.
	\end{equation*}
	
	Finally, Lemma~\ref{lem:stationary} gives $Lf_M=0$,
	and $f_M\in E_{7/6,7/6}$.
	Applying \eqref{eq:fixed-commutator-bound} to $f_M$
	proves \eqref{eq:fixed-L-cutoff-mode}.
	
\end{proof}

We can now combine the estimate for $X$ data with the estimates
for the two functions $(1-\chi)f_M$ and $\chi f_M$.

\begin{lemma}[Estimates for $\mathcal Z$ data]\label{lem:fixed-Z}
	For every $t>0$,
	\begin{equation}\label{eq:fixedZ}
		\|S(t)\|_{\mathcal Z\to\mathcal Z}
		+\|tS'(t)\|_{\mathcal Z\to\mathcal Z}\lesssim 1 .
	\end{equation}
	For every $0<t_1<t_2$,
	\begin{equation}\label{eq:fixed-Z-difference}
		\|S(t_2)-S(t_1)\|_{\mathcal Z\to\mathcal Z}
		\lesssim \int_{t_1}^{t_2}\frac{dr}{r}.
	\end{equation}
	Moreover,
	\begin{equation*}
		S(t)f_M=f_M,\qquad t\ge0.
	\end{equation*}
	The operators defined for $\mathcal Z$ and $Y$ data agree
	on $\mathcal Z\cap Y$.
\end{lemma}

\begin{proof}
	Lemma~\ref{lem:appendix-semigroup-commutation}
	already proves $S(t)f_M=f_M$ for $t\ge0$.
	Its proof uses the local equation, initial trace, and commutation
	with $L$, without using \eqref{eq:fixedZ}.
	
	For \eqref{eq:fixedZ}, we first estimate $S(t)(\chi f_M)$.
	By \eqref{eq:fixed-L-cutoff-mode}, $L(\chi f_M)\in Y$.
	The local integral identity from
	Lemma~\ref{lem:appendix-initial-trace} and the commutation
	identity from Lemma~\ref{lem:appendix-semigroup-commutation}
	give
	\begin{equation}\label{eq:fixed-cutoff-integral}
		S(t)(\chi f_M)
		=
		\chi f_M+\int_0^t S(r)L(\chi f_M)\,dr,
		\qquad t\ge0.
	\end{equation}
		The identity \eqref{eq:fixed-cutoff-integral} initially holds
		in every local space $L^\infty(I)$.
		By \eqref{eq:kernel-Y-continuity},
		$r\mapsto S(r)L(\chi f_M)$ is continuous in $\mathcal Z$
		for $r>0$.
		By \eqref{eq:kernel-YZ}, its $\mathcal Z$ norm is integrable
		down to zero.
		Thus the identity \eqref{eq:fixed-cutoff-integral}
		can be understood in $\mathcal Z$. Consequently, we can obtain
	\begin{equation*}
		\|S(t)(\chi f_M)\|_{\mathcal Z}
		\lesssim
		\|\chi f_M\|_{\mathcal Z}
		+\|L(\chi f_M)\|_Y\int_0^\infty k(r)\,dr
		\lesssim 1.
	\end{equation*}
	
	For the differentiated kernel applied to $\chi f_M$,
	Lemmas~\ref{lem:appendix-local-equation}
	and~\ref{lem:appendix-semigroup-commutation} give
	\begin{equation*}
		\int_0^\infty t\partial_tF(t,\cdot,r)\chi(r)f_M(r)\,dr
		=
		tS(t)L(\chi f_M).
	\end{equation*}
	The equality initially holds locally in $L^\infty$.
	The right-hand side belongs to $\mathcal Z$ by
	\eqref{eq:kernel-YZ}, and
	\begin{equation*}
		\|tS(t)L(\chi f_M)\|_{\mathcal Z}
		\lesssim t\,k(t)\|L(\chi f_M)\|_Y\lesssim 1,
	\end{equation*}
	because
	\begin{equation*}
		tk(t)=
		\begin{cases}
			t^{3\delta}\le1,&0<t\le1,\\
			t^{-\kappa}\le1,&t\ge1.
		\end{cases}
	\end{equation*}
	Thus the differentiated-kernel expression also belongs
	to $\mathcal Z$ and satisfies the same bound.
	
	Since $(1-\chi)f_M=f_M-\chi f_M$ and $S(t)f_M=f_M$,
	the same two bounds hold for $(1-\chi)f_M$.
	For $g\in\mathcal Z$, apply these bounds to its two coefficients
	in \eqref{eq:Znorm}, and apply \eqref{eq:kernel-XZ}
	and \eqref{eq:kernel-derivative-XZ} to its remainder $h\in X$.
	We obtain
	\begin{equation}\label{eq:kernel-ZZ}
		\left\|\int_0^\infty F(t,\cdot,r)g(r)\,dr\right\|_{\mathcal Z}
		+
		\left\|\int_0^\infty t\partial_tF(t,\cdot,r)g(r)\,dr
		\right\|_{\mathcal Z}
		\lesssim \|g\|_{\mathcal Z}.
	\end{equation}
	
	For \eqref{eq:fixed-Z-difference}, we apply \eqref{eq:kernel-YZ} \eqref{eq:fixed-cutoff-integral} and $k(r)\le r^{-1}$
	\begin{equation*}
		\|(S(t_2)-S(t_1))(\chi f_M)\|_{\mathcal Z}
		\lesssim \|L(\chi f_M)\|_Y\int_{t_1}^{t_2}k(r)\,dr
		\lesssim \int_{t_1}^{t_2}\frac{dr}{r}.
	\end{equation*}
	The same estimate also holds for $(1-\chi)f_M$.
	For $h\in X$, use \eqref{eq:kernel-X-continuity}.
	Combining these bounds through \eqref{eq:Znorm} proves
	\eqref{eq:fixed-Z-difference}.
	
	Finally, for $g\in\mathcal Z\cap Y$, both definitions of $S(t)g$
	use the same kernel integral, which is absolutely convergent
	by Lemma~\ref{lem:appendix-absolute-kernel} and
	\eqref{eq:envelope}.
	The two operators therefore agree.
	
\end{proof}

\vspace{10pt}
\section{The equation with coefficients \texorpdfstring{$a_0(t),a_\infty(t)$}{a0(t), a-infinity(t)}}\label{sec:linear}
For measurable real functions $a_0,a_\infty$, consider
\begin{equation}\label{eq:coefficients}
 \partial_t u(t)=
 [a_0(t)^2(1-\chi)+a_\infty(t)^2\chi]Lu(t).
\end{equation}
Here multiplication by $\chi$ acts on the frequency variable. Integral solutions of \eqref{eq:coefficients}
are understood in the sense of Definition~\ref{definition of integral solution}.

\begin{theorem}\label{thm:linear}
There exists $\eta>0$ such that, if
\begin{equation*}
 \max_{j\in\{0,\infty\}}\norm{a_j-1}_{L^\infty(0,\infty)}\le\eta,
\end{equation*}
then \eqref{eq:coefficients} has solution operators for $0\le s<t<\infty$,
$U_{a_0,a_\infty}(t,s)$ with local initial trace and
\begin{equation}\label{eq:L1}
\norm{U_{a_0,a_\infty}(t,s)}_{\Z\to\Z}\lesssim 1,
\end{equation}
\begin{equation}\label{eq:L2}
\norm{U_{a_0,a_\infty}(t,s)}_{Y\to\Z}\lesssim k(t-s).
\end{equation}

If $\widetilde a_0,\widetilde a_\infty$ satisfy the same smallness condition,
then
\begin{equation}\label{eq:L3Z}
\norm{U_{a_0,a_\infty}(t,s)-U_{\widetilde a_0,\widetilde a_\infty}(t,s)}_{\Z\to\Z}
 \lesssim \sum_{j\in\{0,\infty\}}
       \norm{a_j-\widetilde a_j}_{L^\infty(s,t)},
\end{equation}
\begin{equation}\label{eq:L3Y}
\norm{U_{a_0,a_\infty}(t,s)-U_{\widetilde a_0,\widetilde a_\infty}(t,s)}_{Y\to\Z}
 \lesssim k(t-s)\sum_{j\in\{0,\infty\}}
       \norm{a_j-\widetilde a_j}_{L^\infty(s,t)}.
\end{equation}

The two definitions agree on $Y\cap\Z$, and
\begin{equation}\label{eq:UYY}
 \norm{U_{a_0,a_\infty}(t,s)}_{Y\to Y}\lesssim 1 
 \qquad(0<t-s\le1/2).
\end{equation}
For initial data $g\in\mathcal Z$, the integral solution of
\eqref{eq:coefficients} is unique among solutions bounded in
$\mathcal Z$ on every finite time interval and having local
initial trace $g$.
(This solution is $U_{a_0,a_\infty}(t,0)g$, but we postpone
the identification to Lemma~\ref{lem:forcing}). Moreover
\begin{equation*}
 U_{a_0,a_\infty}(t,r)U_{a_0,a_\infty}(r,s)=U_{a_0,a_\infty}(t,s),
 \qquad 0\le s\le r\le t.
\end{equation*}
\begin{equation*}
 U_{a_0,a_\infty}(s,s)=\Id\quad(s\ge0),
 \qquad U_{a_0,a_\infty}(t,s)f_M=f_M\quad(t\ge s\ge0).
\end{equation*}
\end{theorem}
\begin{proof}
To prove this theorem we first estimate the two operators
$S(\int_s^t a_0^2)$ and $S(\int_s^t a_\infty^2)$.
We then construct $U_{a_0,a_\infty}$ by a convergent series, verify
\eqref{eq:coefficients}, prove \eqref{eq:L3Z}--\eqref{eq:L3Y}, and finally
prove uniqueness.

\paragraph{Step 1: the estimates for $S(\int_s^t a_j^2)$.}
Take $\eta\le1/4$. For $j=0,\infty$,
\begin{equation*}
|a_j-1|\le\eta \Longrightarrow\frac12(t-s)\le\int_s^t a_j(\tau)^2\dd\tau \le2(t-s).
\end{equation*}

Also
\begin{equation}\label{eq:timechanges}
\left|\int_s^t a_\infty^2-\int_s^t a_0^2\right|
 \le\int_s^t|a_\infty^2-a_0^2|\dd\tau
 \le5\eta(t-s).
\end{equation}

For each number $r$ between the two integrals in \eqref{eq:timechanges},
\begin{equation*}
 \frac1r\le\frac2{t-s},\qquad
 k(r)\le k((t-s)/2)\le2^{1+\kappa}k(t-s).
\end{equation*}

The positive-time integral estimates in Lemmas~\ref{lem:fixed-Z} and Lemma~\ref{lem:fixed-Y} now give
\begin{equation*}
\norm{S\left(\int_s^t a_\infty(\tau)^2\dd\tau\right)-S\left(\int_s^t a_0(\tau)^2\dd\tau\right)}_{\Z\to\Z}
 \lesssim \frac{1}{t-s}\left|\int_s^t a_\infty^2-\int_s^t a_0^2\right|
\lesssim \eta.
\end{equation*}
and 
\begin{equation}\label{eq:Dbounds}
\norm{S\left(\int_s^t a_\infty(\tau)^2\dd\tau\right)-S\left(\int_s^t a_0(\tau)^2\dd\tau\right)}_{Y\to\Z}
 \lesssim  \frac{k(t-s)}{t-s}\left|\int_s^t a_\infty^2-\int_s^t a_0^2\right|
\lesssim \eta k(t-s).
\end{equation}

Similarly, we also have
\begin{equation}\label{eq:timeLipZ}
 \left\|S\left(\int_s^t a_j^2\right)
          -S\left(\int_s^t\widetilde a_j^2\right)\right\|_{\Z\to\Z}
 \lesssim
           \norm{a_j-\widetilde a_j}_{L^\infty(s,t)}.
\end{equation}
and 
\begin{equation}\label{eq:timeLipY}
 \left\|S\left(\int_s^t a_j^2\right)
          -S\left(\int_s^t\widetilde a_j^2\right)\right\|_{Y\to\Z}
 \lesssim  k(t-s)
          \norm{a_j-\widetilde a_j}_{L^\infty(s,t)}.
\end{equation}

\paragraph{Step 2: construct the approximation to $U_{a_0,a_\infty}$.}
We begin with the explicit operator $Q(t,s)$ defined by
\begin{equation}\label{eq:Q}
 Q(t,s):=(1-\chi)S\left(\int_s^t a_0(\tau)^2\dd\tau\right)+\chi S\left(\int_s^t a_\infty(\tau)^2\dd\tau\right).
\end{equation}
Next, we compute $\left(\partial_t-[a_0^2(1-\chi)+a_\infty^2\chi]L\right)Q$.

	The following computations are formal operator identities.
	To justify them, fix $g\in\mathcal Z\cup Y$ and a compact
	interval $I\subset(0,\infty)$, and apply every expression to $g$.
    Lemma~\ref{lem:appendix-Z-differentiability} gives
    $S(\cdot)g\in C^1((0,\infty);\mathcal Z)$ for
    $g\in\mathcal Z\cup Y$.
    The chain rule for the Lipschitz function
    $t\mapsto\int_s^t a_j(q)^2\,dq$ therefore gives
	\begin{equation*}
		\frac{d}{dt}
		\left(
		S\left(\int_s^t a_j(\tau)^2\,d\tau\right)g
		\right)\Big|_I
		=
		a_j(t)^2
		\left(
		LS\left(\int_s^t a_j(\tau)^2\,d\tau\right)g
		\right)\Big|_I,
		\qquad j\in\{0,\infty\},
	\end{equation*}
	as a strong derivative in $L^\infty(I)$. After some calculations, we see
\begin{equation*}
\left(\partial_t-[a_0^2(1-\chi)+a_\infty^2\chi]L\right)
 (1-\chi)S\left(\int_s^t a_0(\tau)^2\,d\tau\right)
\end{equation*}
is the sum of
\begin{equation*}
-(a_\infty^2-a_0^2)(1-\chi)\chi LS\left(\int_s^t a_0(\tau)^2\dd\tau\right)
\end{equation*}
\begin{equation*}
[a_0^2(1-\chi)+a_\infty^2\chi](L\chi-\chi L)S\left(\int_s^t a_0(\tau)^2\dd\tau\right)
\end{equation*}
Similarly, the expression
\begin{equation*}
\left(\partial_t-[a_0^2(1-\chi)+a_\infty^2\chi]L\right)
 \chi S\left(\int_s^t a_\infty(\tau)^2\,d\tau\right)
\end{equation*}
is the sum of
\begin{equation*}
(a_\infty^2-a_0^2)(1-\chi)\chi LS\left(\int_s^t a_\infty(\tau)^2\dd\tau\right)
\end{equation*}
\begin{equation*}
-[a_0^2(1-\chi)+a_\infty^2\chi](L\chi-\chi L)S\left(\int_s^t a_\infty(\tau)^2\dd\tau\right).
\end{equation*}

Consequently, $\left(\partial_t-[a_0^2(1-\chi)+a_\infty^2\chi]L\right)Q$ is the
negative of the following operator, which will be the first term of a
recursive construction:
\begin{equation}\label{eq:defect}
V_1(t,s;a_0,a_\infty)
 :=\Big([a_0(t)^2(1-\chi)+a_\infty(t)^2\chi](L\chi-\chi L)
 -(a_\infty(t)^2-a_0(t)^2)(1-\chi)\chi L\Big)
 \big[S(\int_s^t a_\infty^2)-S(\int_s^t a_0^2)\big].
\end{equation}

For each integer $n\ge1$, define
\begin{equation}\label{eq:R}
 V_{n+1}(t,s;a_0,a_\infty)
 =\int_s^t V_1(t,r;a_0,a_\infty)V_n(r,s;a_0,a_\infty)\dd r.
\end{equation}

Lemma~\ref{lem:commutator} bounds $L\chi-\chi L$ and $(1-\chi)\chi L$
from $\Z$ to $Y$. Hence, using \eqref{eq:Dbounds},
\begin{equation*}
\norm{V_1(t,s;a_0,a_\infty)}_{\Z\to Y}
 \le C\eta,
\end{equation*}
\begin{equation}\label{eq:Ebounds}
\norm{V_1(t,s;a_0,a_\infty)}_{Y\to Y}
 \le C\eta k(t-s),
\end{equation}
for some positive constant $C>0$.

To prove convergence of \eqref{eq:R}, we need the following estimate:
split it at $t/2$, substitute $r\mapsto t-r$ in the second half, and use that $k$ is decreasing,
\begin{equation}\label{eq:convolution}
\int_0^t k(t-r)k(r)\dd r
=2\int_0^{t/2}k(t-r)k(r)\dd r
 \le2k(t/2)\int_0^{t/2}k(r)\,dr
 \le2^{2+\kappa}\|k\|_{L^1}k(t).
\end{equation}

We claim, for all $n\ge1$, that
\begin{equation*}
\norm{V_n(t,s;a_0,a_\infty)}_{\Z\to Y}
 \le(C\eta)^n\norm k_{L^1}^{n-1},
\end{equation*}
\begin{equation}\label{eq:Rbounds}
\norm{V_n(t,s;a_0,a_\infty)}_{Y\to Y}
 \le(C\eta)^n(2^{2+\kappa}\norm k_{L^1})^{n-1}k(t-s).
\end{equation}

The case $n=1$ is \eqref{eq:Ebounds}. If the first estimate holds for $n$,
then \eqref{eq:R} yields
\begin{equation*}
 \norm{V_{n+1}(t,s;a_0,a_\infty)}_{\Z\to Y}
 \le(C\eta)^{n+1}\norm k_{L^1}^{n-1}\int_s^t k(t-r)\dd r
 \le(C\eta)^{n+1}\norm k_{L^1}^n.
\end{equation*}
For the second estimate, the induction step gives
\begin{equation*}
\norm{V_{n+1}(t,s;a_0,a_\infty)}_{Y\to Y}
 \le(C\eta)^{n+1}(2^{2+\kappa}\norm k_{L^1})^{n-1}
                    \int_s^t k(t-r)k(r-s)\dd r
 \le(C\eta)^{n+1}(2^{2+\kappa}\norm k_{L^1})^n k(t-s).
\end{equation*}

This proves both estimates and the convergence of every integral defining
$V_n$. Choose $\eta$ small so that
\begin{equation}\label{eq:eta}
 C\eta\,2^{2+\kappa}\norm k_{L^1}\le\frac14.
\end{equation}
The geometric series in \eqref{eq:Rbounds} then gives
\begin{equation}\label{asd}
\sum_{n\ge1}\norm{V_n(t,s;a_0,a_\infty)}_{\Z\to Y}\le C\eta,
\end{equation}
\begin{equation}\label{eq:Rsum}
\sum_{n\ge1}\norm{V_n(t,s;a_0,a_\infty)}_{Y\to Y}\le C\eta k(t-s).
\end{equation}

\paragraph{Step 3: define $U_{a_0,a_\infty}$ and prove \eqref{eq:L1}--\eqref{eq:L2}.}
Define
\begin{equation}\label{eq:U}
U_{a_0,a_\infty}(t,s):=Q(t,s)+\sum_{n\ge1}\int_s^t Q(t,r)V_n(r,s;a_0,a_\infty)\,dr.
\end{equation}

Lemmas~\ref{lem:fixed-Z}, Lemma~\ref{lem:fixed-Y} and boundedness of
multiplication by $\chi,1-\chi$ give
\begin{equation*}
\|Q(t,s)\|_{\mathcal Z\to\mathcal Z}\lesssim 1 ,
\end{equation*}
\begin{equation*}
\|Q(t,s)\|_{Y\to\mathcal Z}\lesssim k(t-s).
\end{equation*}

For the series in \eqref{eq:U}, (\ref{asd}) and \eqref{eq:Rsum} give respectively
\begin{equation*}
 \sum_{n\ge1}\int_s^t k(t-r)
     \norm{V_n(r,s;a_0,a_\infty)}_{\Z\to Y}\dd r
 \lesssim \eta\int_s^t k(t-r)\dd r\lesssim \eta\norm k_{L^1}
\end{equation*}
and
\begin{equation*}
 \sum_{n\ge1}\int_s^t k(t-r)
     \norm{V_n(r,s;a_0,a_\infty)}_{Y\to Y}\dd r
 \lesssim \eta\int_s^t k(t-r)k(r-s)\dd r
 \lesssim \eta \norm k_{L^1}k(t-s).
\end{equation*}
This proves absolute convergence in operator norms and
\eqref{eq:L1}--\eqref{eq:L2}.

\paragraph{Step 4: verify the differential equation and initial trace.}
	We retain the following differentiation as a formal calculation.
	Lemmas~\ref{lem:appendix-finite-sums},
	\ref{lem:appendix-series-convergence}
	and~\ref{lem:appendix-U-equation} rigorously establish the
	corresponding integral equation in $L^\infty(I)$ for every
	compact interval $I\subset(0,\infty)$ and every $t\ge s$,
	with time integrals understood as Bochner integrals in that space.

Apply $\partial_t-[a_0(t)^2(1-\chi)+a_\infty(t)^2\chi]L$ to the finite
sum in \eqref{eq:U} with $1\le n\le N$. The computation in Step 2 and
\eqref{eq:R} give
\begin{equation*}
	V_n(t,s;a_0,a_\infty)-\int_s^tV_1(t,r;a_0,a_\infty)V_n(r,s;a_0,a_\infty)\,dr
	=V_n(t,s;a_0,a_\infty)-V_{n+1}(t,s;a_0,a_\infty).
\end{equation*}
Summing these identities, the result of applying the differential operator is
\begin{equation*}
	-V_1(t,s;a_0,a_\infty)+\sum_{n=1}^N
	[V_n(t,s;a_0,a_\infty)-V_{n+1}(t,s;a_0,a_\infty)]
	=-V_{N+1}(t,s;a_0,a_\infty).
\end{equation*}

For each fixed $g\in\mathcal Z\cup Y$ and $t>s$,
Lemma~\ref{lem:appendix-series-convergence} gives convergence
of the finite sums in $L^1((s,t);\mathcal Z)$ and
\begin{equation*}
	\int_s^t
	\|V_{N+1}(r,s;a_0,a_\infty)g\|_Y\,dr
	\longrightarrow0.
\end{equation*}
Passing to the limit in the finite-sum integral identity
\eqref{eq:appendix-finite-integral}, as justified in
Lemma~\ref{lem:appendix-U-equation}, proves that
\eqref{eq:U} satisfies \eqref{eq:coefficients}
in the local integral sense.

Lemma~\ref{lem:appendix-U-equation} also gives
\begin{equation*}
	\lim_{t\downarrow s}
	\|U_{a_0,a_\infty}(t,s)g-g\|_{L^\infty(I)}=0
\end{equation*}
for every compact interval $I\subset(0,\infty)$.
The short-time estimate \eqref{eq:UYY} is proved separately
in the same lemma using \eqref{eq:fixedYY}
and \eqref{eq:Rsum}.

\paragraph{Step 5: prove \eqref{eq:L3Z}--\eqref{eq:L3Y}.}
We must compare \eqref{eq:U} for $a_0,a_\infty$ and
$\widetilde a_0,\widetilde a_\infty$. First,
\begin{equation*}
\norm{[a_0^2(1-\chi)+a_\infty^2\chi]
       -[\widetilde a_0^2(1-\chi)+\widetilde a_\infty^2\chi]}_{L^\infty_\omega} \lesssim (|a_0-\widetilde a_0|+|a_\infty-\widetilde a_\infty|),
\end{equation*}
\begin{equation*}
|(a_\infty^2-a_0^2)-(\widetilde a_\infty^2-\widetilde a_0^2)|
 \lesssim (|a_0-\widetilde a_0|+|a_\infty-\widetilde a_\infty|).
\end{equation*}

Subtract the two products in \eqref{eq:defect}, using
$AB-\widetilde A\widetilde B=(A-\widetilde A)B+
\widetilde A(B-\widetilde B)$. The factors containing $a_j^2$ have the
bounds just displayed. The factors containing $S$ have the bounds
\eqref{eq:Dbounds}, \eqref{eq:timeLipZ}, and \eqref{eq:timeLipY}.
 Thus, one can derive
\begin{equation*}
\norm{V_1(t,s;a_0,a_\infty)-V_1(t,s;\widetilde a_0,\widetilde a_\infty)}_{\Z\to Y}
 \le C\sum_{j\in\{0,\infty\}}\norm{a_j-\widetilde a_j}_{L^\infty(s,t)},
\end{equation*}
\begin{equation}\label{eq:DeltaE}
\norm{V_1(t,s;a_0,a_\infty)-V_1(t,s;\widetilde a_0,\widetilde a_\infty)}_{Y\to Y}
 \le Ck(t-s)\sum_{j\in\{0,\infty\}}\norm{a_j-\widetilde a_j}_{L^\infty(s,t)}.
\end{equation}

Iterating \eqref{eq:R} writes $V_n$ as the integral of $n$
factors $V_1$. To subtract two such products use the exact identity
\begin{equation*}
 A_1\cdots A_n-B_1\cdots B_n
 =\sum_{\ell=1}^n B_1\cdots B_{\ell-1}(A_\ell-B_\ell)
                         A_{\ell+1}\cdots A_n.
\end{equation*}
In each summand, \eqref{eq:DeltaE} bounds the one difference factor and
\eqref{eq:Ebounds} bounds the other $n-1$ factors. Repeated application of
\eqref{eq:convolution} and the induction used in \eqref{eq:Rbounds} gives
\begin{equation*}
\norm{V_n(t,s;a_0,a_\infty)-V_n(t,s;\widetilde a_0,\widetilde a_\infty)}_{\Z\to Y} \le nC(C\eta\norm k_{L^1})^{n-1}
       \sum_{j\in\{0,\infty\}}\norm{a_j-\widetilde a_j}_{L^\infty(s,t)},
\end{equation*}
\begin{equation*}
\fitdisplay{\norm{V_n(t,s;a_0,a_\infty)-V_n(t,s;\widetilde a_0,\widetilde a_\infty)}_{Y\to Y} \le nC(C\eta\,2^{2+\kappa}\norm k_{L^1})^{n-1}k(t-s)
       \sum_{j\in\{0,\infty\}}\norm{a_j-\widetilde a_j}_{L^\infty(s,t)}.}
\end{equation*}

Then one can obtain
\begin{equation*}
	\|U_{a_0,a_\infty}(t,s)
	-U_{\widetilde a_0,\widetilde a_\infty}(t,s)\|_{\mathcal Z\to\mathcal Z}
	\lesssim
	\left(1+\int_s^t k(t-r)\,dr\right)
	\sum_{j\in\{0,\infty\}}
	\|a_j-\widetilde a_j\|_{L^\infty(s,t)},
\end{equation*}
\begin{equation*}
	\|U_{a_0,a_\infty}(t,s)
	-U_{\widetilde a_0,\widetilde a_\infty}(t,s)\|_{Y\to\mathcal Z}
	\lesssim
	\left(k(t-s)+\int_s^t k(t-r)k(r-s)\,dr\right)
	\sum_{j\in\{0,\infty\}}
	\|a_j-\widetilde a_j\|_{L^\infty(s,t)}.
\end{equation*}
Since $k\in L^1(0,\infty)$, \eqref{eq:convolution}
now yields \eqref{eq:L3Z}--\eqref{eq:L3Y}.

\paragraph{Step 6: uniqueness and composition.}
Fix $0\le s<t<\infty$.
Let $u$ be an integral solution of \eqref{eq:coefficients}
on $[s,t]$ in the sense of Definition \ref{definition of integral solution},
with $u(s)\in\mathcal Z$, and suppose that
\begin{equation*}
	\sup_{s\le r\le t}\|u(r)\|_{\mathcal Z}<\infty.
\end{equation*}
We formally differentiate, for $s<r<t$,
\begin{equation}\label{sd}
	\left[
	S\left(\int_r^t a_0(q)^2\,dq\right)(1-\chi)
	+
	S\left(\int_r^t a_\infty(q)^2\,dq\right)\chi
	\right]u(r).
\end{equation}
The product rule, the equation for $u$, and the commutation
identity in Lemma~\ref{lem:appendix-semigroup-commutation}
give the sum of
\begin{equation*}
	-a_0(r)^2S\left(\int_r^t a_0(q)^2\,dq\right)
	L((1-\chi)u(r))
	-a_\infty(r)^2S\left(\int_r^t a_\infty(q)^2\,dq\right)
	L(\chi u(r)),
\end{equation*}
\begin{equation*}
	S\left(\int_r^t a_0(q)^2\,dq\right)(1-\chi)
	[a_0(r)^2(1-\chi)+a_\infty(r)^2\chi]Lu(r),
\end{equation*}
and
\begin{equation*}
	S\left(\int_r^t a_\infty(q)^2\,dq\right)\chi
	[a_0(r)^2(1-\chi)+a_\infty(r)^2\chi]Lu(r).
\end{equation*}

To combine these expressions, use
\begin{equation*}
	(1-\chi)[a_0(r)^2(1-\chi)+a_\infty(r)^2\chi]
	=
	a_0(r)^2(1-\chi)
	+(a_\infty(r)^2-a_0(r)^2)(1-\chi)\chi,
\end{equation*}
\begin{equation*}
	\chi[a_0(r)^2(1-\chi)+a_\infty(r)^2\chi]
	=
	a_\infty(r)^2\chi
	-(a_\infty(r)^2-a_0(r)^2)(1-\chi)\chi,
\end{equation*}
together with the commutator identities
\begin{equation*}
	L((1-\chi)u(r))
	=(1-\chi)Lu(r)-(L\chi-\chi L)u(r),
\end{equation*}
\begin{equation*}
	L(\chi u(r))
	=\chi Lu(r)+(L\chi-\chi L)u(r).
\end{equation*}
The derivative is therefore the sum of
\begin{equation}\label{eq:backwarderror-1}
	-(a_\infty(r)^2-a_0(r)^2)
	S\left(\int_r^t a_0(q)^2\,dq\right)
	(L\chi-\chi L)u(r),
\end{equation}
\begin{equation}\label{eq:backwarderror-2}
	-a_\infty(r)^2
	\left[
	S\left(\int_r^t a_\infty(q)^2\,dq\right)
	-
	S\left(\int_r^t a_0(q)^2\,dq\right)
	\right](L\chi-\chi L)u(r),
\end{equation}
and
\begin{equation}\label{eq:backwarderror}
	-(a_\infty(r)^2-a_0(r)^2)
	\left[
	S\left(\int_r^t a_\infty(q)^2\,dq\right)
	-
	S\left(\int_r^t a_0(q)^2\,dq\right)
	\right](1-\chi)\chi Lu(r).
\end{equation}
By Lemma~\ref{lem:commutator}, \eqref{eq:fixedY},
and \eqref{eq:Dbounds}, the $\mathcal Z$ norms of
\eqref{eq:backwarderror-1}, \eqref{eq:backwarderror-2},
and \eqref{eq:backwarderror} are bounded, respectively, by
\begin{equation}\label{ds}
	\eta k(t-r)\|u(r)\|_{\mathcal Z},
	\qquad
	\eta k(t-r)\|u(r)\|_{\mathcal Z},
	\qquad
	\eta^2 k(t-r)\|u(r)\|_{\mathcal Z}.
\end{equation}

	To justify the above differentiation, apply
	Lemma~\ref{lem:appendix-product-rule} on
	$[s+\varepsilon,t-\varepsilon]$, where
	$0<\varepsilon<(t-s)/2$.
	For the term involving $a_0$, take
	$\theta(r)=\int_r^t a_0(q)^2\,dq$, $h(r)=(1-\chi)u(r)$,
	and
	\begin{equation*}
		v(r)=(1-\chi)
		[a_0(r)^2(1-\chi)+a_\infty(r)^2\chi]Lu(r).
	\end{equation*}
	For the term involving $a_\infty$, take
	$\theta(r)=\int_r^t a_\infty(q)^2\,dq$, $h(r)=\chi u(r)$,
	and
	\begin{equation*}
		v(r)=\chi
		[a_0(r)^2(1-\chi)+a_\infty(r)^2\chi]Lu(r).
	\end{equation*}
	The integral equation for $u$ gives the required local
	integral identity for $h$.
	The boundedness of $u$ in $\mathcal Z$ and
	\eqref{eq:fixed-L-E} give the bounds on $h$ and $v$
	required by Lemma~\ref{lem:appendix-product-rule}.
	Moreover, $\theta$ is Lipschitz and
	$\theta(r)\ge(1-\eta)^2\varepsilon>0$.
	Thus the differentiated expression is locally absolutely
	continuous in $r$, with a strong derivative in $L^\infty(I)$
	for almost every $r\in(s,t)$, for every compact
	$I\subset(0,\infty)$.
	The replacement of $LS(\theta(r))h(r)$ by
	$S(\theta(r))Lh(r)$ follows from
	Lemma~\ref{lem:appendix-semigroup-commutation}.

Next, we identify both endpoint limits of (\ref{sd}).
At $r=t-\varepsilon$, \eqref{eq:fixed-local-trace-rate},
applied to $(1-\chi)u(t-\varepsilon)$ and
$\chi u(t-\varepsilon)$, together with the local
continuity of $u$, gives the limit $u(t)$ in $L^\infty(I)$.

At $r=s+\varepsilon$, the integral equation and
\eqref{eq:fixed-L-E} give
\begin{equation*}
	\|u(s+\varepsilon)-u(s)\|_{E_{3/2,3/2}}
	\lesssim
	\varepsilon\sup_{s\le r\le t}\|u(r)\|_{\mathcal Z}.
\end{equation*}
Using the positive-time bounds and continuity in
Lemma~\ref{lem:appendix-positive-continuity}, we recover
the endpoint value
\begin{equation*}
	\left[
	S\left(\int_s^t a_0(q)^2\,dq\right)(1-\chi)
	+
	S\left(\int_s^t a_\infty(q)^2\,dq\right)\chi
	\right]u(s).
\end{equation*}

The bounds for \eqref{eq:backwarderror-1}
\eqref{eq:backwarderror}, estimated above, and
$k\in L^1(0,\infty)$ imply that their sum is
Bochner integrable in $\mathcal Z$.
Consequently, we can integrate over $[s+\varepsilon,t-\varepsilon]$ and let $\varepsilon\downarrow0$, which implies that
\begin{equation*}
	u(t)
	-
	\left[
	S\left(\int_s^t a_0(q)^2\,dq\right)(1-\chi)
	+
	S\left(\int_s^t a_\infty(q)^2\,dq\right)\chi
	\right]u(s)
\end{equation*}
equals the integral over $(s,t)$ of the sum of
\eqref{eq:backwarderror-1},
\eqref{eq:backwarderror-2}, and \eqref{eq:backwarderror}.
This equality first holds in every local $L^\infty(I)$
space. Since both sides belong to $\mathcal Z$,
it also holds in $\mathcal Z$.

Then one can obtain
\begin{equation*}
	\left\|
	u(t)
	-
	\left[
	S\left(\int_s^t a_0(q)^2\,dq\right)(1-\chi)
	+
	S\left(\int_s^t a_\infty(q)^2\,dq\right)\chi
	\right]u(s)
	\right\|_{\mathcal Z}
	\lesssim
	\eta\int_s^t k(t-r)\|u(r)\|_{\mathcal Z}\,dr.
\end{equation*}
By \eqref{eq:fixedZ} and the boundedness of multiplication
by $1-\chi$ and $\chi$ from Lemma~\ref{lem:commutator},
\begin{equation*}
	\left\|
	\left[
	S\left(\int_s^t a_0(q)^2\,dq\right)(1-\chi)
	+
	S\left(\int_s^t a_\infty(q)^2\,dq\right)\chi
	\right]u(s)
	\right\|_{\mathcal Z}
	\lesssim \|u(s)\|_{\mathcal Z}.
\end{equation*}
The triangle inequality therefore yields
\begin{equation*}
	\|u(t)\|_{\mathcal Z}
	\lesssim \|u(s)\|_{\mathcal Z}
	+\eta\int_s^t k(t-r)\|u(r)\|_{\mathcal Z}\,dr.
\end{equation*}

Apply this estimate to the
difference of two integral solutions $u,\widetilde u$
with the same initial value. Taking the supremum gives
\begin{equation*}
	\sup_{s<r<t}\|u(r)-\widetilde u(r)\|_{\mathcal Z}
	\lesssim
	\eta\|k\|_{L^1}
	\sup_{s<r<t}\|u(r)-\widetilde u(r)\|_{\mathcal Z}.
\end{equation*}
Choosing $\eta$ sufficiently small, we establish the uniqueness on $[s,t]$.

For $s<r$ and $g\in\mathcal Z\cup Y$,
\eqref{eq:L1}--\eqref{eq:L2} give
$U_{a_0,a_\infty}(r,s)g\in\mathcal Z$.
By Lemma~\ref{lem:appendix-U-equation},
$U_{a_0,a_\infty}(\cdot,r)U_{a_0,a_\infty}(r,s)g$
and $U_{a_0,a_\infty}(\cdot,s)g$ are integral solutions
on every finite interval starting at $r$, with the
same initial value $U_{a_0,a_\infty}(r,s)g$.
Both are uniformly bounded in $\mathcal Z$ on such intervals
by \eqref{eq:L1}--\eqref{eq:L2}; uniqueness gives
\begin{equation*}
	U_{a_0,a_\infty}(t,r)U_{a_0,a_\infty}(r,s)g
	=U_{a_0,a_\infty}(t,s)g,
	\qquad t\ge r.
\end{equation*}
The case $r=s$ follows from
$U_{a_0,a_\infty}(s,s)=\mathrm{Id}$.

Finally, Lemma~\ref{lem:fixed-Z} gives $S(\tau)f_M=f_M$.
By \eqref{eq:defect} and \eqref{eq:R},
$V_n(t,s;a_0,a_\infty)f_M=0$ for every $n\ge1$.
Thus \eqref{eq:U} and \eqref{eq:Q} give
\begin{equation*}
	U_{a_0,a_\infty}(t,s)f_M
	=(1-\chi)f_M+\chi f_M=f_M,
	\qquad t\ge s.
\end{equation*}
All assertions of Theorem~\ref{thm:linear} have now been proved.

\end{proof}

\begin{remark}[Motivation for the construction of $U_{a_0,a_\infty}$]
	For the equation $\partial_tu=a_j(t)^2Lu$, with
	$j\in\{0,\infty\}$, the natural solution map is the following time-substituted operator
	\begin{equation*}
		S\left(\int_s^t a_j(\tau)^2\,d\tau\right).
	\end{equation*}
	In \eqref{eq:coefficients}, the coefficient of $L$ equals
	$a_0(t)^2$ on $\omega\le1$ and $a_\infty(t)^2$ on
	$\omega\ge 2$. This suggests combining these two operators
	using $1-\chi$ and $\chi$, as in \eqref{eq:Q}.
	
	The operator $Q(t,s)$ need not solve
	\eqref{eq:coefficients}, because $L$ does not commute
	with multiplication by $\chi$, and the coefficients
	overlap when $0<\chi<1$.
	Nevertheless, \eqref{eq:defect} gives the formal identity
	\begin{equation*}
		\left(
		\partial_t-
		[a_0(t)^2(1-\chi)+a_\infty(t)^2\chi]L
		\right)Q(t,s)
		=-V_1(t,s;a_0,a_\infty).
	\end{equation*}
	The useful feature of \eqref{eq:defect} is that
	$L\chi-\chi L$ and $(1-\chi)\chi L$ act on the
	difference of the two time-substituted operators.
	Lemma~\ref{lem:commutator} maps these expressions into $Y$,
	while \eqref{eq:Dbounds} makes that difference small.
	If $a_0=a_\infty$, the difference vanishes and
	$Q(t,s)$ is already the exact solution operator.
	
	To correct the error $-V_1(t,s;a_0,a_\infty)$, we first add
	\begin{equation*}
		\int_s^t Q(t,r)V_1(r,s;a_0,a_\infty)\,dr.
	\end{equation*}
	Formally applying the differential operator from
	\eqref{eq:coefficients} gives the boundary contribution
	$V_1(t,s;a_0,a_\infty)$, which cancels the original error,
	and leaves
	\begin{equation*}
		-\int_s^t
		V_1(t,r;a_0,a_\infty)V_1(r,s;a_0,a_\infty)\,dr
		=-V_2(t,s;a_0,a_\infty).
	\end{equation*}
	This explains the recursion \eqref{eq:R}:
	adding the correction involving $V_n$ cancels the error
	$-V_n$ and leaves $-V_{n+1}$.
	Consequently, the error after $N$ corrections is
	\begin{equation*}
		-V_1(t,s;a_0,a_\infty)
		+\sum_{n=1}^N
		\bigl[
		V_n(t,s;a_0,a_\infty)
		-V_{n+1}(t,s;a_0,a_\infty)
		\bigr]
		=-V_{N+1}(t,s;a_0,a_\infty).
	\end{equation*}
	In fact this construction follows the classical parametrix method:
	one starts with an approximate solution operator and successively
	corrects its error; see
	\cite[Chapter~1]{Friedman1964}
	and
	\cite[Section~2.1]{KnopovaKulik2018}.
	In particular, \eqref{eq:R} generates a Neumann series
	for a Volterra integral equation.

\end{remark}

We next use $U_{a_0,a_\infty}$ to solve the corresponding
inhomogeneous equation by Duhamel's formula.
\begin{lemma}\label{lem:forcing}
Let $a_0,a_\infty$ satisfy Theorem~\ref{thm:linear}, $z\in\Z$, and
$g\in L^\infty((0,T);Y)$. Then
\begin{equation*}
 u(t)=U_{a_0,a_\infty}(t,0)z
             +\int_0^t U_{a_0,a_\infty}(t,s)g(s)\dd s,\qquad 0\le t\le T
\end{equation*}
is the unique bounded $\Z$-valued integral solution of
\begin{equation*}
 \partial_tu=[a_0(t)^2(1-\chi)+a_\infty(t)^2\chi]Lu+g
\end{equation*}
with local initial trace $z$, and
\begin{equation*}
 \sup_{0<t<T}\norm{u(t)}_\Z
 \lesssim \norm z_\Z+\norm k_{L^1}\norm g_{L^\infty((0,T);Y)}.
\end{equation*}
\end{lemma}
\begin{proof}
	The integral-solution property, local initial trace, and
	the corresponding estimate follow from
	Lemmas~\ref{lem:appendix-Duhamel-bound} 
	and~\ref{lem:appendix-Duhamel-equation}.
	The difference of two such solutions satisfies the homogeneous
	equation with zero initial trace, so uniqueness follows
	from Theorem~\ref{thm:linear}.
	
\end{proof}

\vspace{10pt}
\section{The nonlinear estimate}\label{nonlinear}
For $f\in\Z$ use the coefficients in \eqref{eq:endpointfunctionals}:
\begin{equation*}
 f=a_0(f)(1-\chi)f_M+a_\infty(f)\chi f_M+h,
 \qquad h\in X.
\end{equation*}
Define the nonlinear operator
\begin{equation}\label{eq:Fdefinition}
 \F(f)=\C(f)-[a_0(f)^2(1-\chi)+a_\infty(f)^2\chi]Lf.
\end{equation}
We will control its $Y$-norm in order to integrate it using \eqref{eq:L2}.
\begin{lemma}[The $Y$-estimates for $\F$]\label{lem:nonlinear}
There exists $r_0>0$ such that for $u,v\in\Z$ with
$\norm u_\Z,\norm v_\Z\le r_0$,
\begin{equation}\label{eq:N1}
\norm{\F(f_M+u)}_Y\lesssim \norm u_\Z^2,
\end{equation}
\begin{equation}\label{eq:N2}
\norm{\F(f_M+u)-\F(f_M+v)}_Y
 \lesssim (\norm u_\Z+\norm v_\Z)\norm{u-v}_\Z.
\end{equation}

\end{lemma}
\begin{proof}
	The collision terms below are absolutely convergent by
	Lemma~\ref{lem:CDG}, and $Lf$ and $Lh$ are well-defined by
	Lemma~\ref{lem:appendix-L-mapping}.
	
	For $f=f_M+u$, linearity of the coefficient functionals gives
	\begin{equation*}
		a_0(f)=1+a_0(u),\qquad
		a_\infty(f)=1+a_\infty(u).
	\end{equation*}
	Moreover, $f$ and $u$ have the same remainder $h\in X$:
	\begin{equation*}
		h=u-a_0(u)(1-\chi)f_M-a_\infty(u)\chi f_M.
	\end{equation*}
	Consequently,
	\begin{equation*}
		\norm h_X\le\norm u_\Z\le r_0<1,\qquad
		|a_\infty(f)-a_0(f)|
		=|a_\infty(u)-a_0(u)|
		\le\norm u_\Z.
	\end{equation*}
	We also have $|a_0(f)|,|a_\infty(f)|\le2$.
	
	To prove \eqref{eq:N1}, we estimate the collision terms
	separately at $0$ and $\infty$.
	One can first rewrite \eqref{eq:Fdefinition} as follows.
	\begin{equation}\label{eq:nonlinear-endpoint-split}
		\mathcal F(f)
		=
		(1-\chi)\bigl[\mathcal C(f)-a_0(f)^2Lf\bigr]
		+
		\chi\bigl[\mathcal C(f)-a_\infty(f)^2Lf\bigr].
	\end{equation}
	For either $j=0$ or $j=\infty$,
	expanding around $a_j(f)f_M$ gives the exact identity:
	\begin{equation}\label{eq:endpoint-cubic-remainder}
		\mathcal C(f)-a_j(f)^2Lf
		=
		3a_j(f)\mathcal C
		\bigl(f_M,f-a_j(f)f_M,f-a_j(f)f_M\bigr)
		+
		\mathcal C
		\bigl(f-a_j(f)f_M,f-a_j(f)f_M,f-a_j(f)f_M\bigr).
	\end{equation}
	Thus \eqref{eq:endpoint-cubic-remainder} is used with
	$j=0$ where $1-\chi$ is supported, and with $j=\infty$
	where $\chi$ is supported.
	
	To analyze the term $f-a_j(f)f_M$ in \eqref{eq:endpoint-cubic-remainder}, 
	\eqref{eq:Znorm} gives
	\begin{equation*}
		f-a_0(f)f_M
		=
		(a_\infty(f)-a_0(f))\chi f_M+h,
	\end{equation*}
	\begin{equation*}
		f-a_\infty(f)f_M
		=
		(a_0(f)-a_\infty(f))(1-\chi)f_M+h.
	\end{equation*}
	The above identities imply
	\begin{equation}\label{cv}
		\|f-a_0(f)f_M\|_{E_{\frac76-\delta,\frac76}}
		\lesssim\|u\|_{\mathcal Z},
	\end{equation}
	\begin{equation}\label{vc}
		\|f-a_\infty(f)f_M\|_{E_{\frac76,\frac76+\delta}}
		\lesssim\|u\|_{\mathcal Z}.
	\end{equation}
	The subtraction of $a_0(f)f_M$ therefore improves
	the low-frequency exponent by $\delta$, whereas
	the subtraction of $a_\infty(f)f_M$ improves the
	high-frequency exponent by $\delta$.
	
	Before applying Lemma~\ref{lem:CDG}, we check its
	exponent conditions.
	All low exponents used below belong to
	$\{\frac76-\delta,\frac76\}\subset(1,5/4)$.
	All high exponents belong to
	$\{\frac76,\frac76+\delta\}$ and exceed $1$.
	The sum of any two exponents is strictly less than $5/2$, since
	\begin{equation*}
		2\left(\frac76+\delta\right)<\frac52.
	\end{equation*}
	Thus Lemma~\ref{lem:CDG} applies to all six permutations
	in each collision term below.
	
	For $j=0$, we first estimate the two collision terms in
	\eqref{eq:endpoint-cubic-remainder}.
	Applying Lemma~\ref{lem:CDG} to all six permutations
	of $f_M,f-a_0(f)f_M,f-a_0(f)f_M$ gives the
	low-frequency exponent
	\begin{equation*}
		\max\left(
		\frac76,\frac76-\delta,
		\frac32-2\delta,
		\frac43-\delta,\frac43-2\delta
		\right)
		=\frac32-2\delta.
	\end{equation*}
	Applying the same lemma to three copies of
	$f-a_0(f)f_M$ gives the low-frequency exponent
	\begin{equation*}
		\max\left(
		\frac76-\delta,\frac32-3\delta,\frac43-2\delta
		\right)
		=\frac32-3\delta.
	\end{equation*}
	Both equalities follow from $0<\delta<1/12$.
	In both applications, all high exponents in (\ref{cv}) equal $7/6$, so the resulting high exponent is
	\begin{equation*}
		\min\left(3\cdot\frac76-2,\frac76+\frac12\right)
		=\frac32.
	\end{equation*}
	Consequently,
	\begin{equation*}
		\left\|
		\mathcal C(f_M,f-a_0(f)f_M,f-a_0(f)f_M)
		\right\|_{E_{\frac32-2\delta,\frac32}}
		\lesssim\|u\|_{\mathcal Z}^2,
	\end{equation*}
	\begin{equation*}
		\left\|
		\mathcal C(f-a_0(f)f_M,f-a_0(f)f_M,f-a_0(f)f_M)
		\right\|_{E_{\frac32-3\delta,\frac32}}
		\lesssim\|u\|_{\mathcal Z}^3.
	\end{equation*}
	Both low exponents are at most the low-frequency exponent
	$\frac32-2\delta$ defining $Y$.
	Multiplication by $1-\chi$ in
	\eqref{eq:nonlinear-endpoint-split} removes the
	high-frequency tail.
	On the remaining interval $[1,2]$, the weights are
	bounded above and below.
	Hence \eqref{eq:endpoint-cubic-remainder} gives
	\begin{equation*}
		\left\|
		(1-\chi)\bigl[\mathcal C(f)-a_0(f)^2Lf\bigr]
		\right\|_Y
		\lesssim
		|a_0(f)|\|u\|_{\mathcal Z}^2+\|u\|_{\mathcal Z}^3.
	\end{equation*}
	
	For $j=\infty$, applying Lemma~\ref{lem:CDG}
	to all six permutations of
	$f_M,f-a_\infty(f)f_M,f-a_\infty(f)f_M$
	gives the high-frequency exponent
	\begin{equation*}
		\min\left(
		\frac32+2\delta,\frac53,\frac53+\delta
		\right)
		=\frac32+2\delta.
	\end{equation*}
	Applying the same lemma to three copies of
	$f-a_\infty(f)f_M$ gives the high-frequency exponent
	\begin{equation*}
		\min\left(
		\frac32+3\delta,\frac53+\delta
		\right)
		=\frac32+3\delta.
	\end{equation*}
	Both equalities follow from $0<\delta<1/12$.
	In both applications, all low exponents in (\ref{vc}) equal $7/6$, so the resulting low exponent is
	\begin{equation*}
		\max\left(\frac76,\frac43,\frac32\right)=\frac32.
	\end{equation*}
	Consequently,
	\begin{equation*}
		\left\|
		\mathcal C(f_M,f-a_\infty(f)f_M,f-a_\infty(f)f_M)
		\right\|_{E_{\frac32,\frac32+2\delta}}
		\lesssim\|u\|_{\mathcal Z}^2,
	\end{equation*}
	\begin{equation*}
		\left\|
		\mathcal C(f-a_\infty(f)f_M,f-a_\infty(f)f_M,
		f-a_\infty(f)f_M)
		\right\|_{E_{\frac32,\frac32+3\delta}}
		\lesssim\|u\|_{\mathcal Z}^3.
	\end{equation*}
	Both high exponents are at least the high-frequency exponent
	$\frac32+2\delta$ defining $Y$.
	Multiplication by $\chi$ in
	\eqref{eq:nonlinear-endpoint-split} removes the
	low-frequency singularity.
	Hence \eqref{eq:endpoint-cubic-remainder} gives
	\begin{equation*}
		\left\|
		\chi\bigl[\mathcal C(f)-a_\infty(f)^2Lf\bigr]
		\right\|_Y
		\lesssim
		|a_\infty(f)|\|u\|_{\mathcal Z}^2+\|u\|_{\mathcal Z}^3.
	\end{equation*}
	
	Consequently, \eqref{eq:endpoint-cubic-remainder}
	and \eqref{eq:nonlinear-endpoint-split} yield
	\begin{equation*}
		\|\mathcal F(f_M+u)\|_Y
		\lesssim
		\bigl(|a_0(f_M+u)|+|a_\infty(f_M+u)|\bigr)
		\|u\|_{\mathcal Z}^2
		+\|u\|_{\mathcal Z}^3.
	\end{equation*}
	Together with $|a_0(f)|,|a_\infty(f)|\le2$ and
	$\|u\|_{\mathcal Z}\le r_0<1$, this proves \eqref{eq:N1}.
	
	To prove \eqref{eq:N2}, we use the estimates just
	established with different inputs.
	Let $\widetilde f=f_M+v$.
	For $j\in\{0,\infty\}$, linearity of $a_j$ gives
	\begin{equation*}
		(f-a_j(f)f_M)-(\widetilde f-a_j(\widetilde f)f_M)
		=
		(u-v)-a_j(u-v)f_M.
	\end{equation*}
	Applying the same decomposition from \eqref{eq:Znorm}
	to $u-v$, we obtain
	\begin{equation*}
		\left\|
		(f-a_0(f)f_M)-(\widetilde f-a_0(\widetilde f)f_M)
		\right\|_{E_{\frac76-\delta,\frac76}}
		\lesssim\|u-v\|_{\mathcal Z},
	\end{equation*}
	\begin{equation*}
		\left\|
		(f-a_\infty(f)f_M)
		-(\widetilde f-a_\infty(\widetilde f)f_M)
		\right\|_{E_{\frac76,\frac76+\delta}}
		\lesssim\|u-v\|_{\mathcal Z}.
	\end{equation*}
	Also,
	\begin{equation*}
		|a_j(f)-a_j(\widetilde f)|
		=|a_j(u-v)|
		\le\|u-v\|_{\mathcal Z},
		\qquad j\in\{0,\infty\}.
	\end{equation*}
	
	Subtract \eqref{eq:endpoint-cubic-remainder} at
	$f$ and $\widetilde f$ by replacing one factor at a time.
	Lemma~\ref{lem:CDG} can similarly give 
	\begin{equation*}
	\norm{\mathcal F(f_M+u)-\mathcal F(f_M+v)}_Y
	\lesssim
	(\norm u_\Z+\norm v_\Z)\norm{u-v}_\Z.
	\end{equation*}
	This proves \eqref{eq:N2}.
	
\end{proof}

\begin{remark}[The choice of the linear term]
	The decomposition in \eqref{eq:Fdefinition} is chosen
	to match the KZ coefficient at $0$ and $\infty$.
	It also ensures that the remaining collision terms
	contain enough factors with improved frequency weights
	to belong to $Y$.
	
	To explain the choice of coefficients, consider first
	a constant multiple $af_M$, where $a\in\mathbb R$.
	Symmetric trilinearity gives
	\begin{equation*}
		\mathcal C(af_M+h)
		=
		a^2Lh+3a\mathcal C(f_M,h,h)+\mathcal C(h,h,h),
	\end{equation*}
	because $\mathcal C(f_M)=0$ and
	$Lh=3\mathcal C(f_M,f_M,h)$.
	Thus the linearization at $af_M$ is $a^2L$.
	
	For a general $f\in\mathcal Z$, the leading KZ
	coefficients are $a_0(f)$ at zero and $a_\infty(f)$
	at infinity, by \eqref{eq:endpointfunctionals}.
	The corresponding linearizations are therefore
	$a_0(f)^2L$ and $a_\infty(f)^2L$.
	The expression
	\begin{equation*}
		[a_0(f)^2(1-\chi)+a_\infty(f)^2\chi]Lf
	\end{equation*}
	agrees with these two linearizations on $\omega\le1$
	and $\omega\ge 2$, respectively.
	
	The important point is that
	\eqref{eq:endpoint-cubic-remainder} contains at least
	two copies of $f-a_j(f)f_M$.
	These supply the two improvements by $\delta$
	required by the weights of $Y$.
	A term with only one such copy, such as
	$3\mathcal C(f_M,f_M,h)=Lh$, gives only the exponents
	$\frac32-\delta$ and $\frac32+\delta$ through
	Lemma~\ref{lem:CDG}; these do not suffice for a
	$Y$ estimate.
	Matching the coefficient $a_j(f)^2$ at each endpoint
	cancels precisely these linear contributions.
\end{remark}

\vspace{10pt}
\section{Global contraction and proof of the stability theorem}\label{sec:contraction}
\begin{proof}[Proof of Theorem~\ref{thm:nonlinear}]
	We now use the preceding lemmas to close the argument on the entire
	half-line, rather than on a finite time interval whose length is later extended.
	And we will work in $L^\infty((0,\infty);\Z)$.
	For $u(t)=u_0(t)(1-\chi)f_M+u_\infty(t)\chi f_M+h(t)$, the two
	coefficients of $f_M+u(t)$ are
	\begin{equation*}
		a_0(f_M+u(t))=1+u_0(t),\qquad
		a_\infty(f_M+u(t))=1+u_\infty(t).
	\end{equation*}
	Their boundedness and their differences satisfy
	\begin{equation}\label{eq:pathbound}
		\max_{j\in\{0,\infty\}}\norm{u_j}_{L^\infty_t}
		\le\norm u_{L^\infty_t\Z},
	\end{equation}
	\begin{equation*}
		\sum_{j\in\{0,\infty\}}\norm{u_j-v_j}_{L^\infty_t}
		\le2\norm{u-v}_{L^\infty_t\Z}.
	\end{equation*}
	
	For $u$ in the closed ball $\norm u_{L^\infty_t\Z}\le \rho$, with
	$\rho\le\min(r_0,\eta)$, define
	\begin{equation*}
		(\mathcal Tu)(t)={}U_{1+u_0,1+u_\infty}(t,0)(f_0-f_M) +\int_0^t U_{1+u_0,1+u_\infty}(t,s)\F(f_M+u(s))\dd s.
	\end{equation*}
	
	Absolute integrability follows from $k\in L^1$ and \eqref{eq:N1}.
	Let $C_0$ be a constant valid in \eqref{eq:L1}, and let $C_1$ be large
	enough for all products of constants below. From \eqref{eq:L1},
	\eqref{eq:L2}, and \eqref{eq:N1},
	\begin{equation}\label{eq:mapbound}
		\norm{\mathcal Tu}_{L^\infty_t\Z}
		\le C_0\norm{(f_0-f_M)}_\Z+C_1\norm{k}_{L^1(0,\infty)}\rho^2.
	\end{equation}
	For the difference of two maps, subtract the homogeneous terms, then
	split the integral difference into the following two terms:
	\begin{equation*}
		[U_{1+u_0,1+u_\infty}(t,s)-U_{1+v_0,1+v_\infty}(t,s)]\F(f_M+u(s)),
	\end{equation*}
	\begin{equation*}
		U_{1+v_0,1+v_\infty}(t,s)[\F(f_M+u(s))-\F(f_M+v(s))].
	\end{equation*}
	
	Equations \eqref{eq:L3Z}, \eqref{eq:L3Y}, \eqref{eq:N1},
	\eqref{eq:N2}, and \eqref{eq:pathbound} therefore give
	\begin{equation}\label{eq:mapLip}
		\norm{\mathcal Tu-\mathcal Tv}_{L^\infty_t\Z}
		\le C_1\left(\norm{(f_0-f_M)}_\Z+\norm{k}_{L^1(0,\infty)}\rho^2+\norm{k}_{L^1(0,\infty)}\rho\right)
		\norm{u-v}_{L^\infty_t\Z}.
	\end{equation}
	
	Here is an explicit order in which the smallness choices can be made.
	Choose $R>0$ so that $R\le\min(r_0,\eta,1/2)$ and
	\begin{equation*}
		C_1\norm{k}_{L^1(0,\infty)}(R+R^2)\le\tfrac14.
	\end{equation*}
	Then choose $\varepsilon>0$ so that
	\begin{equation*}
		2C_0\varepsilon\le R,\quad
		4C_1\norm{k}_{L^1(0,\infty)} C_0\varepsilon\le1,\quad
		C_1\varepsilon\le\tfrac14.
	\end{equation*}
	For $\varepsilon_0=\norm{(f_0-f_M)}_\Z\le\varepsilon$ set
	$\rho=2C_0\varepsilon_0$. Equation \eqref{eq:mapbound} maps this ball into
	itself, and \eqref{eq:mapLip} has constant at most $1/2$.
	Starting with $u^{(0)}=0$ and iterating $u^{(n+1)}=\mathcal Tu^{(n)}$,
	the differences satisfy
	$\norm{u^{(n+1)}-u^{(n)}}_{L^\infty_t\Z}\le2^{-n}
	\norm{u^{(1)}-u^{(0)}}_{L^\infty_t\Z}$.
	The sum converges in the complete space $L^\infty_t\Z$. Its limit $u$
	is a fixed point by \eqref{eq:mapLip}, and two fixed points in the
	$R$ ball coincide by the same inequality.
	
	For this fixed point, Lemma~\ref{lem:forcing}, with the local integral
	identity justified in Lemma~\ref{lem:appendix-Duhamel-equation}, applied on every finite
	interval gives
	\begin{equation*}
		\partial_tu=\big[a_0(f_M+u)^2(1-\chi)+a_\infty(f_M+u)^2\chi\big]Lu+\F(f_M+u)=\C(f_M+u),
	\end{equation*}
	since $Lf_M=0$ and \eqref{eq:Fdefinition} holds. Its local initial trace
	is $(f_0-f_M)$. Thus $f=f_M+u$ is an integral solution of \eqref{eq:integralsolution}.
	 Moreover, the uniqueness follows directly from 	Lemma~\ref{lem:forcing} and fixed point argument.

\end{proof}

\vspace{10pt}
\section{Global stability with a time-dependent external force}
\label{sec:time-dependent-forcing}

We consider
\begin{equation}\label{eq:forced-equation}
	\partial_t f=\mathcal C(f)+\phi(t),
	\qquad f(0)=f_0,
\end{equation}
where $\phi(t)=\phi(t,\cdot)$ takes values in $Y$.
The spaces $X,Y,\mathcal Z$, the function $f_M$,
and all fixed parameters are unchanged.

We first specify the time-dependent smallness condition
on $\phi$.
We then extend Lemma~\ref{lem:forcing} to this class
of external forces and apply the contraction argument
from Section~\ref{sec:contraction}.

\subsection{The norm of the external force}

Recall that
\begin{equation*}
	k(t)=
	\begin{cases}
		t^{-1+3\delta},&0<t\le1,\\
		t^{-1-\kappa},&t\ge1,
	\end{cases}
	\qquad
	\|k\|_{L^1(0,\infty)}
	=\frac1{3\delta}+\frac1\kappa.
\end{equation*}
Let $\mathcal P$ be the space of strongly measurable
functions $\phi:(0,\infty)\to Y$, for which
\begin{equation}\label{eq:force-norm}
	\|\phi\|_{\mathcal P}
	:=
	\sup_{t>0}
	\int_0^t k(t-s)\|\phi(s)\|_Y\,ds
	<\infty.
\end{equation}
This is the quantity naturally associated with
the estimate \eqref{eq:L2}.

Since $k$ is positive and decreasing,
\eqref{eq:force-norm} implies
\begin{equation}\label{eq:force-local-integrability}
	\int_0^T\|\phi(s)\|_Y\,ds
	\le
	\frac1{k(T)}
	\int_0^T k(T-s)\|\phi(s)\|_Y\,ds
	\le
	\frac{\|\phi\|_{\mathcal P}}{k(T)},
	\qquad T>0.
\end{equation}
In particular, $\mathcal P \hookrightarrow L_{loc}^1((0,\infty);Y)$. Moreover, we see
\begin{equation*}
	\|\phi\|_{\mathcal P}
	\le
	\|k\|_{L^1(0,\infty)}
	\|\phi\|_{L^\infty((0,\infty);Y)}, \qquad L^\infty((0,\infty);Y)\hookrightarrow\mathcal P.
\end{equation*}
For a time-independent force $\phi(t)=\phi_0\in Y$,
\begin{equation*}
	\|\phi\|_{\mathcal P}
	=
	\|k\|_{L^1(0,\infty)}\|\phi_0\|_Y.
\end{equation*}

An integral solution of \eqref{eq:forced-equation}
means that, for every compact interval
$I\subset(0,\infty)$,
\begin{equation}\label{eq:forced-integral}
	f(t)|_I
	=
	f_0|_I+
	\int_0^t
	\bigl(\mathcal C(f(s))+\phi(s)\bigr)|_I\,ds,
	\qquad t\ge0,
\end{equation}
where the integral is a Bochner integral in $L^\infty(I)$.

\subsection{The Duhamel formula for forces in $\mathcal P$}

The hypotheses of Lemma~\ref{lem:forcing} require
boundedness of the force on finite time intervals.
We therefore record the extension needed for
\eqref{eq:force-norm}.

\begin{lemma}\label{lem:forcing-P}
	Let $a_0,a_\infty$ satisfy the assumptions of
	Theorem~\ref{thm:linear}, let $z\in\mathcal Z$,
	and let $g\in\mathcal P$.
	Define
	\begin{equation}\label{eq:forcing-P-formula}
		w(t)
		=
		U_{a_0,a_\infty}(t,0)z
		+\int_0^t U_{a_0,a_\infty}(t,s)g(s)\,ds.
	\end{equation}
	The integral in \eqref{eq:forcing-P-formula}
	is a Bochner integral in $\mathcal Z$, and
	\begin{equation}\label{eq:forcing-P-bound}
		\sup_{t>0}\|w(t)\|_{\mathcal Z}
		\lesssim
		\|z\|_{\mathcal Z}+\|g\|_{\mathcal P}.
	\end{equation}
	For every compact interval $I\subset(0,\infty)$,
	$w|_I\in C([0,\infty);L^\infty(I))$ and
	\begin{equation}\label{eq:forcing-P-integral}
		w(t)|_I
		=
		z|_I+
		\int_0^t
		\Bigl(
		[a_0(r)^2(1-\chi)+a_\infty(r)^2\chi]Lw(r)
		+g(r)
		\Bigr)|_I\,dr,
		\qquad t\ge0.
	\end{equation}
	The integral in \eqref{eq:forcing-P-integral}
	is a Bochner integral in $L^\infty(I)$.
	Moreover, $w$ is the unique integral solution
	of \eqref{eq:forcing-P-integral} that is bounded
	in $\mathcal Z$ on every finite time interval.
\end{lemma}

\begin{proof}
	For the existence and bound (\ref{eq:forcing-P-bound}), use \eqref{eq:L2}:
	\begin{equation*}
		\int_0^t
		\|U_{a_0,a_\infty}(t,s)g(s)\|_{\mathcal Z}\,ds
		\lesssim
		\int_0^t k(t-s)\|g(s)\|_Y\,ds
		\le\|g\|_{\mathcal P}.
	\end{equation*}
	The above estimate proves Bochner integrability.
	Together with \eqref{eq:L1}, it proves
	\eqref{eq:forcing-P-bound}.

	To prove \eqref{eq:forcing-P-integral}, we use
	the argument of
	Lemma~\ref{lem:appendix-Duhamel-equation}.
	The only additional point is to verify integrability
	when $g$ is merely locally integrable in time.
	By Lemma~\ref{lem:appendix-L-mapping},
	\eqref{eq:envelope}, and \eqref{eq:L2},
	\begin{equation*}
		\|LU_{a_0,a_\infty}(r,s)g(s)\|_{L^\infty(I)}
		\lesssim_I
		k(r-s)\|g(s)\|_Y,
		\qquad 0<s<r.
	\end{equation*}
	Hence, for every $t>0$,
	\begin{equation*}
		\int_0^t\int_s^t
		\|LU_{a_0,a_\infty}(r,s)g(s)\|_{L^\infty(I)}
		\,dr\,ds
		\lesssim_I
		\|k\|_{L^1(0,\infty)}
		\int_0^t\|g(s)\|_Y\,ds
		<\infty,
	\end{equation*}
	where the last inequality follows from
	\eqref{eq:force-local-integrability}.
	
	We may therefore integrate
	\eqref{eq:appendix-U-integral} with datum $g(s)$
	over $0<s<t$ and interchange the time integrals by Fubini.
	For almost every $r$, boundedness of
	$L:\mathcal Z\to L^\infty(I)$ also gives
	\begin{equation*}
		L\left(
		\int_0^r U_{a_0,a_\infty}(r,s)g(s)\,ds
		\right)\Big|_I
		=
		\int_0^r
		\bigl(LU_{a_0,a_\infty}(r,s)g(s)\bigr)|_I\,ds.
	\end{equation*}
	Thus the integral term in
	\eqref{eq:forcing-P-formula} satisfies
	\eqref{eq:forcing-P-integral} with zero initial value.
	The homogeneous term satisfies the corresponding
	identity by Lemma~\ref{lem:appendix-U-equation}.
	Adding the identities proves
	\eqref{eq:forcing-P-integral} for every $t\ge0$.
	
	By \eqref{eq:forcing-P-bound},
	$Lw|_I$ is bounded on finite time intervals.
	By \eqref{eq:force-local-integrability},
	$g|_I$ is locally Bochner integrable.
	Consequently, \eqref{eq:forcing-P-integral}
	implies continuity and the initial trace in $L^\infty(I)$.
	
	Finally, the difference of two such integral
	solutions satisfies the homogeneous equation
	\eqref{eq:coefficients} with zero local initial trace.
	The uniqueness assertion of
	Theorem~\ref{thm:linear} proves the required uniqueness.
	
\end{proof}

\subsection{Global stability}

\begin{proof}[Proof of Theorem \ref{thm:forced-stability}]
	We follow Section~\ref{sec:contraction}.
	The new contributions are the Duhamel integral
	containing $\phi$ and its dependence on the
	coefficients $a_0,a_\infty$.
	
	For $u\in L^\infty((0,\infty);\mathcal Z)$,
	the notation $a_j(f_M+u)$ below denotes the
	time-dependent function
	$t\mapsto a_j(f_M+u(t))$, for $j\in\{0,\infty\}$.
	By \eqref{eq:endpointfunctionals},
	\begin{equation*}
		\max_{j\in\{0,\infty\}}
		\|a_j(f_M+u)-1\|_{L^\infty(0,\infty)}
		\le
		\|u\|_{L^\infty((0,\infty);\mathcal Z)},
	\end{equation*}
	and
	\begin{equation*}
		\sum_{j\in\{0,\infty\}}
		\|a_j(f_M+u)-a_j(f_M+v)\|_{L^\infty(0,\infty)}
		\le
		2\|u-v\|_{L^\infty((0,\infty);\mathcal Z)}.
	\end{equation*}
	
Let $r_0$ be the radius in Lemma~\ref{lem:nonlinear}
and $\eta$ the coefficient bound in
Theorem~\ref{thm:linear}.
Choose $R\le\min(r_0,\eta,1)$.
Then \eqref{eq:N1}--\eqref{eq:N2} apply on the ball
$\|u\|_{L^\infty((0,\infty);\mathcal Z)}\le R$,
and the coefficients $a_0(f_M+u),a_\infty(f_M+u)$
satisfy the assumptions of Theorem~\ref{thm:linear}.
We therefore define
	\begin{equation*}
		(\mathcal T_\phi u)(t)
		:=
		U_{a_0(f_M+u),a_\infty(f_M+u)}(t,0)(f_0-f_M)
	\end{equation*}
	\begin{equation}\label{eq:forced-map}
		+
		\int_0^t
		U_{a_0(f_M+u),a_\infty(f_M+u)}(t,s)
		[\mathcal F(f_M+u(s))+\phi(s)]\,ds,
		\qquad t>0.
	\end{equation}
	By \eqref{eq:N1},
	\begin{equation*}
		\|\mathcal F(f_M+u)\|_{\mathcal P}
		\lesssim
		\|k\|_{L^1(0,\infty)}R^2.
	\end{equation*}
	Thus $\mathcal F(f_M+u)+\phi$ belongs to $\mathcal P$ and
	Lemma~\ref{lem:forcing-P} therefore makes
	\eqref{eq:forced-map} well-defined.
	
	We first show that \eqref{eq:forced-map} preserves
	a sufficiently small ball.
	Estimates \eqref{eq:L1}, \eqref{eq:L2},
	\eqref{eq:N1}, and \eqref{eq:force-norm} give
	\begin{equation}\label{eq:forced-map-bound}
		\|\mathcal T_\phi u\|_{L^\infty((0,\infty);\mathcal Z)}
		\le
		C\bigl(
		\|f_0-f_M\|_{\mathcal Z}
		+\|\phi\|_{\mathcal P}
		\bigr)
		+
		C\|k\|_{L^1(0,\infty)}R^2.
	\end{equation}
	
	For the contraction estimate, the homogeneous term
	and the terms containing $\mathcal F$ are treated
	exactly as in \eqref{eq:mapLip}, using
	\eqref{eq:L3Z}, \eqref{eq:L3Y}, \eqref{eq:N1},
	and \eqref{eq:N2}.
	The additional difference is
	\begin{equation*}
		\int_0^t
		\bigl[
		U_{a_0(f_M+u),a_\infty(f_M+u)}(t,s)
		-
		U_{a_0(f_M+v),a_\infty(f_M+v)}(t,s)
		\bigr]\phi(s)\,ds.
	\end{equation*}
	By \eqref{eq:L3Y}, its $\mathcal Z$ norm is bounded by
	\begin{equation*}
		C\|u-v\|_{L^\infty((0,\infty);\mathcal Z)}
		\int_0^t k(t-s)\|\phi(s)\|_Y\,ds
		\le
		C\|\phi\|_{\mathcal P}
		\|u-v\|_{L^\infty((0,\infty);\mathcal Z)}.
	\end{equation*}
	Combining the estimates gives
	\begin{equation}\label{eq:forced-map-difference}
		\|\mathcal T_\phi u-\mathcal T_\phi v\|
		_{L^\infty((0,\infty);\mathcal Z)}
		\le
		C\Bigl(
		\|f_0-f_M\|_{\mathcal Z}
		+\|\phi\|_{\mathcal P}
		+\|k\|_{L^1(0,\infty)}(R+R^2)
		\Bigr)
		\|u-v\|_{L^\infty((0,\infty);\mathcal Z)}.
	\end{equation}
	
	Now we can first choose
	$R_{\mathrm{ext}}\le\min(r_0,\eta,1/2)$ so that
	\begin{equation*}
		C\|k\|_{L^1(0,\infty)}
		(R_{\mathrm{ext}}+R_{\mathrm{ext}}^2)
		\le\frac14.
	\end{equation*}
	Next choose $\varepsilon_{\mathrm{ext}}>0$ so that
	\begin{equation*}
		2C\varepsilon_{\mathrm{ext}}\le R_{\mathrm{ext}},
		\qquad
		C\varepsilon_{\mathrm{ext}}\le\frac14.
	\end{equation*}
	For the data in \eqref{eq:forced-smallness}, set $R=
	2C\bigl(
	\|f_0-f_M\|_{\mathcal Z}
	+\|\phi\|_{\mathcal P}
	\bigr).$ Then $R\le R_{\mathrm{ext}}$, 
	\eqref{eq:forced-map-bound} gives
	\begin{equation*}
		\|\mathcal T_\phi u\|_{L^\infty((0,\infty);\mathcal Z)}
		\le\frac R2+\frac R4\le R.
	\end{equation*}
	\eqref{eq:forced-map-difference} gives a
	contraction constant at most $1/2$.
	The contraction theorem therefore gives a fixed
	point $u$ in the ball of radius $R$.
	
	It remains to verify that the fixed point is an
	integral solution, rather than only a solution of
	\eqref{eq:forced-map}.
	Apply Lemma~\ref{lem:forcing-P} with
	\begin{equation*}
		z=f_0-f_M,\qquad
		g(t)=\mathcal F(f_M+u(t))+\phi(t),
	\end{equation*}
	and coefficients $a_0(f_M+u),a_\infty(f_M+u)$.
	Then the fixed point $u$ satisfies
	\begin{equation*}
		u(t)|_I
		=
		(f_0-f_M)|_I
		+
		\int_0^t
		\Bigl(
		[a_0(f_M+u(s))^2(1-\chi)
		+a_\infty(f_M+u(s))^2\chi]Lu(s)
		+\mathcal F(f_M+u(s))+\phi(s)
		\Bigr)|_I\,ds
	\end{equation*}
	for every compact $I\subset(0,\infty)$ and every
	$t\ge0$, with a Bochner integral in $L^\infty(I)$.
	By $Lf_M=0$ and \eqref{eq:Fdefinition}, the integrand
	equals $\mathcal C(f_M+u(s))+\phi(s)$.
	Thus $f=f_M+u$ satisfies
	\eqref{eq:forced-integral} and has the asserted continuity.
	
	Moreover, \eqref{eq:forced-map-bound} directly ensures $\sup_{t\ge0}\|f(t)-f_M\|_{\mathcal Z}
	\le R$, which proves \eqref{eq:forced-stability} with
	$C_{\mathrm{ext}}=2C$. Finally, uniqueness follows by Lemma \ref{lem:forcing-P} and fixed point argument.

\end{proof}
\begin{remark}
	The stationary solutions constructed in
	\cite[Theorem~1.1]{CDG} remain close to the mass
	KZ spectrum under a sufficiently small,
	time-independent force.
	Theorem~\ref{thm:forced-stability} gives a dynamical
	stability statement: solutions starting sufficiently
	close to $f_M$ remain uniformly close for all
	$t\ge0$, even when the external force depends on time.
	The smallness condition is measured by
	\eqref{eq:force-norm}, and
	$L^\infty((0,\infty);Y)\hookrightarrow\mathcal P$.
	
	Theorem~2.1 of \cite{EMVN} shows that the
	low-frequency KZ singularity persists on a local
	time interval: the solution has the form
	$f(t,\omega)=a(t)\omega^{-7/6}
	+O(\omega^{-7/6+\delta/2})$ as $\omega\downarrow0$.
	Theorem~\ref{thm:forced-stability} likewise preserves
	the KZ asymptotics, but gives uniform control for
	all $t\ge0$ under the stated smallness assumptions.
	More precisely, \eqref{eq:forced-stability} and
	\eqref{eq:Znorm} give
	\begin{equation*}
		f(t,\omega)
		=
		a_0(f(t))\omega^{-7/6}
		+
		O\!\left(
		\bigl(\|f_0-f_M\|_{\mathcal Z}
		+\|\phi\|_{\mathcal P}\bigr)\omega^{-7/6+\delta}
		\right),
		\qquad \omega\downarrow0,
	\end{equation*}
	uniformly for $t\ge0$.
	
	Thus the connection is the persistence of the KZ
	singularity; the additional conclusion here is
	its global, uniformly small deviation from $f_M$
	in the presence of a time-dependent force.

\end{remark}

\vspace{10pt}
\section{Convergence to a stationary solution under time-independent forcing}
\label{sec:stationary-convergence}

We now assume that the external force is independent of time:
\begin{equation}\label{eq:asymptotic-equation}
	\partial_t f=\mathcal C(f)+\phi(\omega),
	\qquad f(0)=f_0.
\end{equation}
We will prove convergence to a stationary solution in slightly weaker
weighted spaces. We will also give a necessary and sufficient condition
for convergence in the original space $\mathcal Z$.

To distinguish the weights used below, for $0<\theta<1/12$ write
\begin{equation*}
	X_\theta=E_{\frac76-\theta,\frac76+\theta},
	\qquad
	Y_\theta=E_{\frac32-2\theta,\frac32+2\theta},
\end{equation*}
and
\begin{equation*}
	\mathcal Z_\theta
	=
	\mathbb R(1-\chi)f_M
	\oplus\mathbb R\chi f_M
	\oplus X_\theta.
\end{equation*}
Thus $X=X_\delta$, $Y=Y_\delta$, and
$\mathcal Z=\mathcal Z_\delta$. By definition, we also have 
\begin{equation}\label{eq:asymptotic-embeddings}
	\|z\|_{\mathcal Z_\theta}\le\|z\|_{\mathcal Z_\delta},
	\qquad
	\|g\|_{Y_\theta}\le\|g\|_{Y_\delta},
	\qquad 0<\theta\le\delta.
\end{equation}

Fix
\begin{equation*}
	\delta<\delta^+<\min\left\{2\delta,\frac1{12}\right\},
	\qquad
	0<\kappa<
	\min\left\{\frac{3\delta}{2},3v_0-\frac{11}{2}\right\},
\end{equation*}
where $v_0>11/6$ is the constant in
Theorem~\ref{thm:EMV}.
For $\theta\in\{\delta/2,\delta\}$, set
\begin{equation*}
	k_\theta(t)=
	\begin{cases}
		t^{-1+3\theta},&0<t\le1,\\
		t^{-1-\kappa},&t\ge1.
	\end{cases}
\end{equation*}
In particular, $k_\theta\in L^1(0,\infty)$. 

We will apply Theorem~\ref{thm:linear} with both
$\theta=\delta/2$ and $\theta=\delta$.
Choose $\eta>0$ small enough that its coefficient assumptions
hold for both choices whenever
\begin{equation*}
	\max_{j\in\{0,\infty\}}
	\|a_j-1\|_{L^\infty(0,\infty)}\le\eta.
\end{equation*}
The resulting estimates are
\begin{equation}\label{eq:asymptotic-U-bounds}
	\|U_{a_0,a_\infty}(t,s)z\|_{\mathcal Z_\theta}
	\lesssim\|z\|_{\mathcal Z_\theta},
	\qquad 0\le s<t,
\end{equation}
and
\begin{equation}\label{eq:asymptotic-U-smoothing}
	\|U_{a_0,a_\infty}(t,s)g\|_{\mathcal Z_\theta}
	\lesssim k_\theta(t-s)\|g\|_{Y_\theta},
	\qquad 0\le s<t.
\end{equation}
The constants are uniform over the allowed coefficients.

The notation $U_{a_0,a_\infty}$ is consistent for the two
choices of $\theta$.
Indeed, for initial data in $\mathcal Z_\delta$, both
constructions solve the same integral equation, have the same local initial trace and are bounded in $\mathcal Z_{\delta/2}$ on every finite time interval. Then uniqueness in Theorem~\ref{thm:linear}, applied with $\theta=\delta/2$, identifies them.

Before constructing the stationary solution $f_\ast$ with a prescribed
high-frequency coefficient, we first recall the following result.
\begin{lemma}[{\cite[Theorem~1.1]{CDG}}]\label{CDG lemma}
	\label{lem:CDG-stationary}
	Fix $0<\delta^+<1/12$.
	There exist $\varepsilon_{\mathrm{stat}}>0$ and $C>0$
	such that the following holds.
	Let $b\in[1/2,3/2]$ and 
	$\phi\in E_{\frac32-\delta^+,\frac32+\delta^+}$
	 with
	\begin{equation*}
		\|\phi\|_{E_{\frac32-\delta^+,\frac32+\delta^+}}
		\le\varepsilon_{\mathrm{stat}}.
	\end{equation*}
	Then there exists a nonnegative stationary solution $f_*$
	satisfying
	\begin{equation*}
		\mathcal C(f_*)(\omega)+\phi(\omega)=0
		\qquad\text{for almost every }\omega>0,
	\end{equation*}
	where the collision integral is absolutely convergent.
	
	There is a positive coefficient $a_0(f_*)$ such that
	\begin{equation*}
		\operatorname*{ess\,sup}_{0<\omega\le1}
		\omega^{7/6-\delta^+}
		\left|f_*(\omega)-a_0(f_*)\omega^{-7/6}\right|
		\le
		C\|\phi\|_{E_{\frac32-\delta^+,\frac32+\delta^+}},
	\end{equation*}
	and
	\begin{equation*}
		\operatorname*{ess\,sup}_{\omega\ge1}
		\omega^{7/6+\delta^+}
		\left|f_*(\omega)-b\omega^{-7/6}\right|
		\le
		C\|\phi\|_{E_{\frac32-\delta^+,\frac32+\delta^+}}.
	\end{equation*}
	In particular, $a_\infty(f_*)=b$.
	
	Let $j_M^*>0$ denote the mass-flux constant associated
	with $f_M(\omega)=\omega^{-7/6}$, defined in
	\cite[equation~(4.11)]{CDG}.
	The two coefficients satisfy
	\begin{equation*}
		j_M^*a_0(f_*)^3
		=
		j_M^*b^3+
		\int_0^\infty\omega^{1/2}\phi(\omega)\,d\omega,
	\end{equation*}
	and consequently
	\begin{equation*}
		|a_0(f_*)-b|
		\le
		C\|\phi\|_{E_{\frac32-\delta^+,\frac32+\delta^+}},
	\quad 
		\|f_*-bf_M\|_{\mathcal Z_{\delta^+}}
		\le
		C\|\phi\|_{E_{\frac32-\delta^+,\frac32+\delta^+}}.
	\end{equation*}
	
\end{lemma}  

\begin{lemma}[The stationary solution and its linear term]
	\label{lem:asymptotic-stationary}
	For $b>0$ with $|b-1|+\|\phi\|_{Y_\delta}$ sufficiently small,
	there exists a stationary solution $f_*$ satisfying
	\begin{equation}\label{eq:asymptotic-stationary-equation}
		\mathcal C(f_*)+\phi=0,
		\qquad a_\infty(f_*)=b,
	\end{equation}
	and
	\begin{equation}\label{eq:asymptotic-stationary-bound}
		\|f_*-b f_M\|_{\mathcal Z_{\delta^+}}
		\lesssim\|\phi\|_{Y_\delta}.
	\end{equation}
	Moreover,
	\begin{equation}\label{eq:asymptotic-stationary-L}
		\|Lf_*\|_{Y_\delta}\lesssim\|\phi\|_{Y_\delta}.
	\end{equation}
	In particular,
	\begin{equation}\label{eq:asymptotic-stationary-tail}
		\lim_{R\to\infty}
		\operatorname*{ess\,sup}_{\omega\ge R}
		\omega^{7/6+\delta}|f_*(\omega)-b f_M(\omega)|=0.
	\end{equation}
\end{lemma}

\begin{proof}
	For \eqref{eq:asymptotic-stationary-equation}
	and \eqref{eq:asymptotic-stationary-bound}, we apply
	Lemma~\ref{CDG lemma} with prescribed high-frequency
	coefficient $b$.
	Since $\delta^+<2\delta$,
	\begin{equation*}
		\|\phi\|_{E_{\frac32-\delta^+,\frac32+\delta^+}}
		\le\|\phi\|_{Y_\delta}.
	\end{equation*}
	Thus, for sufficiently small
	$|b-1|+\|\phi\|_{Y_\delta}$, that lemma gives a stationary solution $f_*$ satisfying
	\begin{equation*}
		\mathcal C(f_*)+\phi=0,
		\qquad a_\infty(f_*)=b,
	\end{equation*}
	and
	\begin{equation*}
		\|f_*-bf_M\|_{\mathcal Z_{\delta^+}}
		\lesssim
		\|\phi\|_{E_{\frac32-\delta^+,\frac32+\delta^+}}
		\lesssim\|\phi\|_{Y_\delta}.
	\end{equation*}
	This proves \eqref{eq:asymptotic-stationary-equation}
	and \eqref{eq:asymptotic-stationary-bound}.
	
	We next use the stationary equation to prove
	\eqref{eq:asymptotic-stationary-L}.
	Since $\delta^+>\delta$, the embedding
	$\mathcal Z_{\delta^+}\subset\mathcal Z_\delta$ gives
	\begin{equation*}
		\|f_*-f_M\|_{\mathcal Z_\delta}
		\le
		\|f_*-bf_M\|_{\mathcal Z_\delta}
		+2|b-1|
		\lesssim
		\|\phi\|_{Y_\delta}+|b-1|.
	\end{equation*}
	Consequently, both $f_*$ and $bf_M$ lie in the
	neighborhood where \eqref{eq:N2} applies.
	Moreover, $\mathcal F(bf_M)=0$, since
	$\mathcal C(bf_M)=0$ and $Lf_M=0$.
	Using \eqref{eq:N2}, we obtain
	\begin{equation*}
		\|\mathcal F(f_*)\|_{Y_\delta}
		\lesssim
		\bigl(
		\|f_*-f_M\|_{\mathcal Z_\delta}
		+2|b-1|
		\bigr)
		\|f_*-bf_M\|_{\mathcal Z_\delta}
		\lesssim
		\bigl(\|\phi\|_{Y_\delta}+|b-1|\bigr)
		\|\phi\|_{Y_\delta}.
	\end{equation*}
	
	By \eqref{eq:Fdefinition} and
	\eqref{eq:asymptotic-stationary-equation},
	\begin{equation*}
		\bigl[a_0(f_*)^2(1-\chi)+b^2\chi\bigr]Lf_*
		=
		-\phi-\mathcal F(f_*).
	\end{equation*}
	This equality holds in $L^\infty(I)$ for every
	compact interval $I\subset(0,\infty)$;
	$Lf_*$ is defined by the absolutely convergent
	integrals in
	Lemma~\ref{lem:appendix-L-mapping}.
	Also,
	\begin{equation*}
		|a_0(f_*)-1|
		\le
		|a_0(f_*)-b|+|b-1|
		\le
		\|f_*-bf_M\|_{\mathcal Z_{\delta^+}}+|b-1|.
	\end{equation*}
	After decreasing the smallness threshold, we have
	$a_0(f_*),b\ge1/2$.
	Since $0\le\chi\le1$,
	\begin{equation*}
		a_0(f_*)^2(1-\chi(\omega))+b^2\chi(\omega)
		\ge\frac14,
		\qquad \omega>0.
	\end{equation*}
	Division by this function therefore gives
	\begin{equation*}
		\|Lf_*\|_{Y_\delta}
		\le
		4\bigl(
		\|\phi\|_{Y_\delta}
		+\|\mathcal F(f_*)\|_{Y_\delta}
		\bigr)
		\lesssim\|\phi\|_{Y_\delta}.
	\end{equation*}
	This proves \eqref{eq:asymptotic-stationary-L}.
	For \eqref{eq:asymptotic-stationary-tail}, it directly follows from the definition and $\delta^+>\delta$.
	
\end{proof}

We next show that the high-frequency coefficient of the evolving
solution is fixed by its initial datum.
The same calculation identifies the obstruction to convergence
in $\mathcal Z_\delta$.

\begin{lemma}[High-frequency behavior]
	\label{lem:asymptotic-high-frequency}
Let $\phi\in Y_\delta$, and let $f$ be an integral solution of
\eqref{eq:asymptotic-equation} on $[0,T]$ satisfying
\begin{equation*}
	\sup_{0\le t\le T}
	\|f(t)\|_{\mathcal Z_\delta}<\infty.
\end{equation*}
	Then for every $0\le t\le T$,
	\begin{equation}\label{eq:asymptotic-high-increment}
		\operatorname*{ess\,sup}_{\omega\ge1}
		\omega^{3/2}|f(t,\omega)-f_0(\omega)|
		\lesssim
		t\left(
		\sup_{0\le s\le T}\|f(s)\|_{\mathcal Z_\delta}^{\,3}
		+\|\phi\|_{Y_\delta}
		\right).
	\end{equation}
	Consequently,
	\begin{equation*}
		a_\infty(f(t))=a_\infty(f_0),
		\qquad 0\le t\le T,
	\end{equation*}
	and, for each $0\le t\le T$,
	\begin{equation}\label{eq:asymptotic-high-persistence}
		\lim_{R\to\infty}
		\operatorname*{ess\,sup}_{\omega\ge R}
		\omega^{7/6+\delta}|f(t,\omega)-f_0(\omega)|=0.
	\end{equation}
\end{lemma}

\begin{proof}
	For \eqref{eq:asymptotic-high-increment}, we apply
	Lemma~\ref{lem:CDG} and obtain
	\begin{equation*}
		\|\mathcal C(f(s))\|_{E_{\frac32,\frac32}}
		\lesssim
		\|f(s)\|_{E_{\frac76,\frac76}}^3
		\le\|f(s)\|_{\mathcal Z_\delta}^3.
	\end{equation*}
	Also,
	\begin{equation*}
		\operatorname*{ess\,sup}_{\omega\ge1}
		\omega^{3/2}|\phi(\omega)|
		\le\|\phi\|_{Y_\delta}.
	\end{equation*}
	
	For each integer $n\ge2$, the integral-solution identity holds
	in $L^\infty([1,n])$:
	\begin{equation*}
		f(t)|_{[1,n]}-f_0|_{[1,n]}
		=
		\int_0^t
		\bigl(\mathcal C(f(s))+\phi\bigr)|_{[1,n]}\,ds.
	\end{equation*}
	Multiplication by $\omega^{3/2}$ is bounded on this local space.
	Taking its $L^\infty$ norm, using the preceding two estimates,
	and then letting $n\to\infty$ proves
	\eqref{eq:asymptotic-high-increment}.
	
\end{proof}

For the convergence proof, we now need decay of
$U_{a_0,a_\infty}(t,0)z$ when $a_\infty(z)=0$.
The next lemma obtains it directly from
\eqref{eq:asymptotic-U-bounds} and
\eqref{eq:asymptotic-U-smoothing}, by truncating the
high-frequency part of $z$. 

The decay exponent used below is defined as
\begin{equation}\label{eq:asymptotic-beta}
	\beta:=\frac{3\delta(1+\kappa)}{2+3\delta}.
\end{equation}
Our choice of $\kappa$ implies $0<\beta<\frac{3\delta}{2}<1$.

\begin{lemma}[Decay for data with zero high-frequency coefficient]
	\label{lem:asymptotic-linear-decay}
	Suppose that
	\begin{equation*}
		\max_{j\in\{0,\infty\}}
		\|a_j-1\|_{L^\infty(0,\infty)}\le\eta.
	\end{equation*}
	If $z\in\mathcal Z_\delta$ and $a_\infty(z)=0$, then
	\begin{equation}\label{eq:asymptotic-linear-rate}
		\|U_{a_0,a_\infty}(t,0)z\|_{\mathcal Z_{\delta/2}}
		\lesssim
		(1+t)^{-\beta}\|z\|_{\mathcal Z_\delta},
		\qquad t\ge0.
	\end{equation}
	If, in addition,
	\begin{equation}\label{eq:asymptotic-linear-tail}
		\lim_{R\to\infty}
		\operatorname*{ess\,sup}_{\omega\ge R}
		\omega^{7/6+\delta}|z(\omega)|=0,
	\end{equation}
	then
	\begin{equation}\label{eq:asymptotic-linear-strong}
		\lim_{t\to\infty}
		\|U_{a_0,a_\infty}(t,0)z\|_{\mathcal Z_\delta}=0.
	\end{equation}
\end{lemma}

\begin{proof}
	For \eqref{eq:asymptotic-linear-rate}, write
	\begin{equation*}
		z=a_0(z)(1-\chi)f_M+h,
		\qquad h\in X_\delta.
	\end{equation*}
	For $R\ge 2$, define
	\begin{equation*}
		z_R=\mathbf1_{\{\omega\le R\}}z
		=a_0(z)(1-\chi)f_M+\mathbf1_{\{\omega\le R\}}h.
	\end{equation*}
	The difference $z-z_R$ has both endpoint coefficients equal
	to zero. Hence
	\begin{equation}\label{eq:asymptotic-truncation-error}
		\|z-z_R\|_{\mathcal Z_{\delta/2}}
		=
		\operatorname*{ess\,sup}_{\omega>R}
		\omega^{7/6+\delta/2}|h(\omega)|
		\le R^{-\delta/2}\|z\|_{\mathcal Z_\delta}.
	\end{equation}
	
	We claim that
	\begin{equation}\label{eq:asymptotic-truncation-Y}
		\|z_R\|_{Y_{\delta/2}}
		\lesssim R^{1/3}\|z\|_{\mathcal Z_\delta}.
	\end{equation}
	For $0<\omega\le1$, the pointwise bound
	$|z(\omega)|\le\|z\|_{\mathcal Z_\delta}\omega^{-7/6}$
	gives
	\begin{equation*}
		\omega^{3/2-\delta}|z_R(\omega)|
		\le
		\|z\|_{\mathcal Z_\delta}\omega^{1/3-\delta}
		\le\|z\|_{\mathcal Z_\delta}.
	\end{equation*}
	For $2\le\omega\le R$, we have $z_R=h$, so
	\begin{equation*}
		\omega^{3/2+\delta}|z_R(\omega)|
		\le\|h\|_{X_\delta}\omega^{1/3}
		\le R^{1/3}\|z\|_{\mathcal Z_\delta}.
	\end{equation*}
	On $[1,2]$, all the weights are bounded above and below,
	and $z_R$ vanishes for $\omega>R$.
	This proves \eqref{eq:asymptotic-truncation-Y}.
	
	Apply \eqref{eq:asymptotic-U-bounds} to $z-z_R$ and
	\eqref{eq:asymptotic-U-smoothing} to $z_R$.
	From (\ref{eq:asymptotic-truncation-error}) and (\ref{eq:asymptotic-truncation-Y}), one can deduce
	\begin{equation}\label{eq:asymptotic-linear-truncation}
		\|U_{a_0,a_\infty}(t,0)z\|_{\mathcal Z_{\delta/2}}
		\lesssim
		\left(R^{-\delta/2}+k_{\delta/2}(t)R^{1/3}\right)
		\|z\|_{\mathcal Z_\delta}.
	\end{equation}
	For $t\ge1$, choose
	\begin{equation*}
		R\sim\,t^{\frac{1+\kappa}{1/3+\delta/2}}.
	\end{equation*}
	By \eqref{eq:asymptotic-beta},
	\begin{equation*}
		R^{-\delta/2}\lesssim t^{-\beta},
		\qquad
		t^{-1-\kappa}R^{1/3}\lesssim t^{-\beta}.
	\end{equation*}
	Thus \eqref{eq:asymptotic-linear-truncation} proves
	\eqref{eq:asymptotic-linear-rate} for $t\ge1$.
	For $0\le t\le1$, use
	\eqref{eq:asymptotic-U-bounds} and
	\eqref{eq:asymptotic-embeddings}.
	
	For \eqref{eq:asymptotic-linear-strong}, we apply \eqref{eq:asymptotic-U-bounds}
	\eqref{eq:asymptotic-U-smoothing} again, and derive
	\begin{equation*}
		\|U_{a_0,a_\infty}(t,0)z\|_{\mathcal Z_\delta}
		\lesssim
		\|z-z_R\|_{\mathcal Z_\delta}
		+k_\delta(t)\|z_R\|_{Y_\delta}.
	\end{equation*}
	The assumption (\ref{eq:asymptotic-linear-tail}) ensures
	\begin{equation*}
		\|z-z_R\|_{\mathcal Z_\delta}
		=
		\operatorname*{ess\,sup}_{\omega>R}
		\omega^{7/6+\delta}|z(\omega)|
		\longrightarrow0.
	\end{equation*}
	For every $R\ge 2$, we also have the following estimate
\begin{equation*}
	\|z_R\|_{Y_\delta}
	\lesssim
	R^{1/3+\delta}\|z\|_{\mathcal Z_\delta}.
\end{equation*}
This proves \eqref{eq:asymptotic-linear-strong}.

\end{proof}

Finally, we establish convergence to a stationary solution
under small perturbations, addressing the question raised in \cite{CDG}.

\begin{proof}[Proof of Theorem \ref{thm:stationary-convergence}]
	Write
	\begin{equation*}
		\epsilon
		=\|f_0-f_M\|_{\mathcal Z_\delta}+\|\phi\|_{Y_\delta}\le\varepsilon_*.
	\end{equation*}
	For a time-independent force, the norm in
	\eqref{eq:force-norm} satisfies
	\begin{equation*}
		\|\phi\|_{\mathcal P}
		=\|k_\delta\|_{L^1(0,\infty)}\|\phi\|_{Y_\delta},
	\end{equation*}
	Thus Theorem~\ref{thm:forced-stability} applies when
	$\varepsilon_*$ is sufficiently small.
	
	By that theorem and
	Lemma~\ref{lem:asymptotic-stationary},
	\begin{equation}\label{eq:asymptotic-small-neighborhood}
		\sup_{t\ge0}\|f(t)-f_M\|_{\mathcal Z_\delta}
		+\|f_*-f_M\|_{\mathcal Z_\delta}
		\lesssim\epsilon.
	\end{equation}
	Here $|b-1|\le\|f_0-f_M\|_{\mathcal Z_\delta}$.
	Choose $\varepsilon_*$ small enough that the coefficient
	assumptions for \eqref{eq:asymptotic-U-bounds} and
	\eqref{eq:asymptotic-U-smoothing}, and the nonlinear
	estimates with parameters $\delta$ and $\delta/2$, all apply.
	
	\paragraph{The equation for $f-f_*$.}
	To prove \eqref{eq:asymptotic-nonlinear-rate}, we first set the difference:
	\begin{equation*}
		w(t)=f(t)-f_*.
	\end{equation*}
	Lemma~\ref{lem:asymptotic-high-frequency} directly gives 
	\begin{equation*}
		a_\infty(f(t))=b=a_\infty(f_*),
		\qquad a_\infty(w(t))=0,
		\qquad t\ge0.
	\end{equation*}
	Subtract the stationary equation from the local integral
	equation for $f$, and use \eqref{eq:Fdefinition}.
	For every compact interval $I\subset(0,\infty)$,
	\begin{equation*}
		w(t)|_I
		=
		w(0)|_I+
		\int_0^t
		\bigl(
		[a_0(f(s))^2(1-\chi)+b^2\chi]Lw(s)+G(s)
		\bigr)|_I\,ds,
	\end{equation*}
	where 
	\begin{equation}\label{eq:asymptotic-difference-source}
		G(t)
		=
		\bigl[a_0(f(t))^2-a_0(f_*)^2\bigr](1-\chi)Lf_*
		+\mathcal F(f(t))-\mathcal F(f_*).
	\end{equation}
	
	We claim that, for $\theta\in\{\delta/2,\delta\}$,
	\begin{equation}\label{eq:asymptotic-source-bound}
		\|G(t)\|_{Y_\theta}
		\lesssim\epsilon\|w(t)\|_{\mathcal Z_\theta}.
	\end{equation}
	For the coefficient difference in
	\eqref{eq:asymptotic-difference-source},
	boundedness of $a_0$ and
	\eqref{eq:asymptotic-small-neighborhood} give
	\begin{equation*}
		|a_0(f(t))^2-a_0(f_*)^2|
		\le
		\bigl(|a_0(f(t))|+|a_0(f_*)|\bigr)|a_0(w(t))|
		\lesssim\|w(t)\|_{\mathcal Z_\theta}.
	\end{equation*}
	By \eqref{eq:asymptotic-stationary-L} and
	\eqref{eq:asymptotic-embeddings},
	\begin{equation*}
		\|(1-\chi)Lf_*\|_{Y_\theta}
		\lesssim\|\phi\|_{Y_\delta}\le\epsilon.
	\end{equation*}
	For the difference of $\mathcal F$ in
	\eqref{eq:asymptotic-difference-source},
	Lemma~\ref{lem:nonlinear}, with parameter $\theta$, gives
	\begin{equation*}
		\|\mathcal F(f(t))-\mathcal F(f_*)\|_{Y_\theta}
		\lesssim
		\bigl(
		\|f(t)-f_M\|_{\mathcal Z_\theta}
		+\|f_*-f_M\|_{\mathcal Z_\theta}
		\bigr)\|w(t)\|_{\mathcal Z_\theta}.
	\end{equation*}
	Together with \eqref{eq:asymptotic-small-neighborhood},
	these estimates prove \eqref{eq:asymptotic-source-bound}.
	
	Lemma~\ref{lem:forcing}, applied with parameter $\theta$,
	therefore identifies $w$ with
	\begin{equation}\label{eq:asymptotic-Duhamel}
		w(t)
		=
		U_{a_0(f),b}(t,0)w(0)
		+
		\int_0^t U_{a_0(f),b}(t,s)G(s)\,ds.
	\end{equation}
	The time integral in \eqref{eq:asymptotic-Duhamel} is a
	Bochner integral in $\mathcal Z_\theta$.
	Indeed,
	\begin{equation*}
		\int_0^t
		\|U_{a_0(f),b}(t,s)G(s)\|_{\mathcal Z_\theta}\,ds
		\lesssim
		\epsilon\|k_\theta\|_{L^1}
		\sup_{0\le s\le t}\|w(s)\|_{\mathcal Z_\theta}<\infty.
	\end{equation*}
	
	Taking norms in \eqref{eq:asymptotic-Duhamel} gives
	\begin{equation}\label{eq:asymptotic-Volterra}
		\|w(t)\|_{\mathcal Z_\theta}
		\lesssim
		\|U_{a_0(f),b}(t,0)w(0)\|_{\mathcal Z_\theta}
		+
		\epsilon\int_0^t
		k_\theta(t-s)\|w(s)\|_{\mathcal Z_\theta}\,ds.
	\end{equation}
	
	\paragraph{A uniform bound for $f-f_*$ in $\mathcal Z_\delta$.}
	This bound will be used when interpolating the decay estimate.
	With $\theta=\delta$, apply
	\eqref{eq:asymptotic-U-bounds} in
	\eqref{eq:asymptotic-Volterra} and take the supremum over time:
	\begin{equation*}
		\sup_{t\ge0}\|w(t)\|_{\mathcal Z_\delta}
		\lesssim
		\|w(0)\|_{\mathcal Z_\delta}
		+
		\epsilon\|k_\delta\|_{L^1}
		\sup_{t\ge0}\|w(t)\|_{\mathcal Z_\delta}.
	\end{equation*}
	Choosing $\varepsilon_*$ sufficiently small gives
	\begin{equation}\label{eq:asymptotic-uniform-difference}
		\sup_{t\ge0}\|w(t)\|_{\mathcal Z_\delta}
		\lesssim\|w(0)\|_{\mathcal Z_\delta}.
	\end{equation}
	
	\paragraph{Decay in $\mathcal Z_{\delta/2}$.}
	For \eqref{eq:asymptotic-nonlinear-rate} with
	$\theta=\delta/2$, Lemma~\ref{lem:asymptotic-linear-decay}
	applies because $a_\infty(w(0))=0$.
	Thus \eqref{eq:asymptotic-Volterra} gives
	\begin{equation}\label{eq:asymptotic-decay-inequality}
		\|w(t)\|_{\mathcal Z_{\delta/2}}
		\lesssim
		(1+t)^{-\beta}\|w(0)\|_{\mathcal Z_\delta}
		+
		\epsilon\int_0^t
		k_{\delta/2}(t-s)\|w(s)\|_{\mathcal Z_{\delta/2}}\,ds.
	\end{equation}
	
	We claim that
	\begin{equation}\label{eq:asymptotic-weighted-convolution}
		\int_0^t
		k_{\delta/2}(t-s)(1+s)^{-\beta}\,ds
		\lesssim(1+t)^{-\beta},
		\qquad t\ge0.
	\end{equation}
	For $t\ge2$, on $0\le s\le t/2$ we have $t-s\ge1$,
	and therefore
	\begin{equation*}
		\int_0^{t/2}
		k_{\delta/2}(t-s)(1+s)^{-\beta}\,ds
		\lesssim
		t^{-1-\kappa}\int_0^{t/2}(1+s)^{-\beta}\,ds
		\lesssim t^{-\kappa-\beta}.
	\end{equation*}
	The last inequality uses $\beta<1$.
	On $t/2\le s\le t$,
	\begin{equation*}
		\int_{t/2}^t
		k_{\delta/2}(t-s)(1+s)^{-\beta}\,ds
		\lesssim
		(1+t)^{-\beta}\|k_{\delta/2}\|_{L^1}.
	\end{equation*}
	For $0\le t\le2$, the left-hand side of
	\eqref{eq:asymptotic-weighted-convolution} is at most
	$\|k_{\delta/2}\|_{L^1}$.
	This proves \eqref{eq:asymptotic-weighted-convolution}.
	
	Fix $T>0$. Since $w$ is uniformly bounded in
	$\mathcal Z_\delta$, we have
	\begin{equation*}
		\sup_{0\le r\le T}
		(1+r)^\beta\|w(r)\|_{\mathcal Z_{\delta/2}}
		\le
		(1+T)^\beta
		\sup_{r\ge0}\|w(r)\|_{\mathcal Z_\delta}
		<\infty.
	\end{equation*}
	For $0\le s\le T$,
	\begin{equation*}
		\|w(s)\|_{\mathcal Z_{\delta/2}}
		\le
		(1+s)^{-\beta}
		\sup_{0\le r\le T}
		(1+r)^\beta\|w(r)\|_{\mathcal Z_{\delta/2}}.
	\end{equation*}
	Insert this inequality into
	\eqref{eq:asymptotic-decay-inequality}, use
	\eqref{eq:asymptotic-weighted-convolution}, and take the supremum over $0\le t\le T$:
	\begin{equation*}
		\sup_{0\le t\le T }
		(1+t)^\beta\|w(t)\|_{\mathcal Z_{\delta/2}}
		\lesssim
		\|w(0)\|_{\mathcal Z_\delta}
		+
		\epsilon
		\sup_{0\le t\le T}
		(1+t)^\beta\|w(t)\|_{\mathcal Z_{\delta/2}}.
	\end{equation*}
	The supremum is finite for each $T$.
	Choosing $\epsilon$ sufficiently small and
	absorbing the last term give a bound independent of $T$.
	Letting $T\to\infty$, we obtain
	\begin{equation}\label{eq:asymptotic-half-weight-rate}
		\|w(t)\|_{\mathcal Z_{\delta/2}}
		\lesssim
		(1+t)^{-\beta}\|w(0)\|_{\mathcal Z_\delta},
		\qquad t\ge0.
	\end{equation}
	
	\paragraph{Decay for every $0<\theta<\delta$.}
	For $0<\theta\le\delta/2$,
	\eqref{eq:asymptotic-nonlinear-rate} directly follows from
	\eqref{eq:asymptotic-half-weight-rate}.
	
	For $\delta/2<\theta<\delta$, write
	\begin{equation*}
		w(t)=a_0(w(t))(1-\chi)f_M+h(t),
		\qquad h(t)\in X_\delta.
	\end{equation*}
	Applying interpolation, one can obtain
	\begin{equation*}
		\|h(t)\|_{X_\theta}
		\le
		\|h(t)\|_{X_{\delta/2}}^{\,2(\delta-\theta)/\delta}
		\|h(t)\|_{X_\delta}^{\,2\theta/\delta-1}.
	\end{equation*}
	Using \eqref{eq:asymptotic-uniform-difference} and
	\eqref{eq:asymptotic-half-weight-rate}, we derive
	\begin{equation*}
		\|h(t)\|_{X_\theta}
		\lesssim
		(1+t)^{-2\beta(\delta-\theta)/\delta}
		\|w(0)\|_{\mathcal Z_\delta}.
	\end{equation*}
	Also,
	\begin{equation*}
		|a_0(w(t))|
		\le\|w(t)\|_{\mathcal Z_{\delta/2}}
		\lesssim
		(1+t)^{-\beta}\|w(0)\|_{\mathcal Z_\delta}.
	\end{equation*}
	Adding these estimates proves
	\eqref{eq:asymptotic-nonlinear-rate}.
	Only the estimates with parameters $\delta$ and $\delta/2$
	were used, so no further smallness condition depending on
	$\theta$ is required.
	
	\paragraph{Necessity of \eqref{eq:asymptotic-initial-tail}.}
	We now determine when convergence holds in the original space $\mathcal Z_\delta$.
	For every fixed finite $t$,
	\eqref{eq:asymptotic-high-persistence} and
	\eqref{eq:asymptotic-stationary-tail} give
	\begin{equation*}
		\lim_{R\to\infty}
		\operatorname*{ess\,sup}_{\omega\ge R}
		\omega^{7/6+\delta}
		\left|
		(f(t,\omega)-f_*(\omega))
		-(f_0(\omega)-b f_M(\omega))
		\right|=0.
	\end{equation*}
	By triangle inequality, it implies that
	\begin{equation}\label{eq:asymptotic-tail-equality}
		\lim_{R\to\infty}
		\operatorname*{ess\,sup}_{\omega\ge R}
		\omega^{7/6+\delta}|f(t,\omega)-f_*(\omega)|
		=
		\lim_{R\to\infty}
		\operatorname*{ess\,sup}_{\omega\ge R}
		\omega^{7/6+\delta}|f_0(\omega)-b f_M(\omega)|.
	\end{equation}
	Both limits exist because the essential suprema decrease
	as $R$ increases.
	
	Letting $t\to \infty$ in the left-hand side of (\ref{eq:asymptotic-tail-equality}), we see that
	\eqref{eq:asymptotic-original-convergence} forces
	\eqref{eq:asymptotic-initial-tail}.
	
	\paragraph{Sufficiency of \eqref{eq:asymptotic-initial-tail}.}
	Assume \eqref{eq:asymptotic-initial-tail}.
	Together with \eqref{eq:asymptotic-stationary-tail},
	it implies
	\begin{equation*}
		\lim_{R\to\infty}
		\operatorname*{ess\,sup}_{\omega\ge R}
		\omega^{7/6+\delta}|w(0,\omega)|=0.
	\end{equation*}
	Lemma~\ref{lem:asymptotic-linear-decay} now gives
	\begin{equation*}
		\|U_{a_0(f),b}(t,0)w(0)\|_{\mathcal Z_\delta}
		\longrightarrow0.
	\end{equation*}
	
	It remains to show that the time integral in
	\eqref{eq:asymptotic-Volterra}, with $\theta=\delta$,
	cannot prevent convergence.
	Fix $T>0$ and let $t>T$.
	For the part with $0\le s\le T$,
	\begin{equation*}
		\int_0^T
		k_\delta(t-s)\|w(s)\|_{\mathcal Z_\delta}\,ds
		\le
		\sup_{s\ge0}\|w(s)\|_{\mathcal Z_\delta}
		\int_{t-T}^t k_\delta(r)\,dr
		\longrightarrow0.
	\end{equation*}
	Here \eqref{eq:asymptotic-uniform-difference} bounds the
	supremum, and $k_\delta\in L^1(0,\infty)$.
	For the part with $T\le s\le t$,
	\begin{equation*}
		\int_T^t
		k_\delta(t-s)\|w(s)\|_{\mathcal Z_\delta}\,ds
		\le
		\|k_\delta\|_{L^1}
		\sup_{s\ge T}\|w(s)\|_{\mathcal Z_\delta}.
	\end{equation*}
	 Letting $t\to \infty$, then $T\to\infty$, we can conclude
	\begin{equation*}
		\limsup_{t\to\infty}\|w(t)\|_{\mathcal Z_\delta}
		\lesssim
		\epsilon\|k_\delta\|_{L^1}
		\limsup_{t\to\infty}\|w(t)\|_{\mathcal Z_\delta}.
	\end{equation*}
	Since $\epsilon$ is small enough, it proves \eqref{eq:asymptotic-original-convergence}.

\end{proof}
\begin{remark}
	The convergence in Theorem~\ref{thm:stationary-convergence}
	is stronger than convergence on compact frequency intervals.
	Indeed, the pointwise bound associated with
	$\mathcal Z_\theta$ gives
	\begin{equation*}
		\operatorname*{ess\,sup}_{\omega>0}
		\omega^{7/6}|f(t,\omega)-f_*(\omega)|
		\le
		\|f(t)-f_*\|_{\mathcal Z_{\delta/2}}
		\longrightarrow0.
	\end{equation*}
	Also, boundedness of the coefficient functionals gives
	\begin{equation*}
		a_0(f(t))\longrightarrow a_0(f_*),
		\qquad
		a_\infty(f(t))=a_\infty(f_*)=a_\infty(f_0).
	\end{equation*}
\end{remark}

\newpage
\appendix
\section{Justification of the linear evolution formulas}
We retain the spaces and operators in Sections~\ref{sec:setting}
and~\ref{sec:inputs}. The proofs use the kernel
estimates in Section~\ref{sec:fixed}, not the strong-derivative
interpretation of those estimates. This distinction prevents a
circular use of Lemma~\ref{lem:appendix-Z-differentiability}.

\begin{lemma}[Absolute kernel estimates]\label{lem:appendix-absolute-kernel}
For every $t>0$ and almost every $\omega>0$,
\begin{equation}\label{eq:fixed-absolute-E}
 \int_0^\infty (|F(t,\omega,r)|+|t\partial_tF(t,\omega,r)|)r^{-7/6}\,dr
 \lesssim \omega^{-7/6},
\end{equation}
and
\begin{equation}\label{eq:fixed-absolute-auxiliary}
 \int_0^\infty |F(t,\omega,r)|\min(r^{-7/6},r^{-11/6})\,dr
 \lesssim \min(\omega^{-7/6},\omega^{-11/6}).
\end{equation}
For every compact interval $I\subset(0,\infty)$ and $0<b<B<\infty$,
\begin{equation}\label{eq:fixed-absolute-local}
 \sup_{b\le t\le B}\operatorname*{ess\,sup}_{\omega\in I}
 \int_0^\infty (|F(t,\omega,r)|+|\partial_tF(t,\omega,r)|)r^{-3/2}\,dr
 \lesssim_{b,B,I} 1.
\end{equation}

\end{lemma}
\begin{proof}

	We use Theorem~\ref{thm:EMV} and
	Proposition~\ref{prop:derivative}.
	Choose, for this proof,
	\begin{equation*}
		0<\epsilon<
		\min\left\{
		\frac14,\frac{\mathfrak a}{2},
		3v_0-\frac{11}{2}
		\right\},
	\end{equation*}
	and use the exponent \(b=13/12\) in the
	estimate for \(R_2\).
	
	We first prove the estimates (\ref{eq:fixed-absolute-E}) and (\ref{eq:fixed-absolute-auxiliary}).
	The scaling of \(F\) also gives
	\begin{equation*}
		F(t,\omega,r)
		=
		\omega^{-1}
		F(t/\omega^{1/3},1,r/\omega).
	\end{equation*}
	Consequently, for either \(a=7/6\) or \(a=11/6\),
	\begin{equation*}
		\int_0^\infty |F(t,\omega,r)|r^{-a}\,dr
		=
		\omega^{-a}
		\int_0^\infty
		|F(t/\omega^{1/3},1,u)|u^{-a}\,du.
	\end{equation*}
	It is therefore enough to bound
	\begin{equation*}
		\int_0^\infty |F(t,1,r)|r^{-a}\,dr
	\end{equation*}
	uniformly for \(t>0\), separately for
	\(a=7/6\) and \(a=11/6\).
	
	For \(r>1/2\), the kernel decomposition
	has its Dirac part at \(r=1\) and
	 its three function terms are
	\begin{equation*}
		r^{1/6}\sigma(t/r^{1/3}),
	\end{equation*}
	\begin{equation*}
		r^{-1}R_1(t/r^{1/3},1/r),
	\end{equation*}
	and
	\begin{equation*}
		r^{-1}R_2(t/r^{1/3},1/r).
	\end{equation*}
	For \(0<r<1/2\), we use the far-kernel estimate (\ref{eq:EMVfar})
	directly.
	
	The Dirac contribution is \(e^{-\mathfrak a t}\),
	and is bounded by one.
	The \(R_1\) term is supported in \(2/3<r<2\).
	Its absolute value is bounded there by
	\(C|r-1|^{-5/6}\), so its integral against
	\(r^{-a}\) is uniformly bounded.
	
	For the \(R_2\) term, the time factor in its estimate
	is bounded on \((0,\infty)\).
	Indeed, its small-time exponent is \(3/4-\epsilon>0\),
	and its large-time exponent is
	\(13/4-3v_0+\epsilon<0\).
	Thus
	\begin{equation*}
		|r^{-1}R_2(t/r^{1/3},1/r)|
		\lesssim r^{1/12}.
	\end{equation*}
	For both values of \(a\),
	\begin{equation*}
		\int_{1/2}^\infty r^{1/12-a}\,dr<\infty.
	\end{equation*}
	This proves the required uniform bound for \(R_2\).
	
	For the \(\sigma\) term with \(a=7/6\), the change
	of variable \(\tau=t/r^{1/3}\) gives
	\begin{equation*}
		\int_{1/2}^\infty
		r^{-1}|\sigma(t/r^{1/3})|\,dr
		=
		3\int_0^{2^{1/3}t}\frac{|\sigma(\tau)|}{\tau}\,d\tau.
	\end{equation*}
	The integral over \((0,\infty)\) on the right is
	finite, because \(\sigma(\tau)=O(\tau^4)\) at zero
	and \(\sigma(\tau)=O(\tau^{-3v_0+5/2})\) at infinity.
	For \(a=11/6\), boundedness of \(\sigma\) directly gives
	\begin{equation*}
		\int_{1/2}^\infty
		r^{-5/3}|\sigma(t/r^{1/3})|\,dr
		\lesssim \int_{1/2}^\infty r^{-5/3}\,dr<\infty.
	\end{equation*}
	
	It remains to estimate $0<r<1/2$.
	Since $1/r>2$,  (\ref{eq:EMVfar}) and
	\begin{equation*}
		F(t,1,r)
		=
		r^{-1}F\left(\frac{t}{r^{1/3}},\frac1r,1\right)
	\end{equation*}
	give
	\begin{equation*}
		|F(t,1,r)|
		\lesssim
		r^{-1}\left(\frac1r\right)^{-11/6}
		\begin{cases}
			\left(t/r^{1/3}\right)^{-3v_0+9/2+\epsilon},
			& 0<r\le t^3,\\
			\left(t/r^{1/3}\right)^{1-\epsilon},
			& r>t^3.
		\end{cases}
	\end{equation*}
	Equivalently, for $0<r<1/2$,
	\begin{equation*}
		|F(t,1,r)|
		\lesssim
		\begin{cases}
			t^{-3v_0+9/2+\epsilon}
			r^{v_0-2/3-\epsilon/3},
			& 0<r\le t^3,\\
			t^{1-\epsilon}r^{1/2+\epsilon/3},
			& r>t^3.
		\end{cases}
	\end{equation*}
	On \(0<r<\min(t^3,1/2)\), the large-time far-kernel
	bound gives
	\begin{equation*}
		|F(t,1,r)|r^{-a}
		\lesssim
		t^{-3v_0+9/2+\epsilon}
		r^{v_0-2/3-\epsilon/3-a}.
	\end{equation*}
	For both \(a=7/6\) and \(a=11/6\), the power of
	\(r\) is greater than \(-1\).
	Integration gives
	\begin{equation*}
		\int_0^{\min(t^3,1/2)}
		|F(t,1,r)|r^{-a}\,dr
		\lesssim
		t^{-3v_0+9/2+\epsilon}
		\min(t^3,1/2)^{v_0+1/3-\epsilon/3-a}.
	\end{equation*}
	If \(t^3\leq1/2\), the right-hand side is
	\(Ct^{11/2-3a}\). This is bounded because its
	exponent is \(2\) for \(a=7/6\) and \(0\) for
	\(a=11/6\).
	If \(t^3\geq1/2\), it is bounded because
	\(t\geq2^{-1/3}\) and
	\(-3v_0+9/2+\epsilon<0\).
	
	The remaining interval \(t^3<r<1/2\) occurs only
	when \(t<2^{-1/3}\).
	There the small-time far-kernel estimate gives
	\begin{equation*}
		|F(t,1,r)|r^{-a}
		\lesssim t^{1-\epsilon}r^{1/2+\epsilon/3-a}.
	\end{equation*}
	For \(a=7/6\),
	\begin{equation*}
		t^{1-\epsilon}
		\int_{t^3}^{1/2}r^{-2/3+\epsilon/3}\,dr
		\lesssim t^{1-\epsilon}\lesssim 1.
	\end{equation*}
	For \(a=11/6\),
	\begin{equation*}
		t^{1-\epsilon}
		\int_{t^3}^{1/2}r^{-4/3+\epsilon/3}\,dr
		\lesssim
		t^{1-\epsilon}t^{-1+\epsilon}
		=1.
	\end{equation*}
	This completes the proof that
	\begin{equation*}
		\int_0^\infty |F(t,\omega,r)|r^{-a}\,dr
		\lesssim \omega^{-a},
		\qquad a\in\left\{\frac76,\frac{11}{6}\right\},
	\end{equation*}
	which implies \eqref{eq:fixed-absolute-auxiliary}.
	
	Similarly, by Proposition~\ref{prop:derivative}, we can derive
	\begin{equation*}
		\int_0^\infty
		|t\partial_tF(t,\omega,r)|r^{-7/6}\,dr
		\lesssim \omega^{-7/6}.
	\end{equation*}
	This proves \eqref{eq:fixed-absolute-E}.
	
	We now prove \eqref{eq:fixed-absolute-local}.
	Write $I=[\alpha,\beta]$, where
	$0<\alpha<\beta<\infty$, and restrict $b\le t\le B$.
	We use the scaling formula
	\begin{equation*}
		F(t,\omega,r)
		=
		r^{-1}F\left(\frac{t}{r^{1/3}},\frac{\omega}{r},1\right).
	\end{equation*}
	
	For $0<r<\min(\alpha/2,b^3,1)$, we have
	$\omega/r>2$ and $t/r^{1/3}>1$.
	The estimate from Theorem~\ref{thm:EMV} used here is
	\begin{equation*}
		|F(\tau,z,1)|
		\lesssim z^{-11/6}\tau^{-3v_0+9/2+\epsilon},
		\qquad \tau\ge1,\quad z>2.
	\end{equation*}
	Consequently,
	\begin{equation*}
		|F(t,\omega,r)|
		\lesssim
		r^{-1}\left(\frac{\omega}{r}\right)^{-11/6}
		\left(\frac{t}{r^{1/3}}\right)^{-3v_0+9/2+\epsilon}
		=
		\omega^{-11/6}t^{-3v_0+9/2+\epsilon}
		r^{v_0-2/3-\epsilon/3}.
	\end{equation*}
	Since $\omega\in I$ and $b\le t\le B$, this gives
	\begin{equation*}
		|F(t,\omega,r)|r^{-3/2}
		\lesssim_{b,B,I} r^{v_0-13/6-\epsilon/3}.
	\end{equation*}
	The exponent is greater than $-1$, so this bound
	is integrable at zero.
	
	For $r>\max(2\beta,B^3,1)$, we have
	$0<\omega/r<1/2$ and $0<t/r^{1/3}<1$.
	The Dirac and $R_1$ terms vanish, and
	\begin{equation*}
		F(t,\omega,r)
		=
		r^{-1}\left[
		\sigma\left(\frac{t}{r^{1/3}}\right)
		\left(\frac{\omega}{r}\right)^{-7/6}
		+
		R_2\left(\frac{t}{r^{1/3}},\frac{\omega}{r}\right)
		\right].
	\end{equation*}
	Consequently, 
	\begin{equation*}
		|F(t,\omega,r)|r^{-3/2}
		\lesssim
		(\omega^{-7/6}t^4r^{-7/6}
		+
		\omega^{-13/12}t^{3/4-\epsilon}
		r^{-1/6+\epsilon/3})r^{-3/2}\lesssim_{b,B,I} r^{-8/3}+r^{-5/3+\epsilon/3}.
	\end{equation*}
	Both exponents are less than $-1$, so this bound
	is integrable at infinity.
	
	For the remaining range, $r$ lies in a compact interval and the weight $r^{-3/2}$ is bounded. If $0<\omega/r<2$, the Dirac contribution to
	the weighted integral is
	\begin{equation*}
		e^{-\mathfrak a t\omega^{-1/3}}\omega^{-3/2}
		\le \alpha^{-3/2}.
	\end{equation*}
	For $R_1$, Theorem~\ref{thm:EMV} gives
	\begin{equation*}
		|R_1(\tau,z)|
		\lesssim
		e^{-(\mathfrak a-\epsilon)\tau}|z-1|^{-5/6},
		\qquad \tau>0,
	\end{equation*}
	with support in $|z-1|<1/2$. Hence
	\begin{equation*}
		r^{-1}
		\left|R_1\left(\frac{t}{r^{1/3}},\frac{\omega}{r}\right)\right|
		\lesssim
		r^{-1/6}|\omega-r|^{-5/6}
		\lesssim_{b,B,I} |\omega-r|^{-5/6}.
	\end{equation*}
	Its integral is uniformly bounded for $\omega\in I$, since $5/6<1$.
	For the $\sigma$ and $R_2$ contributions, the variables
	$t/r^{1/3}$ and $\omega/r$ remain bounded above and away
	from zero, so their bounds in Theorem~\ref{thm:EMV}
	are uniform. If $\omega/r>2$, the same conclusion
	follows from the estimate (\ref{eq:EMVfar})

	Thus
	\begin{equation*}
		\sup_{\substack{b\le t\le B\\ \omega\in I}}
		\int_0^\infty |F(t,\omega,r)|r^{-3/2}\,dr
		\lesssim_{b,B,I} 1.
	\end{equation*}
	
	Proposition~\ref{prop:derivative} gives
	the same bound with $F$ replaced by $t\partial_tF$.
	This proves \eqref{eq:fixed-absolute-local}.

\end{proof}

\begin{lemma}[Positive-time continuity of $S(t)$]\label{lem:appendix-positive-continuity}
For every $0<t_1<t_2$,
\begin{equation}\label{eq:fixed-positive-E-continuity}
 \|S(t_2)-S(t_1)\|_{E_{7/6,7/6}\to E_{7/6,7/6}}
 \lesssim \int_{t_1}^{t_2}\frac{dr}{r}.
\end{equation}
For every compact interval $I\subset(0,\infty)$ and $0<b<B<\infty$,
\begin{equation*}
 \sup_{b\le t\le B}\|S(t)\|_{E_{3/2,3/2}\to L^\infty(I)}\lesssim_{b,B,I} 1,
\end{equation*}
and
\begin{equation}\label{eq:fixed-critical-continuity}
 \|S(t_2)-S(t_1)\|_{E_{3/2,3/2}\to L^\infty(I)}
 \lesssim_{b,B,I}|t_2-t_1|,\qquad t_1,t_2\in[b,B].
\end{equation}
Suppose that $\sup_n\|g_n\|_{E_{3/2,3/2}}<\infty$ and
$g_n\to0$ in $L^\infty(J)$ for every compact interval
$J\subset(0,\infty)$. Then
\begin{equation}\label{eq:fixed-critical-local-convergence}
 \sup_{b\le t\le B}\|S(t)g_n\|_{L^\infty(I)}\longrightarrow0.
\end{equation}
\end{lemma}
\begin{proof}

	To distinguish ordinary integration from integration
	against the Dirac mass, write
	\begin{equation*}
		F(t,\omega,r)
		=
		e^{-\mathfrak a t\omega^{-1/3}}\delta_\omega(r)
		+F_{\mathrm{reg}}(t,\omega,r).
	\end{equation*}
	Here \(F_{\mathrm{reg}}\) is the function part of
	the kernel, and \(\delta_\omega\) acts on the
	integration variable \(r\).
	
	For $t>0$, \eqref{eq:fixed-absolute-E} first gives
\begin{equation}\label{eq:fixed-basic-E-bound}
 \|S(t)\|_{E_{7/6,7/6}\to E_{7/6,7/6}}\lesssim 1.
\end{equation}
Fix \(0<t_1<t_2\). We will first prove the identity
\begin{equation}\label{eq:fixed-scalar-time-integral}
 S(t_2)g(\omega)-S(t_1)g(\omega)
 =\int_{t_1}^{t_2}\int_0^\infty\partial_\tau F(\tau,\omega,r)g(r)\,dr\,d\tau
\end{equation}
for $g\in E_{7/6,7/6}\cup E_{3/2,3/2}$ and almost every $\omega>0$.
	For fixed \(\omega,r>0\), away from the diagonal
	singularity $\omega=r$, Proposition~\ref{prop:derivative}
	gives differentiability of
	\(t\mapsto F_{\mathrm{reg}}(t,\omega,r)\).
	Its derivative is bounded on \([t_1,t_2]\), so
	the fundamental theorem of calculus gives
	\begin{equation*}
		F_{\mathrm{reg}}(t_2,\omega,r)
		-
		F_{\mathrm{reg}}(t_1,\omega,r)
		=
		\int_{t_1}^{t_2}
		\partial_\tau F_{\mathrm{reg}}(\tau,\omega,r)
		\,d\tau, \quad \forall\; \omega\ne r.
	\end{equation*}

	For \(g\in E_{7/6,7/6}\),
	\eqref{eq:fixed-absolute-E} gives
	\begin{equation*}
		\int_{t_1}^{t_2}\int_0^\infty
		|\partial_\tau F_{\mathrm{reg}}(\tau,\omega,r)|
		|g(r)|\,dr\,d\tau
		\lesssim
		\|g\|_{E_{7/6,7/6}}\omega^{-7/6}
		\int_{t_1}^{t_2}\frac{d\tau}{\tau}<\infty.
	\end{equation*}
	
	Fubini's theorem now permits interchange of the
	\(r\) and \(\tau\) integrations for the function
	part of the kernel.
	For the Dirac part, the required equality is
	\begin{equation*}
		\bigl(
		e^{-\mathfrak a t_2\omega^{-1/3}}
		-
		e^{-\mathfrak a t_1\omega^{-1/3}}
		\bigr)g(\omega)
		=
		\int_{t_1}^{t_2}
		-\mathfrak a\omega^{-1/3}
		e^{-\mathfrak a\tau\omega^{-1/3}}g(\omega)
		\,d\tau.
	\end{equation*}
	Adding these two equalities proves
	\eqref{eq:fixed-scalar-time-integral}. 
	For \(g\in E_{3/2,3/2}\), we also have
	\begin{equation*}
		\sup_{\substack{t_1\leq\tau\leq t_2\\ \omega\in I}}
		\int_0^\infty
		|\partial_\tau F_{\mathrm{reg}}(\tau,\omega,r)|
		|g(r)|\,dr
		\lesssim_{t_1,t_2,I}\|g\|_{E_{3/2,3/2}}
	\end{equation*}
	from \eqref{eq:fixed-absolute-local}. Consequently, Fubini's theorem applies to the function part
	of the kernel for \(E_{3/2,3/2}\) data as well.

	Then estimate \eqref{eq:fixed-positive-E-continuity} follows
	directly from \eqref{eq:fixed-absolute-E} and \eqref{eq:fixed-scalar-time-integral}. Similarly, \eqref{eq:fixed-critical-continuity}  follows from \eqref{eq:fixed-absolute-local} and \eqref{eq:fixed-scalar-time-integral}.

For \eqref{eq:fixed-critical-local-convergence}, write
$I=[\alpha,\beta]$ and choose
\begin{equation*}
	0<\rho<\min(\alpha/2,b^3,1),
	\qquad
	R>\max(2\beta,B^3,1).
\end{equation*}
We first estimate the contributions from $(0,\rho)$
and $(R,\infty)$ uniformly in $n$, and then use
$g_n\to0$ in $L^\infty([\rho,R])$.

For $0<r<\rho$, the proof of
Lemma~\ref{lem:appendix-absolute-kernel} gives
\begin{equation*}
	|F(t,\omega,r)|
	\lesssim_{b,B,I}r^{v_0-2/3-\epsilon/3},
	\qquad b\le t\le B,\quad \omega\in I.
\end{equation*}
Then we obtain
\begin{equation*}
	\sup_{\substack{b\le t\le B\\ \omega\in I}}
	\int_0^\rho |F(t,\omega,r)|\,|g_n(r)|\,dr
	\lesssim_{b,B,I}\|g_n\|_{E_{3/2,3/2}}
	\int_0^\rho r^{v_0-13/6-\epsilon/3}\,dr\lesssim \|g_n\|_{E_{3/2,3/2}}
	\rho^{v_0-7/6-\epsilon/3}.
\end{equation*}

For $r>R$, the proof of
Lemma~\ref{lem:appendix-absolute-kernel} gives
\begin{equation*}
	|F(t,\omega,r)|
	\lesssim_{b,B,I}
	r^{-7/6}+r^{-1/6+\epsilon/3},
	\qquad b\le t\le B,\quad \omega\in I.
\end{equation*}
Consequently,
\begin{equation*}
	\sup_{\substack{b\le t\le B\\ \omega\in I}}
	\int_R^\infty |F(t,\omega,r)|\,|g_n(r)|\,dr
	\lesssim_{b,B,I}\|g_n\|_{E_{3/2,3/2}}
	\int_R^\infty
	r^{-8/3}+r^{-5/3+\epsilon/3}\,dr.
\end{equation*}
\begin{equation*}
	\lesssim\|g_n\|_{E_{3/2,3/2}}
	\left(R^{-5/3}+R^{-2/3+\epsilon/3}\right).
\end{equation*}

On $[\rho,R]$, estimate \eqref{eq:fixed-absolute-local}
gives
\begin{equation*}
	\sup_{\substack{b\le t\le B\\ \omega\in I}}
	\int_\rho^R |F(t,\omega,r)|\,dr
	\le
	R^{3/2}
	\sup_{\substack{b\le t\le B\\ \omega\in I}}
	\int_0^\infty |F(t,\omega,r)|r^{-3/2}\,dr
	\lesssim_{b,B,I}R^{3/2}.
\end{equation*}
 Therefore, for every $t\in[b,B]$,
\begin{equation*}
	\left\|
	\int_\rho^R F(t,\cdot,r)g_n(r)\,dr
	\right\|_{L^\infty(I)}
	\lesssim_{b,B,I}R^{3/2}
	\|g_n\|_{L^\infty([\rho,R])}.
\end{equation*}

Adding the estimates over the three intervals gives
\begin{equation*}
	\sup_{b\le t\le B}\|S(t)g_n\|_{L^\infty(I)}
	\lesssim_{b,B,I}R^{3/2}\|g_n\|_{L^\infty([\rho,R])}
	+
	\|g_n\|_{E_{3/2,3/2}}
	\left(
	\rho^{v_0-7/6-\epsilon/3}
	+R^{-5/3}+R^{-2/3+\epsilon/3}
	\right).
\end{equation*}
For fixed $\rho,R$, the first term tends to zero as
$n\to\infty$. Since
$\sup_n\|g_n\|_{E_{3/2,3/2}}<\infty$, it follows that
\begin{equation*}
	\limsup_{n\to\infty}
	\sup_{b\le t\le B}\|S(t)g_n\|_{L^\infty(I)}
	\lesssim_{b,B,I}\sup_n\|g_n\|_{E_{3/2,3/2}}
	\left(
	\rho^{v_0-7/6-\epsilon/3}
	+R^{-5/3}+R^{-2/3+\epsilon/3}
	\right).
\end{equation*}
The constant on the right is independent of $\rho,R$.
Letting $\rho\downarrow0$ and $R\to\infty$ makes the
right-hand side tend to zero, proving
\eqref{eq:fixed-critical-local-convergence}.

\end{proof}

To give a meaning to $LS(t)g$, we next record the boundedness of $L$.
\begin{lemma}[The kernel defining $L$]\label{lem:appendix-L-mapping}
For $g\in E_{7/6,7/6}\cup Y$, the kernel integral defining $Lg$
is absolutely convergent for almost every $\omega>0$.
Moreover,
\begin{equation}\label{eq:fixed-L-E}
 \|Lg\|_{E_{3/2,3/2}}\lesssim \|g\|_{E_{7/6,7/6}},
 \qquad g\in E_{7/6,7/6}.
\end{equation}
For every compact interval $I\subset(0,\infty)$,
\begin{equation}\label{eq:fixed-L-Y}
 \|Lg\|_{L^\infty(I)}\lesssim_I\|g\|_Y,\qquad g\in Y.
\end{equation}
\end{lemma}
\begin{proof}
For \eqref{eq:fixed-L-E}, substitute $r=\omega s$ in \eqref{eq:L}.
The integral to bound is
\begin{equation*}
 \int_0^\infty |K(s)|s^{-7/6}\,ds<\infty.
\end{equation*}
Its powers are $s^{-2/3}$ near zero and $s^{-7/3}$ at infinity;
the diagonal singularity is $|s-1|^{-5/6}$. Thus
\begin{equation*}
 |Lg(\omega)|\le
 \left(\mathfrak a+\int_0^\infty|K(s)|s^{-7/6}\,ds\right)
 \|g\|_{E_{7/6,7/6}}\omega^{-3/2}.
\end{equation*}
Similarly, simple calculations yield \eqref{eq:fixed-L-Y}. This proves both absolute convergence
and the asserted estimate.

\end{proof}

We now identify the time derivative at positive time. 
\begin{lemma}[The local differential equation at positive times]
\label{lem:appendix-local-equation}
Let $g\in E_{7/6,7/6}\cup Y$. For every compact interval
$I\subset(0,\infty)$,
\begin{equation*}
 S(\cdot)g|_I\in C^1((0,\infty);L^\infty(I)).
\end{equation*}
For every $t>0$,
\begin{equation}\label{eq:fixed-local-equation}
 \frac{d}{dt}(S(t)g|_I)=(LS(t)g)|_I
 =\left(\int_0^\infty\partial_tF(t,\cdot,r)g(r)\,dr\right)\Big|_I.
\end{equation}
The derivative is a strong derivative in $L^\infty(I)$ and
the last integral is absolutely convergent for almost every $\omega>0$.
For $g\in E_{7/6,7/6}$,
\begin{equation}\label{eq:fixed-LS-E-bound}
 \|LS(t)g\|_{E_{7/6,7/6}}\lesssim\frac 1t\|g\|_{E_{7/6,7/6}},\qquad t>0.
\end{equation}
\end{lemma}
\begin{proof}

	Fix \(g\in E_{7/6,7/6}\) and
	\(I=[\alpha,\beta]\subset(0,\infty)\). Combining (\ref{eq:fixed-L-E}) (\ref{eq:fixed-basic-E-bound}) and (\ref{eq:fixed-positive-E-continuity}), we can conclude that
	\begin{equation*}
		\|LS(t)g\|_{L^\infty(I)}
		\lesssim_I\|g\|_{E_{7/6,7/6}},
		\qquad t>0,
	\end{equation*}
	and
	\begin{equation*}
		\|LS(t_2)g-LS(t_1)g\|_{L^\infty(I)}
		\lesssim_I\|g\|_{E_{7/6,7/6}}
		\int_{t_1}^{t_2}\frac{d\tau}{\tau}.
	\end{equation*}
	Thus \(t\mapsto LS(t)g|_I\) is continuous into
	\(L^\infty(I)\) for \(t>0\).
	
	To identify this continuous function as the strong derivative
	of \(S(t)g|_I\), we claim that
	\begin{equation}\label{eq:fixed-local-integral-positive}
		S(t_2)g|_I-S(t_1)g|_I
		=
		\int_{t_1}^{t_2}LS(\tau)g|_I\,d\tau,
		\qquad 0<t_1<t_2,
	\end{equation}
	where the equality and the Bochner integral are in
	\(L^\infty(I)\). 
	
	We use \eqref{eq:fundamental-tested-identity}: for \(r>0\) and
	$\eta\in C_c^\infty(I)$,
	\begin{equation*}
		\int_I\eta(\omega)\,F(t_2,\omega,r)d\omega
		-
		\int_I\eta(\omega)\,F(t_1,\omega,r)d\omega
		=
		\int_{t_1}^{t_2}
		\langle LF(\tau,\cdot,r),\eta\rangle\,d\tau.
	\end{equation*}
	The pairing on the right means
	\begin{equation*}
		\langle LF(\tau,\cdot,r),\eta\rangle
		=
		-\mathfrak a\int_I
		\eta(\omega)\omega^{-1/3}\,F(\tau,\omega,r)d\omega
		+
		\int_0^\infty
		\left(
		\int_I\eta(\omega)\omega^{-4/3}
		K(\xi/\omega)\,d\omega
		\right)F(\tau,\xi,r)d\xi.
	\end{equation*}

	We multiply this identity by \(g(r)\) and integrate over
	\(r>0\). To justify Fubini, \eqref{eq:fixed-absolute-E} gives
	\begin{equation*}
		\mathfrak a\omega^{-1/3}
		\int_0^\infty |F(\tau,\omega,r)|\,|g(r)|\,dr
		\le C\|g\|_{E_{7/6,7/6}}\omega^{-3/2}.
	\end{equation*}
	\begin{equation*}
		\omega^{-4/3}\int_0^\infty |K(\xi/\omega)|
		\int_0^\infty |F(\tau,\xi,r)|\,|g(r)|\,dr\,d\xi
		\le C\|g\|_{E_{7/6,7/6}}\omega^{-3/2}.
	\end{equation*}
	
	After multiplication by \(|\eta(\omega)|\), both bounds
	are integrable over \([t_1,t_2]\times I\). Fubini and the definitions of \(S(\tau)g\) and \(L\) therefore give
	\begin{equation*}
		\int_I
		\bigl(S(t_2)g-S(t_1)g\bigr)(\omega)\eta(\omega)\,d\omega
		=
		\int_{t_1}^{t_2}\int_I
		LS(\tau)g(\omega)\eta(\omega)\,d\omega\,d\tau.
	\end{equation*}
 Hence the difference
	between the two sides of
	\eqref{eq:fixed-local-integral-positive}, which belongs
	to \(L^\infty(I)\), has zero integral against every
	\(\eta\in C_c^\infty(I)\).
	Density implies that this difference is zero
	almost everywhere. This proves
	\eqref{eq:fixed-local-integral-positive}.
	
     Then the positive-time
	continuity of \(LS(\cdot)g|_I\) ensures that \(S(\cdot)g|_I\) is strongly continuously differentiable
	on \((0,\infty)\), and its strong derivative in
	\(L^\infty(I)\) is \(LS(t)g|_I\).
	
	For the identification with the differentiated kernel and
	the bound in \(E_{7/6,7/6}\), we next justify differentiating
	the kernel representation. Write
	\begin{equation*}
		S(t)g(\omega)
		=
		e^{-\mathfrak a t\omega^{-1/3}}g(\omega)
		+
		\int_0^\infty F_{\mathrm{reg}}(t,\omega,r)g(r)\,dr.
	\end{equation*}
	Here \(F_{\mathrm{reg}}\) is the function part after removing
	the Dirac mass.
	Fix \(t>0\), \(\omega>0\), and \(0<\varepsilon<1/4\) sufficiently small.
We will justify differentiation under the integral by finding
an integrable bound for the time difference quotients,
independent of \(h\) for \(0<|h|<t/2\).
Throughout the following estimates,
\(\tau\in[t/2,3t/2]\), and constants may depend on
\(t,\omega\).

For
\begin{equation*}
	0<r<\min\{\omega/2,(t/2)^3,1\},
\end{equation*}
we have \(\omega/r>2\) and \(\tau/r^{1/3}>1\).
There is no Dirac contribution in this region.
The pointwise estimate used here is
\eqref{eq:EMVfar-derivative}:
\begin{equation*}
	|u\partial_uF(u,z,1)|
	\lesssim z^{-11/6}u^{-3v_0+9/2+\varepsilon},
	\qquad u\ge1,\quad z>2.
\end{equation*}
Applying the chain rule to the identity
\begin{equation*}
	F_{\mathrm{reg}}(\tau,\omega,r)
	=
	r^{-1}F(\tau/r^{1/3},\omega/r,1)
\end{equation*}
therefore gives
\begin{equation*}
	|\partial_\tau F_{\mathrm{reg}}(\tau,\omega,r)|
	\lesssim
	\frac{1}{\tau r}
	\left(\frac{\omega}{r}\right)^{-11/6}
	\left(\frac{\tau}{r^{1/3}}\right)^{-3v_0+9/2+\varepsilon}
	=
	\omega^{-11/6}\tau^{-3v_0+7/2+\varepsilon}
	r^{v_0-2/3-\varepsilon/3}.
\end{equation*}
Then we obtain
\begin{equation*}
	|\partial_\tau F_{\mathrm{reg}}(\tau,\omega,r)g(r)|
	\lesssim_{t,\omega}\|g\|_{E_{7/6,7/6}}
	r^{v_0-11/6-\varepsilon/3}.
\end{equation*}

For
\begin{equation*}
	r>\max\{2\omega,(3t/2)^3,1\},
\end{equation*}
we have \(0<\omega/r<1/2\) and
\(0<\tau/r^{1/3}<1\).
The Dirac term and \(R_1\) vanish in this region, so
\begin{equation*}
	F_{\mathrm{reg}}(\tau,\omega,r)
	=
	r^{-1}
	\left[
	\sigma(\tau/r^{1/3})(\omega/r)^{-7/6}
	+
	R_2(\tau/r^{1/3},\omega/r)
	\right].
\end{equation*}
The time-derivative estimates for these two terms are
(see \eqref{eq:EMVsigma-derivative})
\begin{equation*}
	|u\sigma'(u)|\lesssim u^4,
	\qquad 0<u\le1,
\end{equation*}
and, taking  \(b=13/12\) in \eqref{eq:EMVR2-derivative},
\begin{equation*}
	|u\partial_uR_2(u,z)|
	\lesssim z^{-13/12}u^{3/4-\varepsilon},
	\qquad 0<u\le1,\quad 0<z<1/2.
\end{equation*}
The chain rule consequently yields
\begin{equation*}
	|\partial_\tau F_{\mathrm{reg}}(\tau,\omega,r)|
	\lesssim
	\frac{1}{\tau r}
	\left[
	\left(\frac{\omega}{r}\right)^{-7/6}
	\left(\frac{\tau}{r^{1/3}}\right)^4
	+
	\left(\frac{\omega}{r}\right)^{-13/12}
	\left(\frac{\tau}{r^{1/3}}\right)^{3/4-\varepsilon}
	\right].
\end{equation*}
Multiplication by the bound for \(|g(r)|\) gives
\begin{equation*}
	|\partial_\tau F_{\mathrm{reg}}(\tau,\omega,r)g(r)|
	\lesssim_{t,\omega}\|g\|_{E_{7/6,7/6}}
	\left(r^{-7/3}+r^{-4/3+\varepsilon/3}\right).
\end{equation*}

For the remaining case, both \(r\)
and \(\tau/r^{1/3}\) have positive lower bounds and finite
upper bounds. The time-derivative estimates for the
\(\sigma\), \(R_2\), and far-region terms
(see \eqref{eq:EMVsigma-derivative},
\eqref{eq:EMVR2-derivative}, and
\eqref{eq:EMVfar-derivative}) are therefore
uniformly bounded there. For \(R_1\)
(see \eqref{eq:EMVR1-derivative}), the only additional
factor is
\begin{equation*}
	\left|\frac{\omega}{r}-1\right|^{-5/6}
	=
	r^{5/6}|\omega-r|^{-5/6}.
\end{equation*}
Thus, for \(r\ne\omega\), we can obtain
\begin{equation}\label{eq:appendix-diagonal-majorant}
	|\partial_\tau F_{\mathrm{reg}}(\tau,\omega,r)g(r)|
	\lesssim_{t,\omega}\|g\|_{E_{7/6,7/6}}
	\left(1+|\omega-r|^{-5/6}\right).
\end{equation}
These bounds are integrable in \(r\): at zero,
\begin{equation*}
	v_0-\frac{11}{6}-\frac{\varepsilon}{3}>-1;
\end{equation*}
at infinity,
\begin{equation*}
	-\frac73<-1,
	\qquad
	-\frac43+\frac{\varepsilon}{3}<-1;
\end{equation*}
and the singularity in (\ref{eq:appendix-diagonal-majorant}) is integrable because \(5/6<1\).

For fixed \(r\ne\omega\), \(0<|h|<t/2\), the mean value theorem gives
\begin{equation*}
	\left|
	\frac{F_{\mathrm{reg}}(t+h,\omega,r)
		-F_{\mathrm{reg}}(t,\omega,r)}{h}\,g(r)
	\right|
	\le
	\sup_{\tau\in[t/2,3t/2]}
	|\partial_\tau F_{\mathrm{reg}}(\tau,\omega,r)g(r)|.
\end{equation*}
The preceding estimates bound the right-hand side by
an integrable function of \(r\), independent of \(h\).
Dominated convergence therefore gives
\begin{equation*}
	\lim_{h\to0}
	\int_0^\infty
	\frac{F_{\mathrm{reg}}(t+h,\omega,r)
		-F_{\mathrm{reg}}(t,\omega,r)}{h}\,g(r)\,dr
	=
	\int_0^\infty
	\partial_tF_{\mathrm{reg}}(t,\omega,r)g(r)\,dr.
\end{equation*}
Differentiating the explicit Dirac contribution separately,
we obtain
\begin{equation*}
	\frac{d}{dt}S(t)g(\omega)
	=
	-\mathfrak a\omega^{-1/3}
	e^{-\mathfrak a t\omega^{-1/3}}g(\omega)
	+
	\int_0^\infty
	\partial_tF_{\mathrm{reg}}(t,\omega,r)g(r)\,dr=\int_0^\infty\partial_tF(t,\omega,r)g(r)\,dr.
\end{equation*}
The derivative in this calculation is the ordinary scalar
derivative with respect to $t$.

Since \(S(\cdot)g|_I\) is strongly continuously differentiable
on \((0,\infty)\), and its strong derivative in
\(L^\infty(I)\) is \(LS(t)g|_I\), we can identify that
\begin{equation*}
	LS(t)g|_I
	=
	\left(
	\int_0^\infty
	\partial_tF(t,\cdot,r)g(r)\,dr
	\right)\Big|_I
	\quad\text{in }L^\infty(I).
\end{equation*}

Since $I$ is arbitrary, \eqref{eq:fixed-absolute-E} then implies \eqref{eq:fixed-LS-E-bound}.

Next we show that the same conclusions hold for $g\in Y$. For $g\in Y$, we prove
\begin{equation*}
	S(\cdot)g|_I\in C^1((0,\infty);L^\infty(I))
\end{equation*}
and \eqref{eq:fixed-local-equation} for every compact interval
$I\subset(0,\infty)$.
We follow the proof for $E_{7/6,7/6}$ data: first prove
continuity of $LS(t)g|_I$, then establish
\eqref{eq:fixed-local-integral-positive}, and finally
identify the strong derivative with the differentiated kernel.

For continuity of $LS(t)g|_I$, fix $I=[\alpha,\beta]$
and $0<b<B<\infty$.
By \eqref{eq:kernel-YZ} and
$\mathcal Z\subset E_{7/6,7/6}$,
\begin{equation}\label{asdd}
	\sup_{b\le t\le B}
	\|S(t)g\|_{E_{7/6,7/6}}
	\lesssim_{b,B}\|g\|_Y.
\end{equation}
Also, $Y\subset E_{3/2,3/2}$, so
\eqref{eq:fixed-critical-continuity} gives
\begin{equation}\label{asddd}
	\|S(t_2)g-S(t_1)g\|_{L^\infty(J)}
	\lesssim_{b,B,J}|t_2-t_1|\|g\|_Y
\end{equation}
for every compact interval $J\subset(0,\infty)$ and
$t_1,t_2\in[b,B]$.

To pass from this local continuity to continuity after
applying $L$, choose $0<\rho<\alpha/2$ and $R>2\beta$.
The kernel formula for $L$ gives
\begin{equation*}
	LS(t_2)g(\omega)-LS(t_1)g(\omega)
	=
	-\mathfrak a\omega^{-1/3}
	[S(t_2)g(\omega)-S(t_1)g(\omega)]
	+
	\omega^{-4/3}\int_0^\infty
	K(r/\omega)[S(t_2)g(r)-S(t_1)g(r)]\,dr.
\end{equation*}
For $r<\rho$, use
$|K(r/\omega)|\lesssim(r/\omega)^{1/2}$
from \eqref{eq:Kbounds} and the uniform
$E_{7/6,7/6}$ bound (\ref{asdd}) to obtain
\begin{equation*}
	\sup_{\omega\in I}
	\omega^{-4/3}\int_0^\rho
	|K(r/\omega)|\,|S(t_2)g(r)-S(t_1)g(r)|\,dr
	\lesssim_{b,B,I}
	\|g\|_Y\rho^{1/3}.
\end{equation*}
For $r>R$, use
$|K(r/\omega)|\lesssim(r/\omega)^{-7/6}$
from \eqref{eq:Kbounds} and (\ref{asdd}) to obtain
\begin{equation*}
	\sup_{\omega\in I}
	\omega^{-4/3}\int_R^\infty
	|K(r/\omega)|\,|S(t_2)g(r)-S(t_1)g(r)|\,dr
	\lesssim_{b,B,I}
	\|g\|_YR^{-4/3}.
\end{equation*}
On $[\rho,R]$,
\begin{equation*}
	\sup_{\omega\in I}
	\omega^{-4/3}\int_\rho^R|K(r/\omega)|\,dr
	\le
	\alpha^{-1/3}
	\int_{\rho/\beta}^{R/\alpha}|K(z)|\,dz<\infty,
\end{equation*}
because the singularity at $z=1$ has exponent $5/6<1$.
The multiplication term is bounded by the same local
difference norm, since $I\subset[\rho,R]$.
Consequently,
\begin{equation*}
	\|LS(t_2)g-LS(t_1)g\|_{L^\infty(I)}
	\le
	C_{I,\rho,R}
	\|S(t_2)g-S(t_1)g\|_{L^\infty([\rho,R])}
	+
	C_{b,B,I}\|g\|_Y
	(\rho^{1/3}+R^{-4/3}).
\end{equation*}
For fixed $\rho,R$, let $t_2\to t_1$ using
\eqref{asddd}.
Then let $\rho\downarrow0$ and $R\to\infty$.
This proves
\begin{equation*}
	LS(\cdot)g|_I\in C((0,\infty);L^\infty(I)).
\end{equation*}

For \eqref{eq:fixed-local-integral-positive}, fix
$0<t_1<t_2$.
We first record the absolute kernel bound needed
to integrate \eqref{eq:fundamental-tested-identity}
against $g(r)$:
\begin{equation*}
	\int_0^\infty |F(\tau,\omega,r)|\,|g(r)|\,dr
	\lesssim_{t_1,t_2}\|g\|_Y\omega^{-7/6},
	\qquad t_1\le\tau\le t_2,\quad \omega>0.
\end{equation*}
Indeed, the absolute estimates
\eqref{eq:diagonal}, \eqref{eq:R2shape}, and
\eqref{eq:farshape}, as used in proving
\eqref{eq:kernel-YZ}, bound the contributions from
\eqref{eq:fixed-part-dirac}--\eqref{eq:fixed-part-far}
in $X$, uniformly for $\tau\in[t_1,t_2]$.
Since $X\subset E_{7/6,7/6}$, each contribution is
bounded by $\|g\|_Y\omega^{-7/6}$.
For \eqref{eq:fixed-part-sigma}, use \eqref{eq:cfinal}:
\begin{equation*}
	\omega^{-7/6}
	\int_{\omega/2}^\infty
	r^{1/6}|\sigma(\tau/r^{1/3})|\,|g(r)|\,dr
	\le
	\omega^{-7/6}
	\int_0^\infty
	r^{1/6}|\sigma(\tau/r^{1/3})|\,|g(r)|\,dr
	\lesssim_{t_1,t_2}\|g\|_Y\omega^{-7/6}.
\end{equation*}
Adding the five absolute contributions proves the
required bound.

Consequently, we can obtain
\begin{equation*}
	\mathfrak a\omega^{-1/3}
	\int_0^\infty|F(\tau,\omega,r)|\,|g(r)|\,dr
	\lesssim_{t_1,t_2}\|g\|_Y\omega^{-3/2}.
\end{equation*}
For the integral part of $L$, the same bound gives
\begin{equation*}
	\omega^{-4/3}\int_0^\infty |K(\xi/\omega)|
	\int_0^\infty |F(\tau,\xi,r)|\,|g(r)|\,dr\,d\xi
	\lesssim_{t_1,t_2}
	\|g\|_Y\omega^{-3/2}
	\int_0^\infty |K(z)|z^{-7/6}\,dz.
\end{equation*}
The last integral is finite by \eqref{eq:Kbounds}.
After multiplication by $|\eta(\omega)|$, both bounds
are integrable over $[t_1,t_2]\times I$.

Thus, the above estimates justify the application of Fubini, which yields
\begin{equation*}
	\int_I[S(t_2)g-S(t_1)g](\omega)\eta(\omega)\,d\omega
	=
	\int_{t_1}^{t_2}\int_I
	LS(\tau)g(\omega)\eta(\omega)\,d\omega\,d\tau.
\end{equation*}
Since $LS(\tau)g|_I$ is continuous, its time integral
is a Bochner integral in $L^\infty(I)$.
The difference between the two sides of
\eqref{eq:fixed-local-integral-positive} therefore
belongs to $L^\infty(I)$ and has zero integral against
every $\eta\in C_c^\infty(I)$.
Density proves
\eqref{eq:fixed-local-integral-positive} for $Y$ data.
 Hence for $g\in Y$, one can still ensure
\begin{equation*}
	\frac{d}{dt}(S(t)g|_I)=(LS(t)g)|_I,
	\qquad t>0,
\end{equation*}
with a strong derivative in $L^\infty(I)$.

For identification with the differentiated kernel,
fix $t,\omega>0$ and use
\begin{equation*}
	|g(r)|\le\|g\|_Y r^{-3/2},\qquad r>0.
\end{equation*}
For
$0<r<\min\{\omega/2,(t/2)^3,1\}$,
the scaling formula and
\eqref{eq:EMVfar-derivative} give, uniformly for
$\tau\in[t/2,3t/2]$,
\begin{equation*}
	|\partial_\tau F_{\mathrm{reg}}(\tau,\omega,r)g(r)|
	\lesssim_{t,\omega}
	\|g\|_Y r^{v_0-13/6-\epsilon/3}.
\end{equation*}
For
$r>\max\{2\omega,(3t/2)^3,1\}$,
the estimates
\eqref{eq:EMVsigma-derivative} and
\eqref{eq:EMVR2-derivative}, with exponent
$b=13/12$, give
\begin{equation*}
	|\partial_\tau F_{\mathrm{reg}}(\tau,\omega,r)g(r)|
	\lesssim_{t,\omega}
	\|g\|_Y
	\left(r^{-8/3}+r^{-5/3+\epsilon/3}\right).
\end{equation*}
On the remaining compact interval,
\eqref{eq:EMVR1-derivative} and the other derivative
bounds give
\begin{equation*}
	|\partial_\tau F_{\mathrm{reg}}(\tau,\omega,r)g(r)|
	\lesssim_{t,\omega}
	\|g\|_Y\left(1+|r-\omega|^{-5/6}\right).
\end{equation*}
These bounds are integrable, since
\begin{equation*}
	v_0-\frac{13}{6}-\frac{\epsilon}{3}>-1,
	\qquad
	-\frac83<-1,
	\qquad
	-\frac53+\frac{\epsilon}{3}<-1,
	\qquad
	\frac56<1.
\end{equation*}
Then the dominated convergence,
together with differentiation of the Dirac part, gives
\begin{equation*}
	\frac{d}{dt}S(t)g(\omega)
	=
	-\mathfrak a\omega^{-1/3}
	e^{-\mathfrak a t\omega^{-1/3}}g(\omega)
	+
	\int_0^\infty
	\partial_tF_{\mathrm{reg}}(t,\omega,r)g(r)\,dr
\end{equation*}
for almost every $\omega$, for each fixed $t>0$. Thus, we have the identification:
\begin{equation*}
	(LS(t)g)|_I
	=
	\left(
	\int_0^\infty\partial_tF(t,\cdot,r)g(r)\,dr
	\right)\Big|_I
	\quad\text{in }L^\infty(I).
\end{equation*}
This proves \eqref{eq:fixed-local-equation}
for $g\in Y$ and every $t>0$.

\end{proof}

The positive-time equation gives a limit at zero. We next identify
that limit and include zero as an endpoint of the integral identity.
\begin{lemma}[Initial trace and the integral equation from time zero]
	\label{lem:appendix-initial-trace}
	Let $g\in E_{7/6,7/6}\cup Y$, and set $S(0)g=g$.
	For every compact interval $I\subset(0,\infty)$,
	\begin{equation*}
		S(\cdot)g|_I\in C([0,\infty);L^\infty(I)).
	\end{equation*}
	For $g\in E_{7/6,7/6}$,
	\begin{equation}\label{eq:fixed-local-trace-rate}
		\|S(t)g-g\|_{L^\infty(I)}\lesssim_I t\|g\|_{E_{7/6,7/6}},\qquad0<t\le1.
	\end{equation}
	For $g\in Y$,
	\begin{equation}\label{errrr}
		\|S(t)g-g\|_{L^\infty(I)}\lesssim_I t\|g\|_Y,\qquad0<t\le1.
	\end{equation}
	For either class of data,
	\begin{equation}\label{eq:fixed-local-integral-zero}
		S(t)g|_I=g|_I+\int_0^t(LS(r)g)|_I\,dr,\qquad t\ge0,
	\end{equation}
	where the integral is a Bochner integral in $L^\infty(I)$.
\end{lemma}
\begin{proof}
Fix $I=[\alpha,\beta]\subset(0,\infty)$ and first take
$g\in E_{7/6,7/6}$.
By \eqref{eq:fixed-L-E} and
\eqref{eq:fixed-basic-E-bound},
\begin{equation*}
	\|LS(\tau)g\|_{L^\infty(I)}
	\lesssim_I
	\|S(\tau)g\|_{E_{7/6,7/6}}
	\lesssim
	\|g\|_{E_{7/6,7/6}},
	\qquad \tau>0.
\end{equation*}
Integrating this bound in
\eqref{eq:fixed-local-integral-positive} gives
\begin{equation}\label{eq:appendix-local-lipschitz}
	\|S(t_2)g-S(t_1)g\|_{L^\infty(I)}
	\lesssim_I(t_2-t_1)\|g\|_{E_{7/6,7/6}},
	\qquad 0<t_1<t_2.
\end{equation}
	Thus \(S(t)g|_I\) has a limit in \(L^\infty(I)\) as
	\(t\downarrow0\). It remains to identify this limit
	as \(g|_I\).
	
	Recall that the fundamental solution $F$ has initial data $\delta_r$, which means
	\begin{equation*}
		\lim_{t\downarrow0}
		\int_I\eta(\omega)\,F(t,\omega,r)d\omega
		=\eta(r),
		\qquad \forall\eta\in C_c^\infty(I).
	\end{equation*}
	
The estimates in Theorem~\ref{thm:EMV} give an integrable
bound in $r$, independent of $0<t\le1$, for
$|g(r)|\int_I|\eta(\omega)|\,|F(t,\omega,r)|d\omega$,
so Fubini and dominated convergence yield
	\begin{equation*}
		\lim_{t\downarrow0}
		\int_I \eta(\omega)S(t)g\,d\omega
		=
		\int_I \eta(\omega)g(\omega)\,d\omega.
	\end{equation*}
	The existing \(L^\infty(I)\) limit consequently equals
	\(g|_I\) almost everywhere.
	Letting \(t_1\downarrow0\) in (\ref{eq:appendix-local-lipschitz}) gives
	\begin{equation*}
		\|S(t)g-g\|_{L^\infty(I)}
		\lesssim_I t\|g\|_{E_{7/6,7/6}},
		\qquad 0<t\le1.
	\end{equation*}
	Together with the positive-time continuity, this proves
	\(S(\cdot)g|_I\in C([0,\infty);L^\infty(I))\).
	
	Since the map \(\tau\mapsto LS(\tau)g|_I\) is continuous for
	\(\tau>0\) and uniformly bounded near zero,
	\eqref{eq:fixed-local-integral-positive} now gives
	\begin{equation*}
		S(t)g|_I
		=
		g|_I+\int_0^t LS(\tau)g|_I\,d\tau,
		\qquad t\ge0.
	\end{equation*}
	It is precisely the definition of integral solution.
	
	For $g\in Y$, \eqref{eq:fixedYY} and
	\eqref{eq:fixed-L-Y} give
	\begin{equation*}
		\sup_{0<t\le1}\|LS(t)g\|_{L^\infty(I)}\lesssim_I\|g\|_Y.
	\end{equation*}
	Then the positive-time identity \eqref{eq:fixed-local-integral-positive} therefore makes $S(t)g|_I$ Cauchy at zero. As before, we can apply Theorem \ref{thm:EMV} to justify the application of Fubini, which identifies this limit with $g|_I$. This concludes (\ref{errrr}) and
	\eqref{eq:fixed-local-integral-zero} for $Y$ data.
	
\end{proof}

The local equation and trace allow us to prove the semigroup identity
before using it to commute $S(t)$ and $L$.
\begin{lemma}[Semigroup identity and commutation with $L$]
\label{lem:appendix-semigroup-commutation}
For $g\in E_{7/6,7/6}\cup Y$,
\begin{equation}\label{eq:fixed-semigroup}
 S(t)S(s)g=S(t+s)g,\qquad t,s\ge0,
\end{equation}
where $S(0)=\Id$. For $g\in Y$ and $s>0$, the composition
is defined using $S(s)g\in\mathcal Z$.
For $g\in E_{7/6,7/6}$ and $t>0$, the kernel integral
\begin{equation*}
 S(t)Lg(\omega)=\int_0^\infty F(t,\omega,r)Lg(r)\,dr
\end{equation*}
is absolutely convergent for almost every $\omega>0$. For every compact interval
$I\subset(0,\infty)$ and $g\in E_{7/6,7/6}$,
\begin{equation}\label{eq:fixed-commutation-E}
 (LS(t)g)|_I=(S(t)Lg)|_I\quad\text{in }L^\infty(I).
\end{equation}
Finally,
\begin{equation}\label{eq:fixed-stationary}
 S(t)f_M=f_M,\qquad t\ge0.
\end{equation}
\end{lemma}
\begin{proof}

	For \eqref{eq:fixed-semigroup}, we first consider
	\(g\in E_{7/6,11/6}\).
	This additional decay at infinity allows us to use the
	uniqueness statement in Theorem~\ref{thm:EMV}.
	
	Indeed, \eqref{eq:fixed-absolute-auxiliary} gives
	\begin{equation*}
		|(S(t)g)(\omega)|
		\lesssim
		\|g\|_{E_{7/6,11/6}}
		\min(\omega^{-7/6},\omega^{-11/6}),
		\qquad t>0.
	\end{equation*}
	Thus \(S(t)\) is uniformly bounded on \(E_{7/6,11/6}\).
	Moreover,
	\begin{equation*}
		E_{7/6,11/6}\subset E_{7/6,7/6},
	\end{equation*}
	so Lemma~\ref{lem:appendix-local-equation} applies both to
	\(g\) and to \(S(s)g\).
	
	Fix \(s>0\). By the local differential equation in
	Lemma~\ref{lem:appendix-local-equation},
	\begin{equation*}
		\frac{d}{dt}\bigl(S(t+s)g-S(t)S(s)g\bigr)\Big|_I
		=
		L\bigl(S(t+s)g-S(t)S(s)g\bigr)\Big|_I,
		\qquad t>0,
	\end{equation*}
	with strong derivative in \(L^\infty(I)\), for every
	compact interval \(I\subset(0,\infty)\).
	The positive-time continuity in Lemma~\ref{lem:appendix-local-equation}
	and the local initial trace in Lemma~\ref{lem:appendix-initial-trace} also give
	\begin{equation*}
		\lim_{t\downarrow0}
		\|S(t+s)g-S(t)S(s)g\|_{L^\infty(I)}
		=0.
	\end{equation*}
	Finally, the bound on \(E_{7/6,11/6}\) gives
	\begin{equation*}
		|(S(t+s)g-S(t)S(s)g)(\omega)|
		\lesssim
		\|g\|_{E_{7/6,11/6}}
		\min(\omega^{-7/6},\omega^{-11/6}).
	\end{equation*}
	
	We claim that the uniqueness statement in
	Theorem~\ref{thm:EMV} forces this difference to vanish.
	To apply that statement, consider 
	\begin{equation*}
		F(t,\omega,1)
		+
		\bigl(S(t+s)g-S(t)S(s)g\bigr)(\omega).
	\end{equation*}
	By linearity and the preceding local equation, this measure
	satisfies the same equation as \(F(t,\omega,1)\).
	Its initial measure is \(\delta_1\), because the second term tends to zero locally in \(L^\infty\).
	
	It also belongs to the uniqueness class of
	Theorem~\ref{thm:EMV}. Its Dirac coefficient is unchanged.
	The added function is bounded by \(C\omega^{-7/6}\)
	near zero, by \(C\omega^{-11/6}\) at infinity, and is
	locally bounded. In particular, on
	\(0<|\omega-1|\le1/2\), its boundedness implies a bound
	by \(C|\omega-1|^{-5/6}\).
	Thus the required bounds for the function part are
	preserved. Uniqueness gives
	\begin{equation*}
		F(t,\omega,1)
		+
		\bigl(S(t+s)g-S(t)S(s)g\bigr)(\omega)
		=
		F(t,\omega,1).
	\end{equation*}
	This proves
	\eqref{eq:fixed-semigroup} for
	\(g\in E_{7/6,11/6}\) and \(t,s>0\).
	
	We now extend the equality to \(g\in E_{7/6,7/6}\).
	Define the cutoff
	\begin{equation*}
		g_n(\omega)=\mathbf 1_{[1/n,n]}(\omega)g(\omega),
		\qquad n\ge2.
	\end{equation*}
	Then \(g_n\in E_{7/6,11/6}\),
	\(g_n\to g\) almost everywhere, and
	\begin{equation*}
		|g_n(\omega)|
		\le \|g\|_{E_{7/6,7/6}}\omega^{-7/6}.
	\end{equation*}
	For fixed \(s>0\), dominated convergence, justified by
	\eqref{eq:fixed-absolute-E}, gives
	\begin{equation*}
		(S(s)g_n)(r)\longrightarrow(S(s)g)(r)
		\quad\text{for almost every }r>0.
	\end{equation*}
	The same estimate gives the uniform bound
	\begin{equation*}
		|(S(s)g_n)(r)|
		\lesssim \|g\|_{E_{7/6,7/6}}r^{-7/6}.
	\end{equation*}
	Applying dominated convergence once more, now to \(S(t)\),
	yields
	\begin{equation*}
		(S(t)S(s)g_n)(\omega)
		\longrightarrow
		(S(t)S(s)g)(\omega)
		\quad\text{for almost every }\omega>0.
	\end{equation*}
 Similarly, we also have 
	\begin{equation*}
		(S(t+s)g_n)(\omega)
		\longrightarrow
		(S(t+s)g)(\omega)
		\quad\text{for almost every }\omega>0.
	\end{equation*}
	Passing to the limit in the equality for \(g_n\) proves
	\eqref{eq:fixed-semigroup} for \(t,s>0\).
	Both sides belong to \(E_{7/6,7/6}\) by
	\eqref{eq:fixed-absolute-E}, so this is an equality in
	that space. The cases \(t=0\) or \(s=0\) follow from
	\(S(0)=\Id\).

For $g\in Y$, use the same truncations $g_n=\mathbf1_{[1/n,n]}g$.
They belong to $E_{7/6,7/6}$, satisfy $\|g_n\|_Y\le\|g\|_Y$,
and converge almost everywhere to $g$. For fixed $s>0$,
\eqref{eq:fixed-absolute-local}, with $|g_n|\le\|g\|_Y r^{-3/2}$,
gives $S(s)g_n\to S(s)g$ almost everywhere. The bound \eqref{eq:kernel-YZ} gives
\begin{equation*}
 |S(s)g_n(r)|\lesssim k(s)\|g\|_Y r^{-7/6}.
\end{equation*}
Thus applying dominated convergence again, \eqref{eq:fixed-semigroup} holds for $g\in Y$.

We next prove commutation for $E_{7/6,7/6}$ data. For \eqref{eq:fixed-commutation-E},
	\eqref{eq:fixed-L-E} first gives
	\begin{equation*}
		|Lg(r)|
		\lesssim \|g\|_{E_{7/6,7/6}}r^{-3/2}.
	\end{equation*}
	Thus \eqref{eq:fixed-absolute-local} makes the kernel
	integral \(S(t)Lg\) absolutely convergent and locally
	bounded for every \(t>0\). We next claim that
	\begin{equation}\label{eq:fixed-L-initial-continuity}
		\lim_{t\downarrow0}
		\|LS(t)g-Lg\|_{L^\infty(I)}
		=0
	\end{equation}
	for every compact interval \(I=[\alpha,\beta]\).
	Choose \(0<\rho<\alpha/2\) and \(R>2\beta\).
	By \eqref{eq:Kbounds},
	\begin{equation*}
		\sup_{\omega\in I}
		\omega^{-4/3}
		\int_0^\rho |K(r/\omega)|r^{-7/6}\,dr
		\lesssim_I\int_0^\rho r^{-2/3}\,dr
		\sim\rho^{1/3},
	\end{equation*}
	and
	\begin{equation*}
		\sup_{\omega\in I}
		\omega^{-4/3}
		\int_R^\infty |K(r/\omega)|r^{-7/6}\,dr
		\lesssim_I\int_R^\infty r^{-7/3}\,dr
		\sim R^{-4/3}.
	\end{equation*}
	On \([\rho,R]\), the integral of
	\(|K(r/\omega)|\) is bounded uniformly for \(\omega\in I\);
	the possible singularity has exponent \(5/6<1\).
	Using the formula for \(L\) and the uniform bound
	\eqref{eq:fixed-absolute-E}, we therefore obtain
	\begin{equation*}
		\|LS(v)g-Lg\|_{L^\infty(I)}
		\le
		C_{I,\rho,R}
		\|S(v)g-g\|_{L^\infty([\rho,R])}
		+
		C_I\|g\|_{E_{7/6,7/6}}
		\bigl(\rho^{1/3}+R^{-4/3}\bigr).
	\end{equation*}
	For fixed \(\rho,R\), the first term tends to zero by
	\eqref{eq:fixed-local-trace-rate}.
	Taking the upper limit as \(v\downarrow0\), then letting
	\(\rho\downarrow0\) and \(R\to\infty\), proves
	\eqref{eq:fixed-L-initial-continuity}.
	
To prove \eqref{eq:fixed-commutation-E}, we will apply
$S(t)$ to the difference quotient $(S(h)g-g)/h$.
We first show that this quotient converges locally to
$Lg$ and has a uniform bound that permits passage
to the limit inside the kernel integral of $S(t)$.

For the local convergence,
\eqref{eq:fixed-local-integral-zero} and
\eqref{eq:fixed-L-initial-continuity} give
\begin{equation*}
	\left\|
	\frac{S(h)g-g}{h}-Lg
	\right\|_{L^\infty(I)}
	\le
	\sup_{0<t<h}
	\|LS(t)g-Lg\|_{L^\infty(I)}
	\longrightarrow0
	\qquad(h\downarrow0).
\end{equation*}
To control the same quotients over all frequencies,
recall
\begin{equation*}
	\|LS(t)g\|_{E_{3/2,3/2}}
	\lesssim \|g\|_{E_{7/6,7/6}},
	\qquad t>0.
\end{equation*}
Applying this bound in
\eqref{eq:fixed-local-integral-zero} gives,
\begin{equation}\label{ere}
	\left\|\frac{S(h)g-g}{h}\right\|_{E_{3/2,3/2}}
	\lesssim \|g\|_{E_{7/6,7/6}},
	\qquad h>0.
\end{equation}

We now use these two conclusions to pass to the limit
after applying $S(t)$.
Fix $t>0$ and take $h=1/n$.
The local $L^\infty$ convergence just proved implies,
\begin{equation*}
	n\left((S\left(\frac1n\right)g)(r)-g(r)\right)
	\longrightarrow Lg(r)\quad\text{for almost every }r>0.
\end{equation*}
The uniform $r^{-3/2}$ bound (\ref{ere}) and
\eqref{eq:fixed-absolute-local} allow dominated convergence
in the function part of the kernel of $S(t)$.
The Dirac contribution converges by the same
almost-everywhere convergence. Consequently,
\begin{equation*}
	\left(S(t)\bigl[n(S(1/n)g-g)\bigr]\right)(\omega)
	\longrightarrow (S(t)Lg)(\omega)
	\quad\text{for almost every }\omega>0.
\end{equation*}

To identify this limit with $LS(t)g$, use
\eqref{eq:fixed-semigroup} to rewrite the expression
on the left as
\begin{equation}\label{ax}
	n\bigl(S(t+1/n)g-S(t)g\bigr)(\omega).
\end{equation}
By Lemma~\ref{lem:appendix-local-equation}, (\ref{ax}) converges strongly in $L^\infty(I)$
to $LS(t)g|_I$.
Comparison of the limits proves
\eqref{eq:fixed-commutation-E}.

We finally apply \eqref{eq:fixed-commutation-E}
to prove \eqref{eq:fixed-stationary}.
Lemma~\ref{lem:stationary} gives
$Lf_M=3\mathcal C(f_M)=0$, while
$f_M\in E_{7/6,7/6}$. Hence
\begin{equation*}
	LS(t)f_M=S(t)Lf_M=0,
	\qquad t>0.
\end{equation*}
The local integral identity
\eqref{eq:fixed-local-integral-zero}, applied to $f_M$,
therefore yields
\begin{equation*}
	S(t)f_M=f_M,
	\qquad t\ge0.
\end{equation*}

\end{proof}

We now strengthen the local derivative identity to a strong derivative
in $\mathcal Z$. This is the justification of $S'(t)$ in Section~\ref{sec:fixed}.
\begin{lemma}[Strong differentiation in $\mathcal Z$]
\label{lem:appendix-Z-differentiability}
For every $g\in\mathcal Z\cup Y$,
\begin{equation}\label{sdsd}
	S(\cdot)g\in C^1((0,\infty);\mathcal Z).
\end{equation}
For every $t>0$, define
\begin{equation*}
 S'(t)g=\lim_{h\to0}\frac{S(t+h)g-S(t)g}{h},
\end{equation*}
where the limit is taken in $\mathcal Z$. Then
\begin{equation*}
 S'(t)g=LS(t)g=\int_0^\infty\partial_tF(t,\cdot,r)g(r)\,dr.
\end{equation*}
In particular, both
expressions on the right belong to $\mathcal Z$.
For $0<t_1<t_2$,
\begin{equation}\label{eq:appendix-Z-time-integral}
 S(t_2)g-S(t_1)g=\int_{t_1}^{t_2}S'(r)g\,dr,
\end{equation}
where the integral is a Bochner integral.
If $g\in\mathcal Z$, then
\begin{equation*}
 S'(t)g=S(t)Lg\quad\text{in }\mathcal Z,\qquad t>0.
\end{equation*}
\end{lemma}
\begin{proof}
To prove that $LS(t)g\in\mathcal Z$, use
\eqref{eq:fixed-local-equation} to write
\begin{equation*}
	LS(t)g
	=
	\int_0^\infty\partial_tF(t,\cdot,r)g(r)\,dr.
\end{equation*}
For $g\in\mathcal Z$, \eqref{eq:kernel-ZZ} therefore gives
\begin{equation*}
	\|LS(t)g\|_{\mathcal Z}
	\lesssim t^{-1}\|g\|_{\mathcal Z}.
\end{equation*}
For $g\in Y$, \eqref{eq:kernel-derivative-YZ} gives
\begin{equation*}
	\|LS(t)g\|_{\mathcal Z}
	\lesssim t^{-1}k(t)\|g\|_Y.
\end{equation*}
For continuity of this $\mathcal Z$-valued function, fix $t_0>0$.
The kernel estimates put $S(t_0/2)g$ and $LS(t_0/2)g$ in
$\mathcal Z\subset E_{7/6,7/6}$. Apply the commutation identity
to $S(t_0/2)g$, and then the semigroup identity. For $t>t_0/2$,
\begin{equation*}
 LS(t)g=S(t-t_0/2)LS(t_0/2)g.
\end{equation*}
The equality first holds locally and then in $\mathcal Z$, since
both sides belong to that space. Moreover, the positive-time continuity
\eqref{eq:fixed-Z-difference} implies continuity of $LS(t)g$ in $\mathcal Z$. Consequently, the local identity
\eqref{eq:fixed-local-integral-positive} ensures (\ref{sdsd}) and \eqref{eq:appendix-Z-time-integral}. The last assertion is just the commutation identity of Lemma~\ref{lem:appendix-semigroup-commutation}.

\end{proof}

The uniqueness argument in Section~\ref{sec:linear} differentiates
$S(\theta(r))h(r)$ when $h$ is only an integral solution.
The following lemma supplies precisely that product rule.
\begin{lemma}[Time substitution and the product rule]\label{lem:appendix-product-rule}
Let $a<b$, and let $\theta:[a,b]\to(0,\infty)$ be Lipschitz
continuous, with $0<\inf_{[a,b]}\theta\le\sup_{[a,b]}\theta<\infty$.
Suppose that $h:[a,b]\to E_{7/6,7/6}$ satisfies
\begin{equation*}
 \sup_{a\le r\le b}\|h(r)\|_{E_{7/6,7/6}}<\infty.
\end{equation*}
Suppose that $v(r)\in E_{3/2,3/2}$ for almost every $r$ and 
\begin{equation*}
	\int_a^b\|v(r)\|_{E_{3/2,3/2}}\,dr<\infty.
\end{equation*}
Assume that for any compact interval $J\subseteq (0,\infty)$,
\begin{equation*}
 h(r)|_J=h(a)|_J+\int_a^r v(q)|_J\,dq,\qquad a\le r\le b,
\end{equation*}
where the integral is a Bochner integral in $L^\infty(J)$.
Then for any compact interval $I$, $r\mapsto S(\theta(r))h(r)|_I$ is absolutely continuous.
For almost every $r\in(a,b)$, its strong derivative in $L^\infty(I)$ satisfies
\begin{equation}\label{eq:appendix-product-rule}
 \frac{d}{dr}(S(\theta(r))h(r)|_I)
 =\theta'(r)(LS(\theta(r))h(r))|_I+(S(\theta(r))v(r))|_I.
\end{equation}
Moreover, $LS(\theta(r))h(r)$ can be replaced by $S(\theta(r))Lh(r)$.
\end{lemma}
\begin{proof}
The local integral identity and the bound on $v$ imply
\begin{equation}\label{eq:appendix-critical-increments}
 \|h(r_2)-h(r_1)\|_{E_{3/2,3/2}}
 \le\int_{r_1}^{r_2}\|v(q)\|_{E_{3/2,3/2}}\,dq,
 \qquad a\le r_1<r_2\le b.
\end{equation}

For almost every $r\in(a,b)$, $\theta'(r)$ exists and
the local integral identity for $h$ gives
\begin{equation*}
	\frac{h(r+q)-h(r)}q\longrightarrow v(r)
	\quad\text{in every local }L^\infty\text{ space as }q\to0.
\end{equation*}
At such $r$, \eqref{eq:appendix-critical-increments}
also bounds the above difference quotients in $E_{3/2,3/2}$
for sufficiently small $|q|$. Write
\begin{equation*}
	S(\theta(r+q))\frac{h(r+q)-h(r)}q-S(\theta(r))v(r)
\end{equation*}
as
\begin{equation*}
	S(\theta(r+q))
	\left[\frac{h(r+q)-h(r)}q-v(r)\right]
	+[S(\theta(r+q))-S(\theta(r))]v(r)
\end{equation*}
The first term tends to zero in $L^\infty(I)$ by
\eqref{eq:fixed-critical-local-convergence}, since
the bracket converges locally to zero and is bounded
in $E_{3/2,3/2}$.
The second tends to zero by
\eqref{eq:fixed-critical-continuity} and
$\theta(r+q)\to\theta(r)$.
Hence
\begin{equation*}
	S(\theta(r+q))\frac{h(r+q)-h(r)}q\Big|_I
	\longrightarrow S(\theta(r))v(r)|_I
	\quad\text{in }L^\infty(I).
\end{equation*}
For the fixed function $h(r)\in E_{7/6,7/6}$,
Lemma~\ref{lem:appendix-local-equation} and the chain rule give
\begin{equation*}
	\frac{S(\theta(r+q))-S(\theta(r))}{q}h(r)\Big|_I
	\longrightarrow\theta'(r)LS(\theta(r))h(r)|_I.
\end{equation*}
Adding these two limits proves
\eqref{eq:appendix-product-rule} for almost every $r$.

For absolute continuity, the positive-time bounds in
Lemmas~\ref{lem:appendix-absolute-kernel} and
\ref{lem:appendix-positive-continuity} yield
\begin{equation*}
		\|S(\theta(r_2))h(r_2)-S(\theta(r_1))h(r_1)\|_{L^\infty(I)}
		\lesssim
		\sup_{[a,b]}\|h(r)\|_{E_{7/6,7/6}}
		|\theta(r_2)-\theta(r_1)|
		+
		\int_{r_1}^{r_2}\|v(q)\|_{E_{3/2,3/2}}\,dq.
\end{equation*}
Since $\theta$ is Lipschitz and
$\|v(\cdot)\|_{E_{3/2,3/2}}$ is integrable,
this proves absolute continuity.

\end{proof}

We next justify the series construction in Section~\ref{sec:linear}.
Use $Q(t,s)$ for the bracket in \eqref{eq:Q}, and retain $V_1,V_n$
from \eqref{eq:defect} and \eqref{eq:R}. Products denote composition.
Set
\begin{equation}\label{eq:appendix-UN-definition}
 U_N(t,s)g=Q(t,s)g+\sum_{n=1}^N\int_s^t Q(t,r)V_n(r,s)g\,dr,
 \qquad t>s,
\end{equation}
and $U_N(s,s)g=g$, with the sum empty for $N=0$.
Throughout the remaining lemmas, we assume $|a_j(t)-1|\le\eta_0$ for
$t\ge0$ and $j\in\{0,\infty\}$, where $\eta_0>0$ is small enough. 

\begin{lemma}[The integral identity for finite sums]\label{lem:appendix-finite-sums}
Let $g\in\mathcal Z\cup Y$. The integrals defining $V_n(t,s)g$
are Bochner integrals in $Y$, and the integrals defining $U_N(t,s)g$
are Bochner integrals in $\mathcal Z$.
For every $N\ge0$, every compact interval $I\subset(0,\infty)$,
and every $t\ge s$,
\begin{equation}\label{eq:appendix-finite-integral}
 U_N(t,s)g|_I-g|_I+\int_s^tV_{N+1}(r,s)g|_I\,dr
 =\int_s^t\big([a_0(r)^2(1-\chi)+a_\infty(r)^2\chi]LU_N(r,s)g\big)|_I\,dr.
\end{equation}
Both integrals in this identity are Bochner integrals in $L^\infty(I)$.
\end{lemma}
\begin{proof}
For existence of the integrals in \eqref{eq:R}
and \eqref{eq:U}, the required norm integrability
was established in Steps 2 and 3 of
Theorem~\ref{thm:linear}.

For the integral identity \eqref{eq:appendix-finite-integral},
fix $g\in\mathcal Z\cup Y$ and $t>s$.
Apply the positive-time chain rule to each term of $Q(q,s)g$
for $s+\varepsilon\le q\le t$, where
$0<\varepsilon<t-s$.
Lemma~\ref{lem:appendix-Z-differentiability}
justifies this with values in $\mathcal Z$.
The algebra in Step~2 of Theorem~\ref{thm:linear} gives
\begin{equation}\label{asdf}
	\frac{d}{dq}(Q(q,s)g|_I)
	=
	\big(
	[a_0(q)^2(1-\chi)+a_\infty(q)^2\chi]LQ(q,s)g
	-V_1(q,s)g
	\big)|_I
\end{equation}
for almost every $q\in(s,t)$.
For $g\in\mathcal Z$, $LQ(q,s)g$ is locally bounded
uniformly near $s$ by \eqref{eq:kernel-ZZ},
\eqref{eq:fixed-L-E}, and Lemma~\ref{lem:commutator}.
For $g\in Y$, use \eqref{eq:fixedYY}
and \eqref{eq:fixed-L-Y} when $q-s\le1/2$,
and \eqref{eq:kernel-YZ} away from $s$.
Since $V_1(q,s)g$ has an integrable bound
by \eqref{eq:Ebounds}, then the right-hand side of (\ref{asdf}) is locally integrable down to $s$.
The local trace of $Q$ follows from
Lemma~\ref{lem:appendix-initial-trace}.
Integration over $[s+\varepsilon,t]$ and passage
$\varepsilon\downarrow0$ therefore give
\begin{equation}\label{eq:appendix-Q-integral}
	Q(t,s)g|_I-g|_I
	=
	\int_s^t
	\big(
	[a_0(q)^2(1-\chi)+a_\infty(q)^2\chi]LQ(q,s)g
	-V_1(q,s)g
	\big)|_I\,dq.
\end{equation}

For the integral with index $n$ in \eqref{eq:U},
we will integrate \eqref{eq:appendix-Q-integral},
with initial time $r$ and datum $V_n(r,s)g\in Y$,
over $s<r<t$.
We first justify exchanging the resulting time integrals.

For $q-r\le1/2$, use \eqref{eq:Q},
\eqref{eq:kernel-YY-absolute}, and \eqref{eq:fixed-L-Y};
for $q-r>1/2$, use \eqref{eq:Q},
\eqref{eq:kernel-YZ}, and \eqref{eq:fixed-L-E}.
These give
\begin{equation*}
	\|LQ(q,r)V_n(r,s)g\|_{L^\infty(I)}
	\lesssim_{I,t,s}\|V_n(r,s)g\|_Y.
\end{equation*}
Also, \eqref{eq:Ebounds} gives
\begin{equation*}
	\|V_1(q,r)V_n(r,s)g\|_Y
	\lesssim \eta_0 k(q-r)\|V_n(r,s)g\|_Y.
\end{equation*}
By \eqref{eq:Rbounds} and $k\in L^1$, these bounds
give norm integrability over $s<r<q<t$, which justifies Bochner Fubini.

We now integrate \eqref{eq:appendix-Q-integral}
with initial time $r$ and datum $V_n(r,s)g$ over
$s<r<t$, and use Fubini to write the time integrals
in the order $\int_s^t dq\int_s^q dr$.
By \eqref{eq:fixed-L-E}, $L$ commutes with
the inner $\mathcal Z$-valued Bochner integral,
with equality in $L^\infty(I)$.
Finally, \eqref{eq:R} identifies
\begin{equation*}
	\int_s^q V_1(q,r)V_n(r,s)g\,dr
	=V_{n+1}(q,s)g.
\end{equation*}
Substituting these identities gives
\begin{equation*}
	\left(\int_s^t Q(t,r)V_n(r,s)g\,dr\right)\Big|_I
	-
	\int_s^t
	\left(
	[a_0(q)^2(1-\chi)+a_\infty(q)^2\chi]
	L\int_s^q Q(q,r)V_n(r,s)g\,dr
	\right)\Big|_I\,dq
\end{equation*}
\begin{equation*}
	=
	\int_s^t[(V_n(q,s)-V_{n+1}(q,s))g]|_I\,dq.
\end{equation*}
Summing over $1\le n\le N$ and adding
\eqref{eq:appendix-Q-integral}, the cancellation
proves \eqref{eq:appendix-finite-integral}.

\end{proof}

To pass from this finite identity to $U$, we need
convergence also after integration in time.
\begin{lemma}[Convergence of the series]
\label{lem:appendix-series-convergence}
For $g\in\mathcal Z\cup Y$ and every $t>s$, the sequence
$U_N(t,s)g$ converges in $\mathcal Z$.
Denote its limit by $U(t,s)g$, and set $U(s,s)g=g$.
For every $s<t<\infty$,
\begin{equation*}
	U_N(\cdot,s)g\longrightarrow U(\cdot,s)g
	\quad\text{in }L^1((s,t);\mathcal Z),
\end{equation*}
and
\begin{equation*}
	V_{N+1}(\cdot,s)g\longrightarrow0
	\quad\text{in }L^1((s,t);Y).
\end{equation*}

\end{lemma}
\begin{proof}
The geometric estimates \eqref{eq:Rbounds} and
\eqref{eq:eta} give
\begin{equation*}
	\sum_{n>N}\|V_n(r,s)\|_{\mathcal Z\to Y}
	\lesssim 4^{-N},
	\qquad
	\sum_{n>N}\|V_n(r,s)\|_{Y\to Y}
	\lesssim 4^{-N}k(r-s).
\end{equation*}
Multiplication by $Q(t,r):Y\to\mathcal Z$ and
integration over $s<r<t$ show that the series
defining $U(t,s)g$ converges in $\mathcal Z$, with
\begin{equation*}
	\|(U-U_N)(t,s)g\|_{\mathcal Z}
	\lesssim 4^{-N}\|g\|_{\mathcal Z},
	\qquad g\in\mathcal Z,
\end{equation*}
and
\begin{equation*}
	\|(U-U_N)(t,s)g\|_{\mathcal Z}
	\lesssim 4^{-N}k(t-s)\|g\|_Y,
	\qquad g\in Y,
\end{equation*}
using \eqref{eq:convolution} in the second estimate.

For convergence in $L^1((s,t);\mathcal Z)$,
apply these bounds at time $\tau$ and integrate over
$s<\tau<t$:
\begin{equation*}
	\int_s^t\|(U-U_N)(\tau,s)g\|_{\mathcal Z}\,d\tau
	\lesssim 4^{-N}(t-s)\|g\|_{\mathcal Z},
	\qquad g\in\mathcal Z,
\end{equation*}
\begin{equation*}
	\int_s^t\|(U-U_N)(\tau,s)g\|_{\mathcal Z}\,d\tau
	\lesssim 4^{-N}\|k\|_{L^1}\|g\|_Y,
	\qquad g\in Y.
\end{equation*}
Both bounds tend to zero.
The bounds on $V_{N+1}$ in \eqref{eq:Rbounds},
integrated over the same interval, prove convergence
to zero in $L^1((s,t);Y)$.

\end{proof}

We can now pass to the limit and verify the equation and trace for $U$.
\begin{lemma}[The local integral equation for $U(t,s)$]
\label{lem:appendix-U-equation}
Let $g\in\mathcal Z\cup Y$. For every compact interval $I\subset(0,\infty)$,
\begin{equation*}U(\cdot,s)g|_I\in C([s,\infty);L^\infty(I)),\end{equation*}
and
\begin{equation}\label{eq:appendix-U-integral}
 U(t,s)g|_I=g|_I+
 \int_s^t\big([a_0(r)^2(1-\chi)+a_\infty(r)^2\chi]LU(r,s)g\big)|_I\,dr,
 \quad t\ge s.
\end{equation}
The integral is a Bochner integral in $L^\infty(I)$.
For almost every $t>s$, the strong derivative in $L^\infty(I)$
exists and satisfies
\begin{equation*}
 \frac{d}{dt}(U(t,s)g|_I)
 =\big([a_0(t)^2(1-\chi)+a_\infty(t)^2\chi]LU(t,s)g\big)|_I.
\end{equation*}
Moreover,
\begin{equation}\label{qwqw}
	\|U(t,s)\|_{Y\to Y}\lesssim 1,\qquad0<t-s\le1/2.
\end{equation}
\end{lemma}
\begin{proof}
By Lemma~\ref{lem:appendix-series-convergence},
\eqref{eq:fixed-L-E}, the boundedness of $a_j$,
and the embeddings $\mathcal Z,Y\subset L^\infty(I)$,
letting $N\to\infty$ in
\eqref{eq:appendix-finite-integral} gives
\eqref{eq:appendix-U-integral} in $L^\infty(I)$.

For the trace and continuity, the identity \eqref{eq:appendix-U-integral} gives,
for $s\le t_1<t_2$,
\begin{equation*}
	\|(U(t_2,s)g-U(t_1,s)g)\|_{L^\infty(I)}
	\lesssim_I
	\int_{t_1}^{t_2}\|U(q,s)g\|_{\mathcal Z}\,dq.
\end{equation*}
By Lemma~\ref{lem:appendix-series-convergence},
the integrand belongs to $L^1((s,t))$ for every $t>s$.
Consequently,
\begin{equation*}
	U(\cdot,s)g|_I\in C([s,\infty);L^\infty(I)),
	\qquad
	\lim_{\tau\downarrow s}
	\|U(\tau,s)g-g\|_{L^\infty(I)}=0.
\end{equation*}
Bochner differentiation of
\eqref{eq:appendix-U-integral} yields
\begin{equation*}
	\frac{d}{dt}(U(t,s)g|_I)
	=
	\big(
	[a_0(t)^2(1-\chi)+a_\infty(t)^2\chi]
	LU(t,s)g
	\big)|_I
\end{equation*}
for almost every $t>s$.

For \eqref{qwqw}, let $g\in Y$ and
$0<t-s\le1/2$.
For $s\le r<t$, the coefficient bounds imply
\begin{equation*}
	0<\int_r^t a_j(q)^2\,dq\le1,
	\qquad j\in\{0,\infty\}.
\end{equation*}
Hence \eqref{eq:Q} and \eqref{eq:fixedYY} give
\begin{equation}\label{qw}
	\|Q(t,r)\|_{Y\to Y}\lesssim1.
\end{equation}
Applying \eqref{eq:Rsum} and (\ref{qw}) to
\eqref{eq:appendix-UN-definition} gives
\begin{equation*}
	\|U_N(t,s)g\|_Y\lesssim
	\|g\|_Y+
	\int_s^t\sum_{n=1}^N\|V_n(r,s)g\|_Y\,dr
	\lesssim
	\left(1+\eta_0\int_s^t k(r-s)\,dr\right)\|g\|_Y
	\lesssim\|g\|_Y.
\end{equation*}
From Lemma~\ref{lem:appendix-series-convergence}, one can pass
$N\to\infty$ in the above estimate, which yields
\begin{equation*}
	\|U(t,s)g\|_Y\lesssim\|g\|_Y,
	\qquad 0<t-s\le1/2.
\end{equation*}
This proves \eqref{qwqw}.

\end{proof}

We finally justify the inhomogeneous formula, first as a
$\mathcal Z$-valued integral and then as a local integral solution.
\begin{lemma}[The Duhamel integral in $\mathcal Z$]\label{lem:appendix-Duhamel-bound}
Let $0<T<\infty$, $z\in\mathcal Z$, and
$g\in L^\infty((0,T);Y)$.
For $0<t\le T$, define
\begin{equation*}
 u(t)=U(t,0)z+\int_0^tU(t,s)g(s)\,ds,
\end{equation*}
and set $u(0)=z$. The integral is a Bochner integral in $\mathcal Z$.
Moreover, we have
\begin{equation*}
 \sup_{0<t\le T}\|u(t)\|_{\mathcal Z}
 \lesssim \|z\|_{\mathcal Z}+\|k\|_{L^1(0,\infty)}
 \|g\|_{L^\infty((0,T);Y)}.
\end{equation*}
\end{lemma}
\begin{proof}
By \eqref{eq:L2} and $k\in L^1(0,\infty)$,
\begin{equation*}
	\int_0^t\|U(t,s)g(s)\|_{\mathcal Z}\,ds
	\lesssim
	\|g\|_{L^\infty((0,t);Y)}
	\int_0^t k(t-s)\,ds
	\lesssim
	\|k\|_{L^1(0,\infty)}
	\|g\|_{L^\infty((0,t);Y)}.
\end{equation*}
Together with \eqref{eq:L1} for the homogeneous term,
this gives the stated bound.

\end{proof}

\begin{lemma}[The local integral equation for the Duhamel formula]
\label{lem:appendix-Duhamel-equation}
Under the assumptions of Lemma~\ref{lem:appendix-Duhamel-bound},
for every compact interval $I\subset(0,\infty)$,
\begin{equation*}
	u|_I\in C([0,T];L^\infty(I)),
\end{equation*}
and
\begin{equation}\label{eq:appendix-Duhamel-integral}
 u(t)|_I=z|_I+\int_0^t
 \big([a_0(r)^2(1-\chi)+a_\infty(r)^2\chi]Lu(r)+g(r)\big)|_I\,dr,
 \quad0\le t\le T.
\end{equation}
The integral is a Bochner integral in $L^\infty(I)$.
For almost every $t\in(0,T)$,
\begin{equation*}
 \frac{d}{dt}(u(t)|_I)
 =\big([a_0(t)^2(1-\chi)+a_\infty(t)^2\chi]Lu(t)+g(t)\big)|_I,
\end{equation*}
where the derivative is a strong derivative in $L^\infty(I)$.
Furthermore,
\begin{equation*}
 \|LU(r,s)g(s)\|_{L^\infty(I)}\lesssim_{I,T}\|g(s)\|_Y
\end{equation*}
for almost every $s\in(0,T)$ and every $s<r\le T$. Consequently,
\begin{equation*}
 \int_0^t\int_s^t\|LU(r,s)g(s)\|_{L^\infty(I)}\,dr\,ds
 \lesssim_{I,T}\int_0^t(t-s)\|g(s)\|_Y\,ds<\infty,
 \qquad0<t\le T.
\end{equation*}
\end{lemma}
\begin{proof}
For the last two bounds, when $r-s\le1/2$, combine the $Y\to Y$
bound in Lemma~\ref{lem:appendix-U-equation} with
$L:Y\to L^\infty(I)$. When $r-s>1/2$, use \eqref{eq:L2},
$k(r-s)\le k(1/2)$, and $L:\mathcal Z\to L^\infty(I)$.
This proves the pointwise norm bound and its stated double integral.

For \eqref{eq:appendix-Duhamel-integral},
\eqref{eq:appendix-U-integral} gives
\begin{equation*}
	[U(t,s)g(s)-g(s)]|_I
	=
	\int_s^t
	\big(
	[a_0(r)^2(1-\chi)+a_\infty(r)^2\chi]
	LU(r,s)g(s)
	\big)|_I\,dr.
\end{equation*}
Integrating over $0<s<t$ yields
\begin{equation*}
	\left(\int_0^tU(t,s)g(s)\,ds\right)\Big|_I
	-
	\int_0^t g(s)|_I\,ds
	=
	\int_0^t\int_s^t
	\big(
	[a_0(r)^2(1-\chi)+a_\infty(r)^2\chi]
	LU(r,s)g(s)
	\big)|_I\,dr\,ds.
\end{equation*}
By Fubini, justified by the last bound, the right-hand side equals
\begin{equation*}
	\int_0^t\int_0^r
	\big(
	[a_0(r)^2(1-\chi)+a_\infty(r)^2\chi]
	LU(r,s)g(s)
	\big)|_I\,ds\,dr.
\end{equation*}
By \eqref{eq:fixed-L-E} and
$\mathcal Z\subset E_{7/6,7/6}$,
\begin{equation*}
	\int_0^r(LU(r,s)g(s))|_I\,ds
	=
	\left(
	L\int_0^rU(r,s)g(s)\,ds
	\right)\Big|_I.
\end{equation*}
Consequently,
\begin{equation*}
	\left(\int_0^tU(t,s)g(s)\,ds\right)\Big|_I
	=
	\int_0^t g(r)|_I\,dr
	+
	\int_0^t
	\left(
	[a_0(r)^2(1-\chi)+a_\infty(r)^2\chi]
	L\int_0^rU(r,s)g(s)\,ds
	\right)\Big|_I\,dr.
\end{equation*}
For the homogeneous term,
Lemma~\ref{lem:appendix-U-equation} gives
\begin{equation*}
	U(t,0)z|_I
	=
	z|_I+
	\int_0^t
	\big(
	[a_0(r)^2(1-\chi)+a_\infty(r)^2\chi]
	LU(r,0)z
	\big)|_I\,dr.
\end{equation*}
Adding these identities and using the Duhamel formula
for $u$ proves \eqref{eq:appendix-Duhamel-integral}.
Its Bochner integrand belongs to
$L^1((0,T);L^\infty(I))$, so
\begin{equation*}
	u(\cdot)|_I\in C([0,T];L^\infty(I)),
\end{equation*}
and
\begin{equation*}
	\frac{d}{dt}(u(t)|_I)
	=
	\big(
	[a_0(t)^2(1-\chi)+a_\infty(t)^2\chi]Lu(t)
	+g(t)
	\big)|_I
\end{equation*}
strongly in $L^\infty(I)$ for almost every $t\in(0,T)$.

\end{proof}

\clearpage
\section*{Acknowledgement}
The author thanks Prof. Alex Cohen for providing financial support to cover the costs of ChatGPT usage during the preparation of this work.

\section*{Conflict of interest statement}
The author does not have any possible conflict of interest.

\section*{Data availability statement}
The manuscript has no associated data.
\bibliographystyle{plain}
\bibliography{kz_stability_revised}
\end{document}